\documentclass[11pt]{amsart}
\usepackage [latin1]{inputenc}
\usepackage[colorlinks,linkcolor=blue,anchorcolor=blue,citecolor=green]{hyperref}
\usepackage{mathrsfs}
\usepackage{graphicx}
\usepackage[arrow,matrix]{xy}
\usepackage{amsmath,amssymb,amscd,bbm,amsthm,mathrsfs, bm}
\usepackage{color,xcolor}
\usepackage{graphicx}
\usepackage{manfnt}

\numberwithin{equation}{section}

\newtheorem{Theorem}{Theorem}[section]
\newtheorem{Lemma}{Lemma}[section]
\newtheorem{Corollary}{Corollary}[section]

\theoremstyle{definition}

\newtheorem{Remark}{Remark}[section]

\theoremstyle{remark}

\date{}
\begin{document}

\title{ Globally conformally K\"ahler Einstein metrics on certain holomorphic bundles}
\author{Zhiming Feng}
\address{School of Mathematics and Physics, Leshan Normal University, \\Leshan, Sichuan 614000, P.R. China}
\email{fengzm2008@163.com}


\date{}
\maketitle

\begin{abstract}
The subject of this paper is the explicit momentum construction of complete  Einstein metrics by ODE methods. Using the Calabi ansatz, further generalized by Hwang-Singer, we show that there are  non-trivial complete conformally K\"ahler Einstein metrics on certain Hermitian holomorphic vector bundles and their subbundles over complete K\"ahler-Einstein manifolds. In special cases, we give the explicit expressions of of these metrics. These examples show  that there is a  compact K\"ahler manifold $M$ and its subvariety $N$ whose codimension is greater than 1 such that there is a complete conformally K\"ahler Einstein metric on $M-N$.
\end{abstract}



\section{Introduction and main results}

Let $M$ be a complex $m$-dimensional K\"ahler manifold  endowed with a K\"ahler metric $g$ with respect to an integrable almost complex structure $J$, and $f$ be a positive smooth function $f$ on $M$. A Hermitian metric $\widetilde{g}=\frac{1}{f^2}g$ on $(M,J)$ is called a globally conformally K\"ahler metric.

If there is a constant $\mu$ such that the Ricci tensor $\mathrm{Ric}_{\widetilde{g}}$ of the metric $\widetilde{g}$ satisfies $\mathrm{Ric}_{\widetilde{g}}=\mu\;\widetilde{g}$, then the Hermitian metric $\widetilde{g}$ on $(M,J)$ is called a conformally K\"ahler Einstein metric; if $f$ also satisfies $\mathrm{d}f\wedge\mathrm{d}\triangle f=0$, then the metric $\widetilde{g}$ is called a strongly conformally K\"ahler Einstein metric; this is automatic if $m\geq 3$.

If $\mathrm{Ric}_{\widetilde{g}}(J\cdot,J\cdot)=\mathrm{Ric}_{\widetilde{g}}(\cdot,\cdot)$ and the scalar curvature $s_{\widetilde{g}}$ of $\widetilde{g}$ is constant, then the metric $\widetilde{g}$ on $(M,J)$ is called a conformally K\"ahler, Einstein-Maxwell metric (cKEM metric for short), see \cite{LeBrun-2015,LeBrun-2016}. Note that $\mathrm{Ric}_{\widetilde{g}}(J\cdot,J\cdot)=\mathrm{Ric}_{\widetilde{g}}(\cdot,\cdot)$ if and only if $\overline{\partial}(\mathrm{grad}^{1,0}f)=\overline{\partial}(g^{j\bar{q}}f_{\bar{q}}\partial_j)=0$. Obviously, a conformally K\"ahler Einstein metric is a cKEM metric. When $f$ is constant on $M$, a cKEM metric is a constant scalar curvature K\"ahler (cscK for short) metric, so a cKEM metric with nonconstant $f$  is referred to as non-trivial.

For the case of compact, so far, not many examples of non-trivial cKEM metrics are known.  The conformally K\"ahler Einstein metrics  constructed by Page  \cite{Page} on the one-point-blow-up of $\mathbb{CP}^2$, by Chen-LeBrun-Weber  \cite{Chen-LeBrun-Weber} on the two-point-blow-up of $\mathbb{CP}^2$, by Apostolov-Calderbank-Gauduchon  \cite{Apostolov-Calderbank-Gauduchon-2016,Apostolov-Calderbank-Gauduchon-2015} on 4-orbifolds and by B\'erard-Bergery \cite{Berard-Bergery} on $\mathbb{P}^1$-bundles over Fano K\"ahler-Einstein manifolds.  Non-Einstein cKEM examples are given by LeBrun \cite{LeBrun-2015,LeBrun-2016} on $\mathbb{CP}^1 \times \mathbb{CP}^1$ and the one-point-blow-up of $\mathbb{CP}^2$, by Koca-T{\o}nnesen-Friedman \cite{Koca-Tonnesen-Friedman-2016} on ruled surfaces of higher genus, and by Futaki-Ono \cite{Futaki-Ono-2018} on $\mathbb{CP}^1\times M$ where $M$ is a compact constant scalar curvature K\"ahler manifold of arbitrary dimension. For more research on cKEM metrics, please refer to Apostolov-Maschler \cite{Apostolov-Maschler-2019}, Apostolov-Maschler-Tonnesen-Friedman \cite{Apostolov-Maschler-Tonnesen-Friedman}, Futaki-Ono \cite{Futaki-Ono-2019} and Lahdili \cite{Lahdili-2019,Lahdili-2017,Lahdili-2019-2}.

In \cite{Derdzinski-Maschler-2003}, Derdzinski-Maschler gave the local classification of strongly conformally K\"ahler Einstein metrics. In \cite{Apostolov-Calderbank-Gauduchon-2006},  Apostolov-Calderbank-Gauduchon also gave the local classification of strongly conformally K\"ahler Einstein metrics by using the more general Hamiltonian 2-forms method. That is, if $\widetilde{g}=\frac{1}{f^2}g$ is a strongly conformally K\"ahler Einstein metric, then the metric $g$ can be expressed as
\begin{eqnarray*}
  g &=& zg_0+\frac{z^{m-1}}{F(z)}\mathrm{d}z^2+\frac{F(z)}{z^{m-1}}\theta^2, \\
  \omega &=&\mathrm{d}\mathrm{d}^c\int^z\frac{u^m}{F(u)}\mathrm{d}u=z\omega_0+\mathrm{d}z\wedge\theta,\;\mathrm{d}\theta=\omega_0, J\mathrm{d}z =\frac{F(z)}{z^{m-1}} \theta,\;J\theta=-\frac{z^{m-1}}{F(z)}\mathrm{d}z, \\
  \rho_0&=&c\,\omega_0,
\end{eqnarray*}
where $c$ is a constant, $f=qz+p$ and
\begin{equation}\label{1.0.1 }
  F''(z)-\left(\frac{2(m-1)q}{qz+p}+\frac{m}{z}\right)F'(z)+\frac{2m(m-1)q}{z(qz+p)}F(z)+2cz^{m-2}=0.
\end{equation}

In \cite{Derdzinski-Maschler-2006,Derdzinski-Maschler-2005}, Derdzinski-Maschler gave the classification compact conformally Einstein K\"ahler manifolds $M$,  namely $M$ are the total spaces of  holomorphic $\mathbb{CP}^1$ bundles over  compact K\"ahler-Einstein manifolds.

For the case of non-compact, in \cite{Feng-2020-2}, we study the existence of cKEM metrics on rotationally invariant domains in $\mathbb{C}^d$. For non-compact examples of conformally K\"ahler Einstein metrics, see \cite{Derdzinsk-1984, Apostolov-Calderbank-Gauduchon-2003}.

Let $N$ be the generic norm of an irreducible bounded symmetric domain $\Omega$, for $\lambda>0$ and $r\in\mathbb{N}_{+}=\mathbb{N}-\{0\}$, the Cartan-Hartogs domain defined by
\begin{equation*}\label{1.0.2}
  M=\left\{(z,w)\in \Omega\times\mathbb{C}^{r}: e^{\lambda\phi(z)}\|w\|^2<1\right\},\phi(z)=-\log \mathrm{N}(z,z),
\end{equation*}
its holomorphic invariant K\"ahler metrics can be expressed as
\begin{equation*}
   \omega=\sqrt{-1}\partial\overline{\partial}\left(\nu\phi(z)+F(t)\right),\,t=\lambda\phi(z)+\log\|w\|^2,\nu> 0.
\end{equation*}
So holomorphic invariant conformally K\"ahler metrics $\widetilde{g}$ associated with K\"ahler forms $\widetilde{\omega}$  on the Cartan-Hartogs domain $M$ can be expressed as
\begin{equation*}\label{1.0.3}
   \widetilde{\omega}=\frac{\sqrt{-1}}{f^2(t)}\partial\overline{\partial}\left(\nu\phi(z)+F(t)\right).
\end{equation*}

In this paper, using the Calabi ansatz and  the momentum construction \cite{Hwang-Singer-2002}, we  construct non-trivial complete conformally K\"ahler Einstein metrics $\widetilde{g}$ associated with K\"ahler forms $\widetilde{\omega}$ on certain holomorphic bundles, which are locally expressed as
\begin{equation*}
  M=\left\{(z,w)\in \Omega\times\mathbb{C}^{r}: e^{\lambda\phi(z)}\|w\|^2\in \mathcal{I}\right\} (\mathcal{I}\; \mathrm{are}\; \mathrm{intervals}\; \mathrm{on} \;\mathbb{R})
\end{equation*}
and
\begin{equation*}
   \widetilde{\omega}=\frac{\sqrt{-1}}{f^2(t)}\partial\overline{\partial}\left(\nu\phi(z)+F(t)\right),\,t=\lambda\phi(z)+\log\|w\|^2,\nu\geq 0.
\end{equation*}

The following Theorem \ref{Th:1.1} is the main result of this paper.

\begin{Theorem}\label{Th:1.1}{

 Let $(L_0, h_0,\Pi_0)$ be a Hermitian holomorphic line bundle over a complete K\"{a}hler manifold $(M_0,g_0)$ of complex dimension $d$ such that $2\pi c_1 (L_0,h_0)=-\lambda\omega_0$, where $\omega_0$ denotes the K\"{a}hler form associated with the K\"{a}hler  metric $g_0$,  $c_1(L_0,h_0)$ denotes the curvature of the Chern connection on the Hermitian holomorphic bundle $(L_0,h_0)$, and $\Pi_0:L_0\rightarrow M_0$ is the bundle projection.  The curvature $c_1(L_0,h_0)$ is given by
$$c_1(L_0,h_0) = -\frac{\sqrt{-1}}{2\pi}\partial\bar{\partial}\log h_0(\sigma(z), \sigma(z))$$
for a trivializing holomorphic section $\sigma : U \rightarrow L_0\backslash\{0\}$.

Let $E$ be the direct sum of $r$ copies of $L_0$, i.e. $E = L_0^{\oplus r}$ with an associated hermitian metric  denoted by $h$. The projective bundle $\mathbb{P}(E \oplus 1)$ can be viewed as a compactification of $E$,  where $1$ stands for the trivial line bundle $M_0\times \mathbb{C}$. The projection of the vector bundle $E$ to $M_0$  denoted by $\Pi : E \rightarrow M_0$.

Let $M$ be one of the following sets:
$$E=L_0^{\oplus r}, \;E^{*}=\{u\in E:0<h(u,u)\}, E-\{u\in E: h(u,u)\leq 1\}, $$
$$\mathbb{B}^r_{L_0}=\{u\in E:h(u,u)<1\},\; \mathbb{B}^{r*}_{L_0}= \{u\in E:0<h(u,u)<1\},$$
 $$\mathbb{P}(E \oplus 1)-M_0\;\text{and}\;\mathbb{P}(E \oplus 1)-\{u\in E: h(u,u)\leq 1\}.$$

Let $g_F$ be a K\"{a}hler  metric  on $M$ associated with the K\"{a}hler form
\begin{equation*}
  \omega_F(u)=\nu\;(\Pi^{*}\omega_0)(u)+\sqrt{-1}\partial\overline{\partial}F(\log h(u,u))\;\;(u\in E,\nu=0\;\text{or}\;1).
\end{equation*}
($\omega_F$ can be extended across $M-E$ if $M=\mathbb{P}(E \oplus 1)$, the extension of $\omega_F$ on $M$ is still expressed as $\omega_F$).

Set
\begin{equation*}
 t=\log h(u,u),\; x=F'(t),\;\varphi(x)=F''(t),\;n=d+r,
\end{equation*}
\begin{equation*}
  x_0=\inf\{F'(t)\},\;x_1=\sup\{F'(t)\}, \;f=aF'(t)+b=ax+b>0,\;a\lambda\neq 0.
\end{equation*}

Let $\widetilde{g}=\frac{1}{f^2}g_F$ be a conformally K\"ahler metric on $M$. The following conclusions hold.

$(\mathrm{I})$  For $\nu=1$ and $b=1$, $\widetilde{g}$ is  a complete Einstein  metric (namely $\mathrm{Ric}_{\widetilde{g}}=\mu \;\widetilde{g}$) $M$ if and only if
\begin{equation}\label{1.1}
 \left\{ \begin{array}{l}
  \mathrm{Ric}_{g_{0}}=(\lambda+\mu-2a)\;g_{0},\\\\
  \varphi(x)=\frac{(1+ax)^{2d+1}(a\lambda x+2a-\lambda)}{(1+\lambda x)^{d}}\int_0^x\frac{(1+\lambda u)^{d}\left(\mu(1+\lambda u)-(\lambda+\mu-2a)(1+au)^2\right)}{(1+au)^{2d+2}(a\lambda u+2a-\lambda)^2}\mathrm{d}u,x_0\leq x<x_1,\\\\
 \mu\leq 0,a\leq \lambda,a<0, r=1,
 \end{array}\right.
\end{equation}
here $x_0=0$, $x_1=-\frac{1}{a}$, $M=\mathbb{B}^r_{L_0}$ for $\mu<0$, and $M=L_0^{\oplus r}$ for $\mu=0$.

$(\mathrm{II})$

$(\mathrm{II-1})$  $\widetilde{g}$ is  a complete Einstein  metric (namely $\mathrm{Ric}_{\widetilde{g}}=\mu \;\widetilde{g}$)  on $\mathbb{B}^{r*}_{L_0}$ if and only if
\begin{equation}\label{1.2}
 \left\{ \begin{array}{l}
  \mathrm{Ric}_{g_{0}}=\mu\frac{\nu+\lambda x_0}{(b+ax_0)^2}\;g_{0},\\\\
  \varphi(x)=\mu\frac{(b+ax)^{2d+1}(a\lambda x+2a\nu-b\lambda)}{(\nu+\lambda x)^{d}}\int_{x_0}^x\frac{(\nu+\lambda u)^{d}\left((\nu+\lambda u)-\frac{\nu+\lambda x_0}{(b+ax_0)^2}(b+au)^2\right)}{(b+au)^{2d+2}(a\lambda u+2a\nu-b\lambda)^2}\mathrm{d}u,\;x_0< x<x_1,\\\\
 \mu<0,a<0,r=1,
 \end{array}\right.
\end{equation}
where
$$\left\{\begin{array}{l}
    -\infty\leq x_0<x_1=-\frac{b}{a}\; \mathrm{for}\; \lambda<0\;\mathrm{and}\;-\frac{b}{a}\leq -\frac{\nu}{\lambda}, \\\\
    -\frac{\nu}{\lambda}\leq x_0<x_1=-\frac{b}{a}\; \mathrm{for}\; \lambda>0\;\mathrm{and}\; -\frac{\nu}{\lambda}<-\frac{b}{a}.
  \end{array}\right.
$$

$(\mathrm{II-2})$  $\widetilde{g}$ is a complete Einstein  metric (namely $\mathrm{Ric}_{\widetilde{g}}=\mu \;\widetilde{g}$)  on  $M=E-\{u\in E:h(u,u)\leq 1\}$ if and only if
\begin{equation}\label{1.3}
 \left\{ \begin{array}{l}
  \mathrm{Ric}_{g_{0}}=\mu\frac{\nu+\lambda x_1}{(b+ax_1)^2}\;g_{0},\\\\
  \varphi(x)=\mu\frac{(b+ax)^{2d+1}(a\lambda x+2a\nu-b\lambda)}{(\nu+\lambda x)^{d}}\int_{x_1}^x\frac{(\nu+\lambda u)^{d}\left((\nu+\lambda u)-\frac{\nu+\lambda x_1}{(b+ax_1)^2}(b+au)^2\right)}{(b+au)^{2d+2}(a\lambda u+2a\nu-b\lambda)^2}\mathrm{d}u,\;x_0< x<x_1,\\\\
 \mu<0,a>0,r=1,
 \end{array}\right.
\end{equation}
where
$$\left\{\begin{array}{l}
   -\frac{b}{a}=x_0<  x_1\leq +\infty\; \mathrm{for}\; \lambda>0\;\mathrm{and}\;-\frac{b}{a}\geq -\frac{\nu}{\lambda}, \\\\
   -\frac{b}{a}=  x_0<x_1\leq-\frac{\nu}{\lambda}\; \mathrm{for}\; \lambda<0\;\mathrm{and}\; -\frac{\nu}{\lambda}>-\frac{b}{a}.
  \end{array}\right.
$$

$(\mathrm{III})$ Let $r=1$.

$(\mathrm{III-1})$ $\widetilde{g}$ is a complete Einstein  metric (namely $\mathrm{Ric}_{\widetilde{g}}=\mu \;\widetilde{g}$)  on  $M=\mathbb{P}(E \oplus 1)-M_0$ if and only if
\begin{equation}\label{1.4}
 \left\{ \begin{array}{l}
  \mathrm{Ric}_{g_{0}}=\gamma\;g_{0},\\\\
  \gamma=\frac{a\lambda x_1+2a\nu-b\lambda}{ax_1+b},\\\\
 \varphi(x)=-\gamma\frac{(a\lambda x+2a\nu-b\lambda)(ax+b)^{2d+1}}{(\nu+\lambda x)^d}\int_{x_1}^x\frac{(\nu+\lambda u)^d}{(au+b)^{2d}(a\lambda u+2a\nu-b\lambda)^2}\mathrm{d}u,\;x_0<x<x_1,\\\\
  a>0, \mu=0,
 \end{array}\right.
\end{equation}
where
$$\left\{\begin{array}{l}
   -\frac{b}{a}=x_0<  x_1< +\infty\; \mathrm{for}\; \lambda>0\;\mathrm{and}\;-\frac{b}{a}\geq -\frac{\nu}{\lambda}, \\\\
   -\frac{b}{a}=  x_0<x_1\leq-\frac{\nu}{\lambda}\; \mathrm{for}\; \lambda<0\;\mathrm{and}\; -\frac{\nu}{\lambda}>-\frac{b}{a}.
  \end{array}\right.
$$

$(\mathrm{III-2})$ $\widetilde{g}$ is a complete Einstein  metric (namely $\mathrm{Ric}_{\widetilde{g}}=\mu \;\widetilde{g}$)  on  $M=\mathbb{P}(E \oplus 1)-\{u\in E:h(u,u)\leq 1\}$ if and only if
\begin{equation}\label{1.5}
 \left\{ \begin{array}{l}
  \mathrm{Ric}_{g_{0}}=\gamma\;g_{0},\\\\
  \gamma=\mu\frac{\nu+\lambda x_1}{(b+ax_1)^2}+\frac{a\lambda x_1+2a\nu-b\lambda}{ax_1+b},\\\\
 \varphi(x)=\frac{(a\lambda x+2a\nu-b\lambda)(ax+b)^{2d+1}}{(\nu+\lambda x)^d}\int_{x_1}^x\frac{(\nu+\lambda u)^d(\mu(\nu+\lambda u)-\gamma(b+au)^2)}{(au+b)^{2d}(a\lambda u+2a\nu-b\lambda)^2}\mathrm{d}u,\;x_0<x<x_1,\\\\
 a>0, \mu<0,
 \end{array}\right.
\end{equation}
where
$$\left\{\begin{array}{l}
   -\frac{b}{a}=x_0<  x_1< +\infty\; \mathrm{for}\; \lambda>0\;\mathrm{and}\;-\frac{b}{a}\geq -\frac{\nu}{\lambda}, \\\\
   -\frac{b}{a}=  x_0<x_1\leq-\frac{\nu}{\lambda}\; \mathrm{for}\; \lambda<0\;\mathrm{and}\; -\frac{\nu}{\lambda}>-\frac{b}{a}.
  \end{array}\right.
$$

$(\mathrm{IV})$ Let $r>1$.

$(\mathrm{IV-1})$ $\widetilde{g}$ is a complete Einstein  metric (namely $\mathrm{Ric}_{\widetilde{g}}=\mu \;\widetilde{g}$)  on  $M=\mathbb{P}(E \oplus 1)-M_0$ if and only if
\begin{equation}\label{V-1.1.1}
 \left\{ \begin{array}{l}
  \mathrm{Ric}_{g_0}=r\lambda\;g_0,\\\\
  \varphi(x)= -r\frac{(ax+b)^{2n-1}(ax-b)}{x^{n-1}}\int_{x_1}^x \frac{ u^{n-1}}{(au+b)^{2n-2}(au-b)^2}\mathrm{d}u,\;-\frac{b}{a}=x_0<x<x_1=-\frac{(r+1)b}{(r-1)a},\\\\
  a>0, b<0,\lambda>0,\mu=\nu=0.
 \end{array}\right.
\end{equation}

$(\mathrm{IV-2})$ $\widetilde{g}$ is a complete Einstein  metric (namely $\mathrm{Ric}_{\widetilde{g}}=\mu \;\widetilde{g}$)  on  $M=\mathbb{P}(E \oplus 1)-\{u\in E:h(u,u)\leq 1\}$ if and only if
\begin{equation}\label{V-2.1.1}
 \left\{ \begin{array}{l}
  \mathrm{Ric}_{g_0}=r\lambda\;g_0,\\\\
   \mu=\frac{(ax_1+b)(r(ax_1+b)-(ax_1-b))}{x_1}<0,\\\\
 \varphi(x)= \frac{(ax+b)^{2n-1}(ax-b)}{x^{n-1}}\int_{x_1}^x \frac{ u^{n-1}\left(\mu u-r(b+au)^2\right)}{(b+au)^{2n}(au-b)^2}\mathrm{d}u,\;-\frac{b}{a}=x_0<x<x_1<-\frac{(r+1)b}{(r-1)a},\\\\
 a>0, b<0,\lambda>0,\nu=0.
 \end{array}\right.
\end{equation}

}\end{Theorem}

According to Theorem \ref{Th:1.1}, we have Corollary \ref{Co:1.1},  Corollary \ref{Co:1.2},  Corollary \ref{Co:1.3}, Corollary \ref{Co:1.4} and Corollary \ref{Co:1.5} for specific parameters $a, b,\lambda,\mu$ and $\nu$.

\begin{Corollary}\label{Co:1.1}{
Under the situation of Theorem \ref{Th:1.1},  let $a=\lambda=-1$ and $b=\nu=1$, then $\widetilde{g}$ is  a complete Einstein  metric (namely $\mathrm{Ric}_{\widetilde{g}}=\mu \;\widetilde{g}$) on $M$ if and only if
\begin{equation}\label{1.6}
 \left\{ \begin{array}{l}
  \mathrm{Ric}_{g_{0}}=(\mu+1)\;g_{0},\\\\
  \varphi(x)=\frac{\mu+1}{d+1}\left(1-x-(1-x)^{d+2}\right)-\frac{\mu}{d+2}\left(1-(1-x)^{d+2}\right),0\leq x<1,\\\\
 \mu\leq 0, r=1,
 \end{array}\right.
\end{equation}
here  $M=\mathbb{B}^r_{L_0}$ for $\mu<0$, and $M=L_0^{\oplus r}$ for $\mu=0$.

In particular, for $\mu=0$, \eqref{1.6} is simplified as
\begin{equation}\label{1.7}
 \left\{ \begin{array}{l}
  \mathrm{Ric}_{g_{0}}=\;g_{0},\\\\
  \varphi(x)=\frac{1}{d+1}\left(1-x-(1-x)^{d+2}\right),0\leq x<1,\\\\
  r=1.
 \end{array}\right.
\end{equation}
For $\mu=-1$, \eqref{1.6} is simplified as
\begin{equation}\label{1.8}
 \left\{ \begin{array}{l}
  \mathrm{Ric}_{g_{0}}=0,\\\\
  \varphi(x)=\frac{1}{d+2}\left(1-(1-x)^{d+2}\right),0\leq x<1,\\\\
  r=1.
 \end{array}\right.
\end{equation}
For $\mu=-(d+2)$, \eqref{1.6} is simplified as
\begin{equation}\label{1.9}
 \left\{ \begin{array}{l}
  \mathrm{Ric}_{g_{0}}=-(d+1)\;g_{0},\\\\
  \varphi(x)=x,0\leq x<1,\\\\
  F(t)=e^t,-\infty\leq t<0,\\\\
  r=1.
 \end{array}\right.
\end{equation}

 }\end{Corollary}

\begin{Remark}
 We denote the dimension $d$,  the genus $p$, and the generic norm $\mathrm{N}(z,\overline{w})$ for a Cartan domain $\Omega$ (an irreducible bounded symmetric domain $\Omega$ in $\mathbb{C}^{d}$). For a given positive integer $r$ and a positive real number $\mu$, let
 \begin{equation*}
  \phi(z)=-\mu\log \mathrm{N}(z,\overline{z}),\; h=e^{\lambda\phi(z)}\|w\|^2,\;M=\left\{(z,w)\in \Omega\times\mathbb{\mathbb{C}}^{r}: \|w\|^2<\mathrm{N}(z,\overline{z})^{\lambda\mu}\right\}.
\end{equation*}

When $r=1$ and $\mu=\frac{p}{d+1}$, by \eqref{1.9}, $\frac{\sqrt{-1}}{1-h}\partial\overline{\partial}(\phi+h)$ is the K\"ahler form of a complete conformally K\"ahler Einstein metric on $M$ for $\lambda=-1$. Note that $\sqrt{-1}\partial\overline{\partial}(\phi-\log(1-h))$ is the K\"{a}hler form of a complete K\"ahler-Einstein metric on $M$ when $\mu=\frac{p}{d+1}$ and $\lambda=1$ \cite{RWYZ}.
\end{Remark}

 \begin{Corollary}\label{Co:1.2}{
Under the situation of Theorem \ref{Th:1.1},  let $r=1$, $a=\lambda=-1$ and $b=\nu=0$, then $\widetilde{g}$ is  a complete Einstein  metric (namely $\mathrm{Ric}_{\widetilde{g}}=\mu \;\widetilde{g}$)  on $\mathbb{B}^{r*}_{L_0}$ if and only if
\begin{equation}\label{1.10}
 \left\{ \begin{array}{l}
  \mathrm{Ric}_{g_{0}}=-\frac{\mu}{x_0}\;g_{0},\\\\
  \varphi(x)=\mu\left(\frac{1}{d+2}\big((\frac{x}{x_0})^{d+2}-1\big)-\frac{\mu}{d+1}\frac{x}{x_0}\big((\frac{x}{x_0})^{d+1}-1\big)\right),\;-\infty\leq x_0< x<0,\\\\
 \mu<0.
 \end{array}\right.
\end{equation}
In particular, for $\mu=-(d+2)$ and $x_0=-\infty$, \eqref{1.10} is simplified as
 \begin{equation}\label{1.11}
 \left\{ \begin{array}{l}
  \mathrm{Ric}_{g_{0}}=0,\\\\
  \varphi(x)=1,\;-\infty< x<0,\\\\
  F(t)=\frac{1}{2}t^2,-\infty<t<0.
 \end{array}\right.
\end{equation}
 }\end{Corollary}

 \begin{Corollary}\label{Co:1.3}{
Under the situation of Theorem \ref{Th:1.1},  let $a<0,\lambda>0$, $r=1$, $b=1$ and $\nu=0$, then $\widetilde{g}$ is  a complete Einstein  metric (namely $\mathrm{Ric}_{\widetilde{g}}=\mu \;\widetilde{g}$)  on $\mathbb{B}^{r*}_{L_0}$ if and only if
\begin{equation}\label{1.12}
 \left\{ \begin{array}{l}
  \mathrm{Ric}_{g_0}=\frac{\lambda\mu x_0}{(1+ax_0)^2}g_0,\\\\
   \varphi(x)=-\frac{\mu}{(1+ax_0)^2}\frac{(1-ax)(1+ax)^{2n-1}}{x^{n-1}}\int_{x_0}^x\frac{(u-x_0)u^{n-1}(1-a^2x_0u)}{(1-au)^2(1+au)^{2n}}\mathrm{d}u,\;0\leq x_0<x<-\frac{1}{a},\\\\
    \mu<0.
 \end{array}\right.
\end{equation}
 In particular, for $x_0=0$ , \eqref{1.12}  is simplified as
 \begin{equation}\label{1.13}
 \left\{ \begin{array}{l}
  \mathrm{Ric}_{g_0}=0,\\\\
  \varphi(x)=-\frac{\mu}{n+1}x^2-\frac{\mu}{a^2}\sum_{k=3}^{n+1}\left(\frac{k-1}{n+k-1}\prod_{j=1}^{k-2}\frac{n-j}{n+j}\right)(ax)^k,\;-1<ax<0,\\\\
   \mu<0.
 \end{array}\right.
\end{equation}
As $a\rightarrow 0^{-}$,  the limit of \eqref{1.13} is
 \begin{equation}\label{1.14}
 \left\{ \begin{array}{l}
  \mathrm{Ric}_{g_0}=0,\\\\
  \omega_F=\frac{n+1}{\mu}\sqrt{-1}\partial\overline{\partial}\log\left|\lambda\phi+\log|w|^2\right|,\\\\
   \mu<0.
 \end{array}\right.
\end{equation}

 }\end{Corollary}

\begin{Remark}
Let $\phi=\|z\|^2$, $h=e^{\lambda\|z\|^2}|w|^2$, and $M=\{(z,w)\in\mathbb{C}^d\times\mathbb{C}:0<|w|^2<e^{-\lambda\|z\|^2}\}$. Then from \eqref{1.11} and \eqref{1.14}, $\frac{-\sqrt{-1}}{2\log h}\partial\overline{\partial}(\log h)^2$ is the K\"{a}hler form of a complete  conformally K\"ahler Einstein metric on $M$ for $\lambda=-1$, and $-\sqrt{-1}\partial\overline{\partial}\log(-\log h))$ is the K\"{a}hler form of a complete  K\"ahler Einstein metric on $M$ for $\lambda=1$.

\end{Remark}

\begin{Corollary}\label{Co:1.4}{
Under the situation of Theorem \ref{Th:1.1},  let $r=a=1,\lambda>0$ and $b=\nu=0$, then $\widetilde{g}$ is a complete Einstein  metric (namely $\mathrm{Ric}_{\widetilde{g}}=\mu \;\widetilde{g}$)  on  $M=\mathbb{P}(E \oplus 1)-M_0$ if
\begin{equation}\label{1.15}
 \left\{ \begin{array}{l}
  \mathrm{Ric}_{g_{0}}=\lambda\;g_{0},\\\\
 \varphi(x)=\frac{1}{d+1}x\left(1-(\frac{x}{x_1})^{d+1}\right),\;0<x<x_1<+\infty,\\\\
  \mu=0.
 \end{array}\right.
\end{equation}

 }\end{Corollary}

\begin{Corollary}\label{Co:1.5}{
Let $M_0$ be an  irreducible Hermitian symmetric space of compact type, i.e.
 $M_0$ is the compact dual of an irreducibly bounded symmetric domain $\Omega$ in $\mathbb{C}^d$,  $\Omega\subset\mathbb{C}^d\subset M$ and  $\mathbb{C}^d$ is dense in $M_0$. There is a  K\"ahler-Einstein metric $\omega_0$ on $M_0$, the restriction of $\omega_0$ on $\mathbb{C}^d$ can be expressed as
\begin{equation*}
  \omega_0=\sqrt{-1}\partial\overline{\partial}\log \mathrm{N}(z,-\bar{z}),
\end{equation*}
where  $\mathrm{N}(z,\overline{z})$ is the generic norm of the Cartan domain $\Omega$.

Set $\mathcal{O}(1)$ is the positive generator of the Picard group $\mathrm{Pic}(M_0)=\mathbb{Z}$. For positive integers $r>1$ and $\lambda$ satisfying a condition $r\lambda=p$, there exists a metric  $h_0$ on a holomorphic line bundle $L_0=\mathcal{O}(-\lambda)$ such that $2\pi c_1 (L_0,h_0)=-\lambda\,\omega_0$, here $p$ is  the genus of the Cartan domain $\Omega$.

Under the situation of Theorem \ref{Th:1.1},  let $a=1$, $\nu=0$ and $b<0$, then $\widetilde{g}$ is a complete Einstein  metric (namely $\mathrm{Ric}_{\widetilde{g}}=\mu \;\widetilde{g}$)  on  $M=\mathbb{P}(L_0^{\oplus r} \oplus 1)-M_0$ if
\begin{equation}\label{1.19}
 \left\{ \begin{array}{l}
  \varphi(x)= -r\frac{(x+b)^{2n-1}(x-b)}{x^{n-1}}\int_{x_1}^x \frac{ u^{n-1}}{(u+b)^{2n-2}(u-b)^2}\mathrm{d}u,\;-b<x<-\frac{(r+1)b}{(r-1)},\\\\
  \mu=0.
 \end{array}\right.
\end{equation}
 }\end{Corollary}

\begin{Remark}
By Gao-Yau-Zhou \cite{Gao-Yau-Zhou-2017},  if $M$ is a compact K\"{a}hler manifold and $N$ is a subvariety with codimension greater than or equal to 2, then there are no complete K\"{a}hler-Einstein metrics on $M-N$.  Corollary \ref{Co:1.5} shows that this conclusion does not hold for the conformally K\"ahler Einstein metric.
\end{Remark}

The paper is organized as follows. In Section 2, by the momentum profiles $\varphi(x)$ (refer to Hwang-Singer \cite{Hwang-Singer-2002}) of  unitary-invariant K\"ahler metrics $g_F$, we derive ordinary differential equations for conformally K\"ahler, Einstein  metrics. In Section 3, we discuss the completeness of  conformally K\"ahler metrics. In Section 4, by using the conclusions of Section 2 and Section 3, we obtain the proof for Theorem \ref{Th:1.1}.

\section{Equations of  conformally K\"ahler, Einstein  metrics}

For convenience, we need Lemma \ref{Le:4.1} and Remark \ref{Re:2.1} of \cite{Feng-2018}, namely:
\begin{Lemma}\label{Le:4.1}{
Let $\phi$ be globally defined real K\"{a}hler potentials on a domain $\Omega\subset\mathbb{C}^d$, $F$ be a smooth real-valued function on some interval $\mathcal{I} \subset \mathbb{R}$ and
$$\Phi_F(z,w)=\nu\phi(z)+F(t),$$
where $z\in \mathbb{C}^{d}$, $t=\lambda\phi(z)+\log \rho^2$, $\rho=\|w\|,$ $w\in\mathbb{C}^{r}$, and $\lambda,\nu\in\mathbb{R}$.

Set
\begin{equation*}\label{e4.2}
Z=(z,w),\;  T\equiv(T_{i\bar{j}}):=\frac{\partial^2\Phi_F}{\partial Z^t\partial \overline{Z}}\equiv\left(
                                                                                                       \begin{array}{cc}
                                                                                                         T_{11} & T_{12}  \\
                                                                                                         T_{21} & T_{22}  \\                                                                                                     \end{array}
                                                                                                     \right)
\equiv\left(
                                                                     \begin{array}{ccc}
                                                                      \frac{\partial^2\Phi_F}{\partial z^t\partial \bar z}  & \frac{\partial^2\Phi_F}{\partial z^t\partial \bar w} \\\\
                                                                       \frac{\partial^2\Phi_F}{\partial w^t\partial \bar z} & \frac{\partial^2\Phi_F}{\partial w^t\partial \bar w} \\
                                                                     \end{array}
                                                                   \right)
\end{equation*}
and
\begin{equation*}\label{e4.3}
T^{-1}:=\left(
         \begin{array}{cc}
           (T^{-1})_{11} & (T^{-1})_{12} \\\\
           (T^{-1})_{21} & (T^{-1})_{22} \\
         \end{array}
       \right).
\end{equation*}
Then
\begin{equation*}\label{e4.4}
  T_{11}=(\nu+\lambda F')\frac{\partial^2\phi}{\partial z^t\partial \bar z}+\lambda^2F'' \frac{\partial\phi}{\partial z^t}\frac{\partial\phi}{\partial \bar{z}},\;\; T_{12}=\lambda\frac{F''}{\rho^2}\frac{\partial\phi}{\partial z^t}w,
\end{equation*}
\begin{equation*}\label{e4.6}
  T_{21}=\lambda\frac{F''}{\rho^2}\overline{w}^t\frac{\partial\phi}{\partial \bar{z}},\;\;  T_{22}=\frac{F'}{\rho^2}I_{r}+\frac{F''-F'}{\rho^4}\overline{w}^tw,
\end{equation*}
\begin{equation*}\label{e4.8}
   \det T=\frac{1}{\rho^{2r}}(F')^{r-1}F''(\nu+\lambda F')^{d}\det(\frac{\partial^2\phi}{\partial z^t\partial \bar z}).
\end{equation*}
\begin{equation*}\label{e4.9.1}
  (T^{-1})_{11}=\frac{1}{\nu+\lambda F'}\left(\frac{\partial^2\phi}{\partial z^t\partial \bar z}\right)^{-1},
\end{equation*}
\begin{equation*}\label{e4.10}
  (T^{-1})_{12}=-\frac{\lambda}{\nu+\lambda F'}\left(\frac{\partial^2\phi}{\partial z^t\partial \bar z}\right)^{-1}\frac{\partial\phi}{\partial z^t}w,
\end{equation*}
\begin{equation*}\label{e4.11}
  (T^{-1})_{21}=-\frac{\lambda}{\nu+\lambda F'}\overline{w}^t\frac{\partial\phi}{\partial \bar{z}}\left(\frac{\partial^2\phi}{\partial z^t\partial \bar z}\right)^{-1}
\end{equation*}
and
\begin{eqnarray*}
\label{e4.12}    (T^{-1})_{22}   & =& \frac{\rho^2}{F'}I_{r}+\left(\frac{1}{F''}-\frac{1}{F'}\right)\overline{w}^tw+\frac{\lambda^2}{\nu+\lambda F'}\frac{\partial\phi}{\partial \bar{z}}\left(\frac{\partial^2\phi}{\partial z^t\partial \bar z}\right)^{-1}\frac{\partial\phi}{\partial z^t}\overline{w}^tw.
\end{eqnarray*}
Where $Z^t$ and $\overline{Z}$ denote the transpose and the conjugate of the row vector $Z=(z,w)$, respectively, $I_{r}$ denotes the identity matrix of order $r$, and symbols
$$
\frac{\partial}{\partial z^t}=\left(\frac{\partial}{\partial z_1},\cdots,\frac{\partial}{\partial z_d}\right)^t,\;
\frac{\partial}{\partial \bar{z}}=\left(\frac{\partial}{\partial \bar{z_1}},\cdots,\frac{\partial}{\partial \bar{z_d}}\right),\;
\frac{\partial^2}{\partial z^t\partial \bar z}=\left(\frac{\partial^2}{\partial z_i\partial \bar z_j}\right).$$
 }\end{Lemma}

\begin{Remark}\rm\label{Re:2.1}
Under assumptions of Lemma \ref{Le:4.1}, $\Phi_F$ is a K\"{a}hler potential on a domain $M=\{(z,w)\in\Omega\times \mathbb{C}^{r}: t\in \mathcal{I}\}$
if and only if

$\mathrm{(i)}$ $\nu+\lambda F'(t)>0,\;F''(t)>0$ for $r=1$;

$\mathrm{(ii)}$ $\nu+\lambda F'(t)>0,\;F'(t)>0,\;F''(t)>0$ for $r>1$.
\end{Remark}

\begin{Lemma}\label{Le:4.2}{Under the situation of Lemma \ref{Le:4.1}, if a Hermitian matrix $T$ is positive definite, $f$ is a function in $t$ and $\overline{\partial}\left(\frac{\partial f}{\partial \overline{Z}}T^{-1}\right)=0,$ then there exist constants $a,b$ such that $ f=aF'(t)+b.$

}\end{Lemma}

\begin{proof}[Proof]Let $x=F'(t)$ and $\varphi(x)=F''(t)$. By Lemma \ref{Le:4.1}, we have
\begin{equation*}
 \frac{\partial f}{\partial \overline{Z}}=\frac{\mathrm{d} f}{\mathrm{d} x}\varphi(x)\left(\lambda\frac{\partial\phi}{\partial \overline{z}},\frac{w}{\rho^2}\right)
\end{equation*}
and
\begin{equation*}
   \frac{\partial f}{\partial \overline{Z}}T^{-1}=\left(0,\frac{\mathrm{d} f}{\mathrm{d} x}w\right).
\end{equation*}
Thus
\begin{equation*}
   \overline{\partial}\left(\frac{\partial f}{\partial \overline{Z}}T^{-1}\right)=\left(0,\overline{\partial}\left(\frac{\mathrm{d} f}{\mathrm{d} x}\right)w\right)=0.
\end{equation*}
That is $ \overline{\partial}\left(\frac{\mathrm{d}f}{\mathrm{d}x}\right)=0.$

Since $\overline{\partial}\left(\frac{\mathrm{d}f}{\mathrm{d}x}\right)=\frac{\mathrm{d}^2f}{\mathrm{d} x^2}\varphi(x)\overline{\partial}t$ and $\varphi(x)>0$, so $\frac{\mathrm{d}^2 f}{\mathrm{d} x^2}=0.$ Then there exist constants $a,b$ such that $f=ax+b.$

\end{proof}

By Lemma \ref{Le:4.2}, if $g_F$ is a K\"{a}hler metric associated with the K\"{a}hler form $\omega_F=\sqrt{-1}\partial\overline{\partial}\Phi_F$, and $\widetilde{g}=\frac{1}{f^2}g_F$ is a cKEM metric, then $f$ must be forms $f=aF'(t)+b$.

\begin{Theorem}\label{Th:4.1}{Let $g_{\phi}$ be a  K\"{a}hler  metric  on a domain $\Omega\subset\mathbb{C}^d$ associated with the K\"{a}hler form $\omega_{\phi}=\sqrt{-1}\partial\overline{\partial}\phi$, namely $g_{\phi}(X,Y)=\omega_{\phi}(X,JY)$,  where $\phi$ is globally defined  the K\"{a}hler potential on $\Omega$, and $J$ is the canonical complex structure on the complex Euclidean space. Suppose that $F$ is a smooth real-valued function on some interval $\mathcal{I} \subset \mathbb{R}$ such that $$\Phi_F(z,w):=\nu\phi(z)+F(t)\;(t=\lambda\phi(z)+\log\|w\|^2,\;\nu=0,1,\;\lambda\neq 0)$$ is a K\"{a}hler potential on a domain $M=\{(z,w)\in\Omega\times \mathbb{C}^{r}: t\in \mathcal{I}\}$. Set $g_F$ is a K\"{a}hler  metric  on the domain $M$ associated with the K\"{a}hler form $\omega_F=\sqrt{-1}\partial\overline{\partial}\Phi_F$.

Let  $\widetilde{g}=\frac{1}{f^2}g_F$, $n=d+r, x=F'(t), \varphi(x)=F''(t)$ and $f=ax+b\neq 0$ on $M$.

$(\mathrm{I})$ If  the scalar curvature $s_{\widetilde{g}}$ is constant on $M$ for the metric $\widetilde{g}=\frac{1}{f^2}g_F$, then  the scalar curvature $s_{g_{\phi}}$ is constant on $\Omega$ for the metric $g_{\phi}$ and
 \begin{equation}\label{eq7.1}
    \left(\frac{(\nu+\lambda x)^dx^{r-1}}{f^{2n-1}}\varphi\right)''=\frac{(\nu+\lambda x)^dx^{r-1}}{f^{2n-1}}\left(\frac{1}{\nu+\lambda x}k_{g_{\phi}}+\frac{r(r-1)}{x}- \frac{k_{\widetilde{g}}}{f^2}\right),
 \end{equation}
where $k_{g_{\phi}}=\frac{1}{2}s_{g_{\phi}}$ and $k_{\widetilde{g}}=\frac{1}{2}s_{\widetilde{g}}$

$(\mathrm{II})$   $\widetilde{g}=\frac{1}{f^2}g_F$ is an Einstein metric on $M$ ( that is $\mathrm{Ric}_{\widetilde{g}}=\mu\; \widetilde{g}$) if and only if
\begin{equation}\label{eq7.3.4}
  \left\{\begin{array}{l}
     \mathrm{Ric}_{g_{\phi}}=\gamma\; g_{g_{\phi}},\\\\
            \varphi'(x)+\left(\frac{d\lambda }{\lambda x+\nu}-\frac{(2d+1)a}{ax+b}-\frac{\lambda a}{a\lambda x+2a\nu-b\lambda}\right)\varphi(x)
   -\frac{\mu(\nu+\lambda x)-\gamma (ax+b)^2}{(ax+b)(a\lambda x+2a\nu-b\lambda)}=0, \\\\
            r=1
         \end{array}
  \right.
\end{equation}
for $\nu a\neq 0$, and
\begin{equation}\label{eq7.3.5}
  \left\{\begin{array}{l}
         \mathrm{Ric}_{g_{\phi}}=\gamma\; g_{g_{\phi}},\\\\
             \varphi'(x)+\left(\frac{d\lambda }{\lambda x+\nu}+\frac{r-1}{x}-\frac{(2n-1)a}{ax+b}-\frac{\lambda a}{a\lambda x+2a\nu-b\lambda}\right)\varphi(x)
   -\frac{\mu(\nu+\lambda x)-\gamma (ax+b)^2}{(ax+b)(a\lambda x+2a\nu-b\lambda)}=0
         \end{array}
  \right.
\end{equation}
for $\nu a=0$. Where $\gamma$ is a constant,  and $\gamma=\frac{\mu\nu}{f^2}+r\lambda=\frac{\mu\nu}{b^2}+r\lambda$ for $r>1$ and $\nu a=0$.
 }\end{Theorem}

\begin{proof}[Proof]

$(\mathrm{I})$

Set
\begin{equation*}
   \rho=\|w\|,\;Z=(z,w),\;  T\equiv(T_{i\bar{j}}):=\frac{\partial^2\Phi_F}{\partial Z^t\partial \overline{Z}}\equiv\left(
                                                                                                       \begin{array}{cc}
                                                                                                         T_{11} & T_{12}  \\
                                                                                                         T_{21} & T_{22}  \\                                                                                                     \end{array}
                                                                                                     \right)
\equiv\left(
                                                                     \begin{array}{ccc}
                                                                      \frac{\partial^2\Phi_F}{\partial z^t\partial \bar z}  & \frac{\partial^2\Phi_F}{\partial z^t\partial \bar w} \\\\
                                                                       \frac{\partial^2\Phi_F}{\partial w^t\partial \bar z} & \frac{\partial^2\Phi_F}{\partial w^t\partial \bar w} \\
                                                                     \end{array}
                                                                   \right)
\end{equation*}
and
\begin{equation*}
\begin{array}{c}
  T^{-1}:=\left(
         \begin{array}{cc}
           (T^{-1})_{11} & (T^{-1})_{12} \\\\
           (T^{-1})_{21} & (T^{-1})_{22} \\
         \end{array}
       \right).
\end{array}
\end{equation*}

Let $k_{\widetilde{g}}=\frac{1}{2}s_{\widetilde{g}}$ and $k_{g_F}=\frac{1}{2}s_{g_F}$, then
\begin{equation*}
  k_{\widetilde{g}} =f^2k_{g_F}+2(2n-1)f\triangle_{g_F} f-n(2n-1)|\mathrm{d}f|_{g_F}^2=f^2k_{g_F}-\frac{2(2n-1)}{n-1}f^{n+1}\triangle_{g_F} f^{-n+1},
\end{equation*}
 where
\begin{equation*}
  \begin{array}{c}
     \triangle_{g_F}=\mathrm{Tr}\left(T^{-1}\frac{\partial^2}{\partial Z^t\partial \bar{Z}}\right),\;k_{g_F}=-\mathrm{Tr}\left(T^{-1}\frac{\partial^2\log\det T}{\partial Z^t\partial \bar{Z}}\right),
  \end{array}
\end{equation*}
for the expression of $T$, see Lemma \ref{Le:4.1}.

Since
\begin{equation*}
   \frac{\partial t}{\partial Z}=\left(\lambda\frac{\partial \phi}{\partial z},\frac{\overline{w}}{\rho^2}\right),\;
  \frac{\partial t}{\partial \overline{Z}}=\left(\lambda\frac{\partial \phi}{\partial \overline{z}},\frac{w}{\rho^2}\right),
  \frac{\partial^2t}{\partial Z^t\partial\overline{Z}}=\left(
                                                         \begin{array}{cc}
                                                           \lambda\frac{\partial^2\phi}{\partial z^t\partial \overline{z}} & 0 \\
                                                           0 & \frac{1}{\rho^2}I_{r}-\frac{\overline{w}^tw}{\rho^4} \\
                                                         \end{array}
                                                       \right)
\end{equation*}
and
\begin{equation*}
    \frac{\partial^2f^{\alpha}}{\partial Z^t\partial \overline{Z}}=((f^{\alpha})'\varphi)'\varphi\frac{\partial t}{\partial Z^t}\frac{\partial t}{\partial \overline{Z}}
  +(f^{\alpha})'\varphi \frac{\partial^2t}{\partial Z^t\partial\overline{Z}},
\end{equation*}
by Lemma \ref{Le:4.1}, it follows that
\begin{eqnarray*}
   & &  \mathrm{Tr}\left(T^{-1}\frac{\partial t}{\partial Z^t}\frac{\partial t}{\partial \overline{Z}}\right)=\frac{\partial t}{\partial \overline{Z}}T^{-1}\frac{\partial t}{\partial Z^t}  \\   
   &=&\left(\lambda\frac{\partial \phi}{\partial \overline{z}},\frac{w}{\rho^2}\right)\left(
         \begin{array}{cc}
           (T^{-1})_{11} & (T^{-1})_{12} \\\\
           (T^{-1})_{21} & (T^{-1})_{22} \\
         \end{array}
       \right)\left(
                \begin{array}{c}
                  \lambda\frac{\partial\phi}{\partial z^t} \\
                  \frac{\overline{w}^t}{\rho^2} \\
                \end{array}
              \right)  \\
   &=& \lambda^2\frac{\partial\phi}{\partial \overline{z}}(T^{-1})_{11} \frac{\partial\phi}{\partial z^t}+\lambda\frac{\partial\phi}{\partial \overline{z}}(T^{-1})_{12}\frac{\overline{w}^t}{\rho^2}+\lambda\frac{w}{\rho^2}(T^{-1})_{21}\frac{\partial\phi}{\partial z^t}+\frac{1}{\rho^4}w(T^{-1})_{22}\overline{w}^t  \\
   &=& \frac{\lambda^2}{\nu+\lambda F'}(1-1-1+1)\frac{\partial\phi}{\partial \bar{z}}\left(\frac{\partial^2\phi}{\partial z^t\partial \bar z}\right)^{-1}\frac{\partial\phi}{\partial z^t}+\frac{1}{F'}+(\frac{1}{F''}-\frac{1}{F'})\\
   &=&\frac{1}{F''}=\frac{1}{\varphi},
\end{eqnarray*}
\begin{eqnarray*}
   && \mathrm{Tr}\left(T^{-1}\frac{\partial^2 t}{\partial Z^t\partial \overline{Z}}\right) \\
   &=& \lambda \mathrm{Tr}\left((T^{-1})_{11}\frac{\partial^2 \phi}{\partial z^t\partial \overline{z}}\right) +\mathrm{Tr}\left((T^{-1})_{22}\left(\frac{1}{\rho^2}I_{r}-\frac{\overline{w}^tw}{\rho^4}\right)\right) \\
   &=& \frac{\lambda d}{\nu+\lambda F'}+\frac{1}{F'}\mathrm{Tr}\left(I_{r}-\frac{\overline{w}^tw}{\rho^2}  \right) \\
   &=&\frac{\lambda d}{\nu+\lambda F'}+\frac{r-1}{F'} =\frac{\lambda d}{\nu+\lambda x}+\frac{r-1}{x},   
\end{eqnarray*}

\begin{eqnarray*}
   & &  \triangle_{g_F} f^{\alpha} =\mathrm{Tr}\left(T^{-1}\frac{\partial^2 f^{\alpha}}{\partial Z^t\partial \overline{Z}}\right)   \\  
   &=& ((f^{\alpha})'\varphi)'\varphi\mathrm{Tr}\left(T^{-1}\frac{\partial t}{\partial Z^t}\frac{\partial t}{\partial \overline{Z}}\right)
  +(f^{\alpha})'\varphi \mathrm{Tr}\left(T^{-1}\frac{\partial^2 t}{\partial Z^t\partial \overline{Z}}\right) \\
   &=& \frac{((f^{\alpha})'\varphi)'\varphi}{\varphi}+(f^{\alpha})'\varphi\left(\frac{\lambda d}{\nu+\lambda x}+\frac{r-1}{x} \right)\\
   &=& ((f^{\alpha})'\varphi)'+\left(\frac{\lambda d}{\nu+\lambda x}+\frac{r-1}{x}\right)(f^{\alpha})'\varphi
\end{eqnarray*}
and
\begin{equation*}
    f^{-\alpha}\triangle_{g_F} f^{\alpha}=\frac{(f^{\alpha})'}{f^{\alpha}}\varphi'+\left(\frac{(f^{\alpha})''}{f^{\alpha}}+\frac{(f^{\alpha})'}{f^{\alpha}}\frac{p'}{p}\right)\varphi,
\end{equation*}
where $p:=(\nu +\lambda x)^dx^{r-1}$.

So by (2.24) of \cite{Feng-2018}, we have
\begin{eqnarray*}
  &  \frac{k_{\widetilde{g}}}{f^2}&= k_{g_F}+\frac{2(2n-1)}{\alpha}f^{-\alpha}\triangle_{g_F} f^{\alpha}  \\
   && =\left(\frac{1}{\nu+\lambda x}k_{g_{\phi}}+\frac{r(r-1)}{x}-\frac{\left(p\varphi\right)''}{p}\right) +\frac{2(2n-1)}{\alpha}\left(\frac{(f^{\alpha})'}{f^{\alpha}}\varphi'+\Big(\frac{(f^{\alpha})''}{f^{\alpha}}+\frac{(f^{\alpha})'}{f^{\alpha}}\frac{p'}{p}\Big)\varphi \right) \\
   &&=\frac{1}{\nu+\lambda x}k_{g_{\phi}}+\frac{r(r-1)}{x}-\left(\varphi''+2\Big(\frac{p'}{p}-\frac{(2n-1)}{\alpha}\frac{(f^{\alpha})'}{f^{\alpha}}\Big)\varphi'\right.\\
   &&\;\;\; \left.+\Big(\frac{p''}{p}-\frac{2(2n-1)}{\alpha}\Big(\frac{(f^{\alpha})''}{f^{\alpha}}+\frac{(f^{\alpha})'}{f^{\alpha}}\frac{p'}{p}\Big)\Big)\varphi\right),
\end{eqnarray*}
here $\alpha=1-n$. Thus, we get
\begin{eqnarray*}
   & & \varphi''+2\left(\frac{p'}{p}-\frac{(2n-1)}{\alpha}\frac{(f^{\alpha})'}{f^{\alpha}}\right)\varphi'
   +\left(\frac{p''}{p}-\frac{2(2n-1)}{\alpha}\Big(\frac{(f^{\alpha})''}{f^{\alpha}}+\frac{(f^{\alpha})'}{f^{\alpha}}\frac{p'}{p}\Big)\right)\varphi   \\
   &=& \frac{1}{\nu+\lambda x}k_{g_{\phi}}+\frac{r(r-1)}{x}- \frac{k_{\widetilde{g}}}{f^2}.
\end{eqnarray*}
Using
\begin{equation*}
    \frac{(f^{\alpha})''}{f^{\alpha}}+\left(\frac{2n-1}{\alpha}+1\right)\left(\frac{(f^{\alpha})'}{f^{\alpha}}\right)^2=0,
\end{equation*}
we have \eqref{eq7.1}.

$(\mathrm{II})$ We complete the proof in four steps.

$(\mathrm{II-1})$ In this part, we prove the following conclusions:

 $\widetilde{g}=\frac{1}{f^2}g_F$ is an Einstein metric on $M$ ( that is $\mathrm{Ric}_{\widetilde{g}}=\mu \;\widetilde{g}$) if and only if
\begin{equation}\label{eq7.2.1}
\mathrm{Ric}_{g_{\phi}}=\gamma\; g_{g_{\phi}},
\end{equation}
\begin{equation}\label{eq7.2.1.2}
  \varphi'+\left(\frac{p'}{p}-\frac{(2n-1)a}{f}+\frac{a\lambda}{\lambda f-2a(\nu+\lambda x)}\right)\varphi+\frac{\mu(\nu+\lambda x)-\gamma f^2}{f(\lambda f-2a(\nu+\lambda x))}=0
\end{equation}
and
\begin{equation}\label{eq7.2}
 \varphi''(x)+\left(\frac{p'}{p}-\frac{2an}{f}\right)\varphi'+\left(\frac{2(2n-1)a^2}{f^2}-\frac{2ap'}{pf}+\left(\frac{p'}{p}\right)'\right)\varphi+\frac{\mu}{f^2}=0,
\end{equation}
where $p=(\nu+\lambda x)^dx^{r-1}$, and $\gamma$ is a constant. For $r>1$, $\varphi(x)$ is also satisfied
\begin{equation}\label{eq7.3}
\left(1-\frac{2ax}{f}\right)\varphi'+\left(\frac{p'}{p}-\frac{2a(n-1)}{f}-\frac{2axp'}{pf}+\frac{2(2n-1)a^2x}{f^2}\right)\varphi+\frac{\mu x}{f^2}-r=0.
\end{equation}

Using (see e.g. \cite{Besse}, 1.161)
\begin{equation*}
    \mathrm{Ric}_{\widetilde{g}}=\mathrm{Ric}_{g_F}+\frac{2(n-1)}{f}D_{g_F}^2f-\frac{1}{f^2}(-2f\triangle_{g_F} f+(2n-1)|\mathrm{d}f|_{g_F}^2)g_F,
\end{equation*}
here $D_{g_F}$ denotes the Levi-Civita connection of $g_F$, we get
\begin{equation*}
    \mathrm{Ric}_{g_F}+\frac{2(n-1)}{f}D_{g_F}^2f=\frac{1}{f^2}(\mu-2f\triangle_{g_F} f+(2n-1)|\mathrm{d}f|_{g_F}^2)g_F,
\end{equation*}
namely
\begin{equation*}
     \frac{2(n-1)}{f}\sqrt{-1}\partial\overline{\partial}f=
   \frac{1}{f^2}(\mu-2f\triangle_{g_F} f+(2n-1)|\mathrm{d}f|_{g_F}^2)\sqrt{-1}\partial\overline{\partial}\Phi_F+\sqrt{-1}\partial\overline{\partial}\log\det(\partial\overline{\partial}\Phi_F)
\end{equation*}
or
\begin{equation}\label{eq7.5}
     \frac{2(n-1)}{f} \frac{\partial^2f}{\partial Z^t\partial\overline{Z}}=
   \frac{1}{f^2}(\mu-2f\triangle_{g_F} f+(2n-1)|\mathrm{d}f|_{g_F}^2)
   \frac{\partial^2\Phi_F}{\partial Z^t\partial\overline{Z}}+\frac{\partial^2\log\det (\partial\overline{\partial}\Phi_F)}{\partial Z^t\partial\overline{Z}}.
\end{equation}

Let
\begin{eqnarray*}
 Q  &=& (r-1)\log F'(t)+\log F''(t)+d\log(\nu+\lambda F'(t))-rt, \\\\
   P  &=& \log\det(\partial\overline{\partial}\Phi_F)=Q+\log\det(\partial\overline{\partial}\phi)+\lambda r\phi, \\\\
  Ric_{g_{\phi}}   &=&  -\frac{\partial^2}{\partial z^t\partial\bar{z}}\log\left(\det\left(\frac{\partial^2\phi}{\partial z^t\partial \bar z}\right)\right).
\end{eqnarray*}
Note that
\begin{eqnarray*}
   Q'(t)    & =&r-1)\frac{F''(t)}{F'(t)}+\frac{F'''(t)}{F''(t)}+d\frac{\lambda F''(t)}{\nu+\lambda F'(t)}-r \\
     &=&\varphi'(x)+\left(\frac{r-1}{x}++\frac{\lambda d}{\nu+\lambda x}\right)\varphi(x)-r  \\
     & = &\frac{(p\varphi)'}{p}(x)-r
\end{eqnarray*}
and
\begin{equation*}
  Q''(t)\;\;=\;\;\left( \frac{(p\varphi)'}{p}\right)'(x)\cdot\varphi(x),
\end{equation*}
so
\begin{eqnarray*}
   & &  \frac{\partial^2\log\det (\partial\overline{\partial}\Phi_F)}{\partial Z^t\partial\overline{Z}}\\
   &=&  Q''\frac{\partial t}{\partial Z^t}\frac{\partial t}{\partial \overline{Z}}+Q'\frac{\partial^2t}{\partial Z^t\partial \overline{Z}}
 +\left(
                             \begin{array}{cc}
                              \frac{\partial^2\log\det(\partial\overline{\partial}\phi)}{\partial z^t\partial\overline{z}}+\lambda r\frac{\partial^2\phi}{\partial z^t\partial\overline{z}}  & 0 \\
                               0 & 0 \\
                             \end{array}
                           \right) \\
   &=& Q''\left(
           \begin{array}{cc}
             \lambda^2\frac{\partial\phi}{\partial z^t}\frac{\partial\phi}{\partial \overline{z}} & \frac{\lambda}{\rho^2}\frac{\partial\phi}{\partial z^t}w \\\\
             \frac{\lambda}{\rho^2}\overline{w}^t\frac{\partial\phi}{\partial\overline{z}} & \frac{1}{\rho^4}\overline{w}^tw \\
           \end{array}
         \right)+Q'\left(
                                                         \begin{array}{cc}
                                                           \lambda\frac{\partial^2\phi}{\partial z^t\partial \overline{z}} & 0 \\
                                                           0 & \frac{1}{\rho^2}I_{r}-\frac{\overline{w}^tw}{\rho^4} \\
                                                         \end{array}
                                                       \right)
                                                       +\left(
                             \begin{array}{cc}
                              \frac{\partial^2\log\det(\partial\overline{\partial}\phi)}{\partial z^t\partial\overline{z}}
                              +\lambda r\frac{\partial^2\phi}{\partial z^t\partial\overline{z}}  & 0 \\
                               0 & 0 \\
                             \end{array}
                           \right) \\
   &=&  \left(
         \begin{array}{cc}
         \lambda(Q'+r)\frac{\partial^2\phi}{\partial z^t\partial\overline{z}}
         + \frac{\partial^2\log\det(\partial\overline{\partial}\phi)}{\partial z^t\partial\overline{z}}
         +\lambda^2Q''\frac{\partial\phi}{\partial z^t}\frac{\partial\phi}{\partial \overline{z}}  & \frac{\lambda Q''}{\rho^2}\frac{\partial\phi}{\partial z^t}w \\\\
         \frac{\lambda Q''}{\rho^2}\overline{w}^t\frac{\partial\phi}{\partial\overline{z}}   &\frac{Q'}{\rho^2}I_{r}+\frac{Q''-Q'}{\rho^4}\overline{w}^tw \\
         \end{array}
       \right),
\end{eqnarray*}
\begin{equation*}
    \frac{\partial f}{\partial Z^t}=f'\varphi\left(\lambda\frac{\partial \phi}{\partial z},\frac{\overline{w}}{\rho^2}\right)^t,\;
  \frac{\partial f}{\partial \overline{Z}}=f'\varphi\left(\lambda\frac{\partial \phi}{\partial \overline{z}},\frac{w}{\rho^2}\right),
\end{equation*}
\begin{equation*}
  \frac{\partial^2f}{\partial Z^t\partial\overline{Z}}=(f'\varphi)'\varphi \frac{\partial t}{\partial Z^t} \frac{\partial t}{\partial \overline{Z}}
 +f'\varphi\frac{\partial^2 t}{\partial Z^t\partial \overline{Z}}=a\varphi\left(
           \begin{array}{cc}
            \lambda\frac{\partial^2\phi}{\partial z^t\partial \overline{z}}+ \lambda^2\varphi'\frac{\partial\phi}{\partial z^t}\frac{\partial\phi}{\partial \overline{z}} & \frac{\lambda \varphi'}{\rho^2}\frac{\partial\phi}{\partial z^t}w \\\\
             \frac{\lambda \varphi'}{\rho^2}\overline{w}^t\frac{\partial\phi}{\partial\overline{z}} &\frac{1}{\rho^2}I_{r}+ \frac{\varphi'-1}{\rho^4}\overline{w}^tw \\
           \end{array}
         \right),
\end{equation*}
\begin{equation*}
     |\mathrm{d}f|_{g_F}^2 =2\mathrm{Tr}\left(T^{-1}\frac{\partial f}{\partial Z^t}\frac{\partial f}{\partial \overline{Z}}\right)
     =2a^2\varphi^2 \mathrm{Tr}\left(T^{-1}\frac{\partial t}{\partial Z^t}\frac{\partial t}{\partial \overline{Z}}\right)
     =2a^2\varphi
\end{equation*}
and
\begin{eqnarray*}
  \triangle_{g_F} f   &=& \mathrm{Tr}\left(T^{-1}\frac{\partial^2f}{\partial Z^t\partial\overline{Z}}\right)=a\varphi'\varphi\mathrm{Tr}\left(T^{-1}\frac{\partial t}{\partial Z^t}\frac{\partial t}{\partial \overline{Z}}\right)
  +a\varphi \mathrm{Tr}\left(T^{-1}\frac{\partial^2 t}{\partial Z^t\partial \overline{Z}}\right) \\
   &=& a\varphi'+a\left(\frac{\lambda d}{\nu+\lambda x}+\frac{r-1}{x}\right)\varphi=a\frac{(p\varphi)'}{p}.
\end{eqnarray*}

Let
\begin{eqnarray*}
  A &:=& x\frac{\mu-2f\triangle_{g_F} f+(2n-1)|\mathrm{d}f|_{g_F}^2}{f^2}+Q'-2a(n-1)\frac{\varphi}{f}  \\
   &=& x\left(\frac{\mu}{f^2}+\frac{-2a(p\varphi)'}{pf}+\frac{2(2n-1)a^2\varphi}{f^2}\right)+\frac{(p\varphi)'}{p}-r-2a(n-1)\frac{\varphi}{f}  \\
   &=& (1-\frac{2ax}{f})\varphi'+\left(\frac{p'}{p}-\frac{2a(n-1)}{f}-\frac{2axp'}{pf}+\frac{2(2n-1)a^2x}{f^2}\right)\varphi+\frac{\mu x}{f^2}-r,
\end{eqnarray*}
\begin{eqnarray*}
  D &:=& (\varphi- x)\frac{\mu-2f\triangle_{g_F} f+(2n-1)|\mathrm{d}f|_{g_F}^2}{f^2}+Q''-Q'-2a(n-1)\frac{(\varphi'-1)\varphi}{f},
\end{eqnarray*}
\begin{eqnarray*}
   C &:=&\varphi\frac{\mu-2f\triangle_{g_F} f+(2n-1)|\mathrm{d}f|_{g_F}^2}{f^2}+Q''-2a(n-1)\frac{\varphi'\varphi}{f}  \\
   &=& \left(\frac{\mu}{f^2}+\frac{-2a\left(\varphi'+\frac{p'}{p}\varphi\right)}{f}+\frac{2(2n-1)a^2\varphi}{f^2}
   +\varphi''+\frac{p'}{p}\varphi'+\left(\frac{p'}{p}\right)'\varphi-2a(n-1)\frac{\varphi'}{f}\right)\varphi \\
   &=& \left(\varphi''(x)+\Big(\frac{p'}{p}-\frac{2an}{f}\Big)\varphi'+\Big(\frac{2(2n-1)a^2}{f^2}-\frac{2ap'}{pf}+\Big(\frac{p'}{p}\Big)'\Big)\varphi+\frac{\mu}{f^2}\right) \varphi
\end{eqnarray*}
and
\begin{eqnarray}
\nonumber \gamma   &:=& \frac{\mu-2f\triangle_{g_F} f+(2n-1)|\mathrm{d}f|_{g_F}^2}{f^2}(\nu+\lambda x)+\lambda(Q'+r)-\frac{2(n-1)a\lambda}{f}\varphi  \\
\nonumber   &=& \frac{\mu-2f\triangle_{g_F} f+(2n-1)|\mathrm{d}f|_{g_F}^2}{f^2}\nu+\lambda r+\lambda A \\
\nonumber   &=& \nu\left(-\frac{2a}{f}\varphi'+\Big(\frac{2(2n-1)a^2}{f^2}-\frac{2ap'}{pf}\Big)\varphi+\frac{\mu}{f^2}\right)+\lambda r+\lambda A \\
\label{e7.8.2.1}   &=&-\frac{2a\nu}{f}\left(\varphi'+\Big(\frac{p'}{p}-\frac{(2n-1)a}{f}\Big)\varphi-\frac{\mu}{2af}\right)+\lambda r+\lambda A  \\
\label{e7.8.2}   &=& \frac{\lambda f-2a(\nu+\lambda x)}{f}\left(\varphi'+\Big(\frac{p'}{p}-\frac{(2n-1)a}{f}\Big)\varphi\right)+\frac{\lambda a}{f}\varphi+\frac{\mu(\nu+\lambda x)}{f^2}.
\end{eqnarray}
Note that the equation \eqref{e7.8.2} is equivalent to the equation \eqref{eq7.2.1.2}.

The equation \eqref{eq7.5}  is equivalent to the following system of equations
\begin{eqnarray}
 \label{eq7.6}    &&  \frac{A}{\rho^2}I_{r}+\frac{D}{\rho^4}\overline{w}^tw=0, \\
\label{eq7.6.1}   &&  C\frac{\lambda}{\rho^2}\frac{\partial\phi}{\partial z^t}w=0, \\
 \label{eq7.6.2}  &&   \gamma\frac{\partial^2 \phi}{\partial z^t\partial\overline{z}}+\frac{\partial^2\log\det(\partial\overline{\partial} \phi)}{\partial z^t\partial\overline{z}}
+C\lambda^2\frac{\partial\phi}{\partial z^t}\frac{\partial\phi}{\partial \overline{z}}=0.
\end{eqnarray}

Using \eqref{eq7.6.1}, we get
\begin{equation*}
     C\frac{\lambda}{\rho^2}\frac{\partial\phi}{\partial z^t}w\overline{w}^t=C\lambda\frac{\partial\phi}{\partial z^t}=0,
\end{equation*}
 namely $ C\frac{\partial\phi}{\partial z^t}=0.$
Then
\begin{equation*}
    \lambda\frac{\mathrm{d}C}{\mathrm{d}x}\varphi \frac{\partial\phi}{\partial z^t}\frac{\partial\phi}{\partial \overline{z}}+C\frac{\partial^2\phi}{\partial z^t\partial \overline{z}}=0,
\end{equation*}
which implies $C^2\frac{\partial^2\phi}{\partial z^t\partial \overline{z}}=0.$
Since $ \frac{\partial^2\phi}{\partial z^t\partial \overline{z}}>0,$ it follows that $  C\equiv A+D=0.$

Using \eqref{eq7.6} and
\begin{equation*}
    \det\left(\beta I_{r}-\frac{A}{\rho^2}I_{r}-\frac{D}{\rho^4}\overline{z}^tz\right)=\left(\beta-\frac{A}{\rho^2}\right)^{r-1}\left(\beta-\frac{A+D}{\rho^2}\right),
\end{equation*}
we have $  A+D=0$ for $r=1$, and $ A=0,\;A+D=0$ for $r>1$. Thus we have \eqref{eq7.2} and \eqref{eq7.3}.

Substituting $C=0$ into \eqref{eq7.6.2}, we obtain
\begin{equation*}
   \gamma\frac{\partial^2 \phi}{\partial z^t\partial\overline{z}}+\frac{\partial^2\log\det(\partial\overline{\partial} \phi)}{\partial z^t\partial\overline{z}}=0.
\end{equation*}

For the case of $d=1$, taking the partial derivative  with respect to the variable $w$ for the above equation, by $\frac{\partial^2 \phi}{\partial z^t\partial\overline{z}}>0$, we have $\frac{\mathrm{d}\gamma}{\mathrm{d}x}\varphi\frac{\overline{w}^t}{\rho^2}=0, $ thus $\frac{\mathrm{d}\gamma}{\mathrm{d}x}=0.$
 This indicates that $\gamma$ is a constant and $g_{\phi}$ is a K\"ahler-Einstein  metric. That is, equation \eqref{eq7.2.1} holds.

For the case of $d>1$, we have that $\gamma$ is a constant and $g_{\phi}$ is a K\"ahler-Einstein  metric.

$(\mathrm{II-2})$ In this part, we prove the following conclusions:
 when $\nu a=0$,   $\mathrm{Ric}_{\widetilde{g}}=\mu \;\widetilde{g}$ if and only if
\begin{equation}\label{eq7.2.2}
\mathrm{Ric}_{g_{\phi}}=\gamma \;g_{g_{\phi}}
\end{equation}
and
\begin{equation}\label{eq7.3.1}
(1-\frac{2ax}{f})\varphi'+\left(\frac{p'}{p}-\frac{2a(n-1)}{f}-\frac{2axp'}{pf}+\frac{2(2n-1)a^2x}{f^2}\right)\varphi+\frac{\mu x}{f^2}=\left\{\begin{array}{l}
                                                                                                                                                     c,\;r=1, \\
                                                                                                                                                     r,\;r>1.
                                                                                                                                                   \end{array}
\right.
\end{equation}
Where $c$ and
\begin{equation}\label{eq7.3.2}
 \gamma=\frac{\mu\nu}{f^2}+\lambda\times\left\{\begin{array}{l}
                                                                                                                                                     c,\;r=1, \\
                                                                                                                                                     r,\;r>1
                                                                                                                                                   \end{array}
\right.
\end{equation}
are   constants.

When $\nu a=0$, from \eqref{e7.8.2.1} we get
\begin{equation*}
  \gamma=\frac{\mu\nu}{f^2}+\lambda r+\lambda A.
\end{equation*}
By $\gamma$ and $\frac{\mu\nu}{f^2}$ are constants, then $A$ is a constant. Namely
\begin{equation}\label{eq7.3-1}
A\equiv (1-\frac{2ax}{f})\varphi'+\left(\frac{p'}{p}-\frac{2a(n-1)}{f}-\frac{2axp'}{pf}+\frac{2(2n-1)a^2x}{f^2}\right)\varphi+\frac{\mu x}{f^2}-r=\widetilde{c}.
\end{equation}

Now, we show that if $\nu a=0$ and $A$ is constant, then the equation \eqref{eq7.2} holds.

 Let
\begin{equation*}
   Q_1 :=\frac{\frac{p'}{p}-\frac{2a(n-1)}{f}-\frac{2axp'}{pf}+\frac{2(2n-1)a^2x}{f^2}}{1-\frac{2ax}{f}} =-\frac{a}{ax-b}+\frac{r-1}{x}+\frac{d\lambda}{\lambda x+\nu}-\frac{a(2n-1)}{ax+b},
\end{equation*}
\begin{equation*}
   Q_2: =\frac{\frac{\mu x}{f^2}-r}{1-\frac{2ax}{f}}=r+\frac{4abr-\mu}{2a(ax-b)}-\frac{\mu}{2a(ax+b)},
\end{equation*}
\begin{equation*}
  P_1  : =\frac{p'}{p}-\frac{2an}{f}=\frac{d\lambda}{\lambda x+\nu}+\frac{r-1}{x}-\frac{2an}{ax+b},
\end{equation*}
\begin{equation*}
   P_2 :=\frac{2(2n-1)a^2}{f^2}-\frac{2ap'}{pf}+\left(\frac{p'}{p}\right)',
\end{equation*}
\begin{equation*}
   P_3  :=\frac{\mu}{f^2},
\end{equation*}
\begin{equation*}
   R : =\varphi''+P_1\varphi'+P_2\varphi+P_3,
\end{equation*}
\begin{equation*}
  S :=\varphi'+Q_1\varphi+Q_2.
\end{equation*}
Then from
\begin{equation*}
  S= \varphi'+Q_1\varphi+Q_2,
\end{equation*}
we get
\begin{equation*}
   S'=\varphi''+Q_1\varphi'+Q_1'\varphi+Q_2'.
\end{equation*}
Thus
\begin{equation*}
    S'-R=(Q_1-P_1)\left(\varphi'+\frac{Q_1'-P_2}{Q_1-P_1}\varphi+\frac{Q_2'-P_3}{Q_1-P_1}\right),
\end{equation*}
where
\begin{equation*}
   Q_1-P_1=-\frac{2ab}{a^2x^2-b^2},\;\frac{Q_1'-P_2}{Q_1-P_1}=Q_1+\frac{d\nu a}{b(\nu+\lambda x)},\;\frac{Q_2'-P_3}{Q_1-P_1}=Q_2.
\end{equation*}
This shows that
\begin{equation*}
    R= S'-(Q_1-P_1)S+\frac{2d\nu a^2}{(a^2x^2-b^2)(\nu+\lambda x)}\varphi.
\end{equation*}
Therefore, if $\nu a=0$ and $S=\frac{\widetilde{c}}{1-\frac{2ax}{f}}$, then $R=0$. That is,  the equation \eqref{eq7.3-1} implies the equation \eqref{eq7.2}.

It is easy to verify that \eqref{eq7.3.1} is equivalent to \eqref{eq7.2.1.2} for $\nu a=0$.

Let $c=r+\widetilde{c}$, by \eqref{eq7.2.1}, \eqref{eq7.2.1.2} and \eqref{eq7.2}, we obtain \eqref{eq7.2.2}, \eqref{eq7.3.1} and \eqref{eq7.3.2}.

$(\mathrm{II-3})$ In this part, we prove the following conclusions:
when $\nu a\neq 0$,  $\mathrm{Ric}_{\widetilde{g}}=\mu \;\widetilde{g}$ if and only if  $r=1$,
\begin{equation}\label{eq7.2.3}
 \mathrm{Ric}_{g_{\phi}}=\gamma \;g_{g_{\phi}}
\end{equation}
and
\begin{equation}\label{eq7.3.3}
  \varphi'+\left(\frac{p'}{p}-\frac{(2n-1)a}{f}+\frac{\lambda a}{\lambda f-2a(\nu+\lambda x)}\right)\varphi+\frac{\mu(\nu+\lambda x)-\gamma f^2}{f(\lambda f-2a(\nu+\lambda x))}=0,
\end{equation}
where $\gamma$ is a constant on $M$.

If $\nu a\neq 0$ and $r>1$, then by \eqref{eq7.3} (that is $A=0$) and
\begin{eqnarray*}
  \gamma  &=&-\frac{2a\nu}{f}\left(\varphi'+\Big(\frac{p'}{p}-\frac{(2n-1)a}{f}\Big)\varphi-\frac{\mu}{2af}\right)+\lambda r+\lambda A  \\
   &=& -\frac{2a\nu}{f}\left(\varphi'+\Big(\frac{p'}{p}-\frac{(2n-1)a}{f}\Big)\varphi-\frac{\mu}{2af}\right)+\lambda r
\end{eqnarray*}
is a constant, we have
\begin{eqnarray}
\label{eq2.22}   && \left(1-\frac{2ax}{f}\right)\left(\varphi'+\Big(\frac{p'}{p}-\frac{(2n-1)a}{f}\Big)\varphi\right)+\frac{a}{f}\varphi+\frac{\mu x}{f^2}-r=0,  \\
\label{eq2.23}   &&   \varphi'+\left(\frac{p'}{p}-\frac{(2n-1)a}{f}\right)\varphi=\frac{\mu}{2af}+\frac{(r\lambda-\gamma)f}{2a\nu}.
\end{eqnarray}
Substituting \eqref{eq2.23} into \eqref{eq2.22}, we have
\begin{equation}\label{eq2.24}
 \varphi=\frac{\lambda r-\gamma}{2\nu}x^2+rx+\frac{\gamma-\lambda r}{2a^2\nu}b^2+\frac{br}{a}-\frac{\mu}{2a^2}.
\end{equation}
Substituting \eqref{eq2.24} into \eqref{eq2.23}, we have
\begin{equation*}
  a^2(r\lambda-\gamma)x^2+2ra^2\nu x+2rab\nu-b^2r\lambda+b^2\gamma-\mu\nu\equiv 0
\end{equation*}
or
\begin{equation*}
  a\lambda(1-n)x^2+((r+1-2n)a\nu+(n-1)b\lambda)x+(r-1)b\nu\equiv 0.
\end{equation*}
This is impossible.

Now let $r=1$ and $n=d+1$, it follows from \eqref{e7.8.2} that
\begin{equation*}
  \gamma=\left(-\lambda+\frac{2(b\lambda-a\nu)}{ax+b}\right)\varphi'(x)
+\left(\frac{2a(2d+1)(a\nu-b\lambda)}{(ax+b)^2}+\frac{2\lambda a}{ax+b}+\frac{\lambda^2d}{\lambda x+\nu}\right)\varphi(x)+\frac{\mu(\lambda x+\nu)}{(ax+b)^2}.
\end{equation*}
Since $\gamma$ is a constant on $M$,  taking the derivative of the above equation with respect to $x$, it follows that
\begin{eqnarray*}
   &&  -\frac{(\lambda x+2\nu)a-b\lambda}{ax+b}\left\{\varphi''(x)+\left(\frac{\lambda d}{\lambda x+\nu}-\frac{2a(d+1)}{ax+b}\right)\varphi'(x) +\frac{\mu}{(ax+b)^2}\right. \\
   &&  \left.+ \left(\frac{2(2d+1)a^2}{(ax+b)^2}-\frac{2ad\lambda}{(\nu+\lambda x)(ax+b)}-\frac{d\lambda^2}{(\nu+\lambda x)^2}\right)\varphi(x)
 \right\}=0,
\end{eqnarray*}
so
\begin{eqnarray*}
   &&\varphi''(x)+\left(\frac{\lambda d}{\lambda x+\nu}-\frac{2a(d+1)}{ax+b}\right)\varphi'(x) +\frac{\mu}{(ax+b)^2} \\
   && +\left(\frac{2(2d+1)a^2}{(ax+b)^2}-\frac{2ad\lambda}{(\nu+\lambda x)(ax+b)}-\frac{d\lambda^2}{(\nu+\lambda x)^2}\right)\varphi(x)=0.
\end{eqnarray*}
Using
\begin{equation*}
 \frac{p'}{p}-\frac{2an}{f}=\frac{\lambda d}{\lambda x+\nu}-\frac{2a(d+1)}{ax+b},\;\frac{\mu}{f^2}=\frac{\mu}{(ax+b)^2},
\end{equation*}
\begin{equation*}
 \frac{2(2n-1)a^2}{f^2}-\frac{2ap'}{pf}+\left(\frac{p'}{p}\right)'= \frac{2(2d+1)a^2}{(ax+b)^2}-\frac{2ad\lambda}{(\nu+\lambda x)(ax+b)}-\frac{d\lambda^2}{(\nu+\lambda x)^2},
\end{equation*}
it follows from $\gamma$ is constant on $M$ that \eqref{eq7.2.1.2} implies \eqref{eq7.2}.

$(\mathrm{II-4})$ Finally, note that \eqref{eq7.3.1} equivalent to \eqref{eq7.2.1.2} for $\nu a=0$, from \eqref{eq7.3.1} and \eqref{eq7.3.3},  $\widetilde{g}=\frac{1}{f^2}g_F$ is an Einstein metric on $M$ (that is $\mathrm{Ric}_{\widetilde{g}}=\mu \;\widetilde{g}$) if and only if equations \eqref{eq7.3.4} and \eqref{eq7.3.5} hold.

\end{proof}

\begin{Remark}
For $r=1$, under the situation of Theorem \ref{Th:4.1}, let $u=\nu+\lambda x$, $\theta=\frac{1}{2\lambda^2\varphi(x)}\mathrm{d}^cu$. By
\begin{equation*}
  \omega_F=\frac{1}{2}\mathrm{d}\mathrm{d}^c(\nu\phi(z)+F(t))=\nu\omega_{\phi}+\frac{1}{2}\Big(F''(t)\mathrm{d}t\wedge\mathrm{d}^ct+F'(t)\mathrm{d}\mathrm{d}^ct\Big),
\end{equation*}
we have
\begin{eqnarray*}
  \omega_F &=&\nu\omega_{\phi}+\frac{1}{2}\Big(\varphi(x)\frac{\mathrm{d}x}{\varphi(x)}\wedge\frac{\mathrm{d}^cx}{\varphi(x)}+\lambda x\mathrm{d}\mathrm{d}^c\phi\Big)  \\
   &=& (\nu+\lambda x)\omega_{\phi}+\frac{1}{2\varphi(x)}\mathrm{d}x\wedge\mathrm{d}^cx \\
   &=&u\omega_{\phi}+\frac{1}{2\lambda^2\varphi(x)}\mathrm{d}u\wedge\mathrm{d}^cu= u\omega_{\phi}+\mathrm{d}u\wedge\theta.
\end{eqnarray*}
From \cite{Apostolov-Calderbank-Gauduchon-2006}, if $\widetilde{g}$ is a conformally Einstein with conformal factor $f$, then
\begin{equation}\label{eq2.25}
  \left\{\begin{array}{l}
           \mathrm{Ric}_{g_{\phi}}=\gamma\; g_{g_{\phi}},\\\\
           G''(u)-\left(\frac{2(n-1)q}{qu+p}+\frac{n}{u}\right)G'(u)+\frac{2n(n-1)q}{z(qu+p)}G(u)+2\gamma u^{n-2}=0,
         \end{array}
  \right.
\end{equation}
where
\begin{equation*}
  G(u)=2\lambda^2u^{n-1}\varphi(x),ax+b=qz+p, q=\frac{a}{\lambda}, p=\frac{b\lambda-a\nu}{\lambda}.
\end{equation*}

Using the variable $u$ and the function $G$,  \eqref{eq7.3.4} is reduced to
\begin{equation}\label{eq2.26}
  \left\{\begin{array}{l}
           \mathrm{Ric}_{g_{\phi}}=\gamma\; g_{g_{\phi}},\\\\
           G'(u)-\left(\frac{(2n-1)q}{qu+p}+\frac{q}{qu-p}\right)G(u)-\frac{2\mu u^n-2\gamma u^{n-1}(qu+p)^2}{(qu+p)(qu-p)}=0.
         \end{array}
  \right.
\end{equation}

The following proves that if $G$ is a solution of \eqref{eq2.26}, then $G$ must be a solution of \eqref{eq2.25}. 

Let $H(u)=\frac{G(u)}{(qu+p)^{2n-1}(qu-p)}$. Equations \eqref{eq2.25} and \eqref{eq2.26} are equivalent to
 \begin{equation}\label{eq2.27}
   \frac{\mathrm{d}}{\mathrm{d}u}\frac{(qu-p)^2(qu+p)^{2n}H'(u)}{u^n}+2\gamma\frac{(qu+p)(qu-p)}{u^2}=0
 \end{equation}
 and
 \begin{equation}\label{eq2.28}
   H'(u)-\frac{2\mu u^n-2\gamma u^{n-1}(qu+p)^2}{(qu+p)^{2n}(qu-p)^2}=0.
 \end{equation}
 respectively. It is easy to see that the solution  $H$ of the equation  \eqref{eq2.28} must satisfies the equation \eqref{eq2.27}.
\end{Remark}

\begin{Remark}
For $\nu=0,a\neq 0, r>1$, under the situation of Theorem \ref{Th:4.1}, let $\omega_{FS}$ be the Fubini-Study metric on $\mathbb{CP}^{r-1}$, $u=x$, $\theta=\frac{1}{2\varphi(x)}\mathrm{d}^cu$ and $\omega_0=\lambda\omega_{\phi}+\omega_{FS}$. By
\begin{equation*}
  \mathrm{d}t=\frac{\mathrm{d}x}{\varphi(x)}=\frac{\mathrm{d}u}{\varphi(u)},\mathrm{d}\mathrm{d}^ct=\mathrm{d}\mathrm{d}^c(\lambda\phi+\log\|w\|^2)
  =2\lambda\omega_{\phi}+2\omega_{FS}
\end{equation*}
and
\begin{equation*}
  \omega_F=\frac{1}{2}\mathrm{d}\mathrm{d}^c(\nu\phi(z)+F(t))=\nu\omega_{\phi}+\frac{1}{2}\Big(F''(t)\mathrm{d}t\wedge\mathrm{d}^ct+F'(t)\mathrm{d}\mathrm{d}^ct\Big),
\end{equation*}
we get
\begin{eqnarray*}
  \omega_F &=&\nu\omega_{\phi}+\frac{1}{2}\Big(\varphi(x)\frac{\mathrm{d}x}{\varphi(x)}\wedge\frac{\mathrm{d}^cx}{\varphi(x)}+ x(2\lambda\omega_{\phi}+2\omega_{FS})\Big)  \\
   &=& (\nu+\lambda x)\omega_{\phi}+x\omega_{FS}+\frac{1}{2\varphi(x)}\mathrm{d}x\wedge\mathrm{d}^cx \\
   &=& u\omega_{0}+\mathrm{d}u\wedge\theta.
\end{eqnarray*}

From \cite{Apostolov-Calderbank-Gauduchon-2006}, if $\widetilde{g}$ is a conformally Einstein with conformal factor $f=ax+b$, then
\begin{equation}\label{eq2.27.1}
  \left\{\begin{array}{l}
           \rho_{0}=c\, \omega_{0},\\\\
           G''(u)-\left(\frac{2(n-1)q}{qu+p}+\frac{n}{u}\right)G'(u)+\frac{2n(n-1)q}{z(qu+p)}G(u)+2c u^{n-2}=0,
         \end{array}
  \right.
\end{equation}
where
\begin{equation*}
  G(u)=2u^{n-1}\varphi(x),ax+b=qz+p, q=a, p=b.
\end{equation*}
Since $\rho_0=\rho_{\phi}+\rho_{FS}=\rho_{\phi}+r\omega_{FS}$, it follows from $\rho_{0}=c\, \omega_{0}$ that
\begin{equation*}
  c=r,\;\rho_{\phi}=\gamma\,\omega_{\phi}=r\lambda\,\omega_{\phi}.
\end{equation*}

Using the variable $u$,  \eqref{eq7.3.5} is reduced to
\begin{equation*}
  \left\{\begin{array}{l}
         \mathrm{Ric}_{g_{\phi}}=r\lambda\, g_{g_{\phi}},\\\\
             \varphi'(u)+\left(\frac{n-1}{u}-\frac{(2n-1)q}{qu+p}-\frac{ q}{qu-p}\right)\varphi(u)
   -\frac{\mu u-r (qu+p)^2}{(qu+p)(qu-p)}=0.
         \end{array}
  \right.
\end{equation*}
That is
\begin{equation*}
  \frac{\mathrm{d}}{\mathrm{d}u}\frac{2u^{n-1}\varphi(u)}{(qu+p)^{2n-1}(qu-p)}=\frac{2\mu u^n-2r u^{n-1}(qu+p)^{2}}{(qu+p)^{2n}(qu-p)^2}
\end{equation*}
or
\begin{equation*}
  \frac{\mathrm{d}}{\mathrm{d}u}\Big(2u^{n-1}\varphi(u)\Big)-\left(\frac{(2n-1)q}{qu+p}+\frac{q}{qu-p}\right)\Big(2u^{n-1}\varphi(u)\Big)-\frac{2\mu u^n-2r u^{n-1}(qu+p)^2}{(qu+p)(qu-p)}=0.
\end{equation*}
Which implies that $2u^{n-1}\varphi(u) $ is a solution of \eqref{eq2.27.1}.
\end{Remark}

\setcounter{equation}{0}
\section{ The Completeness of conformally K\"ahler  metrics}

\begin{Lemma}\label{Le:8.1}{ Under the situation of Theorem \ref{Th:1.1}, for a given point $z_0\in M_0$, let
\begin{equation*}
  M_{z_0}=\{u\in M:u\in\Pi^{-1}(z_0)\},\;N_0=\{u\in M:h(u,u)=0\}.
\end{equation*}
then $M_{z_0}$  is totally geodesic with respect to $\widetilde{g}$. If $\nu=1$ and $N_0\neq \emptyset$, then  $N_0$ is also  a totally geodesic submanifold of $M$ with respect to the metric $\widetilde{g}$.

 }\end{Lemma}

\begin{proof}[Proof]

Given a point $z_0\in M_0 $, take a sufficiently small holomorphic neighborhood $\Omega$ ($\Omega$ is required to be holomorphic homeomorphic  to a domain  in $\mathbb{C}^d$, the following $\Omega$ is not distinguished from its holomorphic homeomorphism in $\mathbb{C}^d$ and is still denoted as $\Omega$) of $z_0$  in $M$ such that the K\"ahler metric $g_0$ has a K\"ahler potential $\phi$ on $\Omega$ and the holomorphic bundle $L_0$ has a local trivializing holomorphic section $\sigma : \Omega \rightarrow L_0\backslash\{0\}$ over $\Omega$, thus $E = L_0^{\oplus r}$ has holomorphic coordinates $(z,w)\in \Omega\times \mathbb{C}^r$,   so the Hermitian metric $h$ of holomorphic vector bundles $E $ can be expressed as
\begin{equation*}
  h(u,u)=e^{\lambda\phi(z)}\|w\|^2,u=(z,w\sigma(z))\in  E.
\end{equation*}
Therefore, the K\"{a}hler form of the K\"{a}hler  metric $g_F$  is expressed as $ \omega_F=\sqrt{-1}\partial\overline{\partial}\Phi_F$, where
\begin{equation*}
  \Phi_F=\nu\phi(z)+F(t), t=\lambda\phi(z)+\log\rho^2, \rho=\|w\|,z\in \Omega\subset M_0, u=(z,w\sigma(z))\in E.
\end{equation*}

For $U\in \mathcal{U}(r)$ (the unitary group of order $r$) and 1 is not the eigenvalue of unitary matrix $U$ , it is easy to see that $N_0$ is the fixed point set of isometric mapping
\begin{equation*}
  (z,w\sigma(z))\in M\longmapsto (z,wU\sigma(z))\in M,
\end{equation*}
so $N_0$ is a totally geodesic submanifold of $(M,\widetilde{g})$.

Let $D$ be the Chern connection for the K\"ahler manifold $(M,g)$, $\widetilde{D}$ be the Levi-Civita connection for the Riemannian manifold $(M,\widetilde{g})$.

Setting
\begin{equation*}
 D\left(
                 \begin{array}{c}
                   \frac{\partial}{\partial Z^t} \\\\
                   \frac{\partial}{\partial \overline{Z}^t} \\
                 \end{array}
               \right)=\Omega_{\mathbb{C}}\otimes\left(
                 \begin{array}{c}
                   \frac{\partial}{\partial Z^t} \\\\
                   \frac{\partial}{\partial \overline{Z}^t} \\
                 \end{array}
               \right)
  ,\;\widetilde{D}\left(
                 \begin{array}{c}
                   \frac{\partial}{\partial Z^t} \\\\
                   \frac{\partial}{\partial \overline{Z}^t} \\
                 \end{array}
               \right)= \widetilde{\Omega}_{\mathbb{C}}\otimes\left(
                 \begin{array}{c}
                   \frac{\partial}{\partial Z^t} \\\\
                   \frac{\partial}{\partial \overline{Z}^t} \\
                 \end{array}
               \right),
\end{equation*}
where
\begin{equation*}
Z=(z,w),\;  \Omega_{\mathbb{C}}=\left(
           \begin{array}{cc}
              (\partial\; T)T^{-1} & 0 \\
             0 &  (\overline{\partial}\; \overline{T})\overline{T^{-1}} \\
           \end{array}
         \right),
\end{equation*}

\begin{eqnarray}
\nonumber \widetilde{\Omega}_{\mathbb{C}}  &=& \left(
                               \begin{array}{cc}
                                 (\partial\; T)T^{-1} & 0 \\
                                 0 & (\overline{\partial}\; \overline{T})\overline{T^{-1}} \\
                               \end{array}
                             \right)
  -\frac{\mathrm{d}f}{f}I_{2n}-\frac{1}{f}\left(
                                                                   \begin{array}{c}
                                                                     \frac{\partial f}{\partial Z^t} \\
                                                                     \frac{\partial f}{\partial \overline{Z}^t} \\
                                                                   \end{array}
                                                                 \right)\left(
                                                                          \begin{array}{cc}
                                                                            \mathrm{d}Z& \mathrm{d}\overline{Z} \\
                                                                          \end{array}
                                                                        \right)\\
\label{eq8.1}   & &+\frac{1}{f}\left(
                                                                                             \begin{array}{cc}
                                                                                               0 & T \\
                                                                                               \overline{T} & 0 \\
                                                                                             \end{array}
                                                                                           \right)\left(
                                                                                                    \begin{array}{c}
                                                                                                      \mathrm{d}Z^t \\
                                                                                                      \mathrm{d}\overline{Z}^t \\
                                                                                                    \end{array}
                                                                                                  \right)\left(
                                                                                                           \begin{array}{cc}
                                                                                                             \frac{\partial f}{\partial Z} & \frac{\partial f}{\partial \overline{Z}} \\
                                                                                                           \end{array}
                                                                                                         \right)\left(
                                                                                                                  \begin{array}{cc}
                                                                                                                    0 & \overline{T^{-1}} \\
                                                                                                                    T^{-1} & 0 \\
                                                                                                                  \end{array}
                                                                                                                \right),
\end{eqnarray}
 the definition of $T$ can be referred to Lemma \ref{Le:4.1}, $I_{2n}$ is the identity matrix of order $2n$,  and $\partial$ and $\overline{\partial}$ can be written as
$$\partial=\mathrm{d}z \frac{\partial}{\partial z^t}+\mathrm{d}w \frac{\partial}{\partial w^t},\;
\overline{\partial}=\mathrm{d}\overline{z} \frac{\partial}{\partial \overline{z}^t}+\mathrm{d}\overline{w} \frac{\partial}{\partial \overline{w}^t}.$$
For the expression of $\widetilde{\Omega}_{\mathbb{C}} $, see \cite{Feng-2020-2}.

There exists local coordinates $(z_1, \cdots, z_{d})$ on a neighborhood of the point $z_0$ such that
\begin{equation*}
  \phi(z_0)=0,\;\frac{\partial\phi}{\partial z}(z_0)=\frac{\partial\phi}{\partial \overline{z}}(z_0)=0,\;\frac{\partial^2\phi}{\partial z^t\partial\overline{z}}(z_0)=I_d\;(\text{the identity matrix of order}\; d)
\end{equation*}
for the K\"{a}hler potential $\phi$ on the domain $\Omega$. Then
\begin{equation*}
  \frac{\partial f}{\partial Z}=\left(
                                       \begin{array}{cc}
                                        f' \lambda\frac{\partial\phi}{\partial z} &  \frac{\partial f}{\partial w} \\
                                       \end{array}
                                     \right)=\left(
                                               \begin{array}{cc}
                                                 0 &  \frac{\partial f}{\partial w} \\
                                               \end{array}
                                             \right),\;\frac{\partial f}{\partial \overline{Z}}=\left(
                                       \begin{array}{cc}
                                        f' \lambda\frac{\partial\phi}{\partial \overline{z}} &  \frac{\partial f}{\partial \overline{w}} \\
                                       \end{array}
                                     \right)=\left(
                                               \begin{array}{cc}
                                                 0 &  \frac{\partial f}{\partial \overline{w}} \\
                                               \end{array}
                                             \right),
\end{equation*}
\begin{equation*}
  T=\left(
                          \begin{array}{cc}
                            (\nu+\lambda F')I_d & 0 \\
                            0 & T_{22} \\
                          \end{array}
                        \right),\; T^{-1}=\left(
           \begin{array}{cc}
             \frac{1}{\nu+\lambda F'}I_d & 0 \\
             0 & T_{22}^{-1}\\
           \end{array}
         \right),
\end{equation*}
\begin{equation*}
 \partial_wT= \left(
  \begin{array}{cc}
    \lambda F''\partial_wtI_d & 0 \\
    0 & \partial_w T_{22} \\
  \end{array}
\right),\; \overline{\partial}_wT= \left(
  \begin{array}{cc}
    \lambda F''\overline{\partial}_wtI_d & 0 \\
    0 & \overline{\partial}_w \overline{T_{22}} \\
  \end{array}
\right)
\end{equation*}
and
\begin{equation*}
  (\partial_wT)T^{-1}=\left(
           \begin{array}{cc}
             \frac{\lambda F''\partial_wt}{\nu+\lambda F'}I_d & 0 \\
             0 &  (\partial_wT_{22})T_{22}^{-1} \\
           \end{array}
         \right),\; (\overline{\partial}_w\overline{T})\overline{T^{-1}}=\left(
           \begin{array}{cc}
             \frac{\lambda F''\overline{\partial}_wt}{\nu+\lambda F'}I_d & 0 \\
             0 &  (\overline{\partial}_w\overline{T_{22}})\overline{T_{22}^{-1}} \\
           \end{array}
         \right)
\end{equation*}
at $z=z_0$, where $\partial_w=\mathrm{d}w\frac{\partial}{\partial w^t}$, $\overline{\partial}_w=\mathrm{d}\overline{w}\frac{\partial}{\partial \overline{w}^t}$, $\partial_wt=\frac{1}{\rho^2}\mathrm{d}w\overline{w}^t$ and $\overline{\partial}_wt=\frac{1}{\rho^2}w\mathrm{d}\overline{w}^t$.

When $z=z_0$, $\mathrm{d}z=0$ and $\mathrm{d}\overline{z}=0$, by \eqref{eq8.1} we have
\begin{equation*}
  \widetilde{D}\left(
                 \begin{array}{c}
                   \frac{\partial}{\partial w^t} \\\\
                   \frac{\partial}{\partial \overline{w}^t} \\
                 \end{array}
               \right)= \widetilde{\Omega}_{\mathbb{C}}^w\otimes\left(
                 \begin{array}{c}
                   \frac{\partial}{\partial w^t} \\\\
                   \frac{\partial}{\partial \overline{w}^t} \\
                 \end{array}
               \right),
\end{equation*}
where
\begin{eqnarray}
\nonumber \widetilde{\Omega}_{\mathbb{C}}^w  &=& \left(
                               \begin{array}{cc}
                                (\partial_wT_{22})T_{22}^{-1} & 0 \\
                                 0 &  (\overline{\partial}_w\overline{T_{22}})\overline{T_{22}^{-1}} \\
                               \end{array}
                             \right)
  -\frac{\mathrm{d}f}{f}I_{2r}-\frac{1}{f}\left(
                                                                   \begin{array}{c}
                                                                     \frac{\partial f}{\partial w^t} \\
                                                                     \frac{\partial f}{\partial \overline{w}^t} \\
                                                                   \end{array}
                                                                 \right)\left(
                                                                          \begin{array}{cc}
                                                                            \mathrm{d}w & \mathrm{d}\overline{w} \\
                                                                          \end{array}
                                                                        \right)\\
 \label{eq8.2}  & &+\frac{1}{f}\left(
                                                                                             \begin{array}{cc}
                                                                                               0 & T_{22} \\
                                                                                               \overline{T_{22}} & 0 \\
                                                                                             \end{array}
                                                                                           \right)\left(
                                                                                                    \begin{array}{c}
                                                                                                      \mathrm{d}w^t \\
                                                                                                      \mathrm{d}\overline{w}^t \\
                                                                                                    \end{array}
                                                                                                  \right)\left(
                                                                                                           \begin{array}{cc}
                                                                                                             \frac{\partial f}{\partial w} & \frac{\partial f}{\partial \overline{w}} \\
                                                                                                           \end{array}
                                                                                                         \right)\left(
                                                                                                                  \begin{array}{cc}
                                                                                                                    0 & \overline{T_{22}^{-1}} \\
                                                                                                                    T_{22}^{-1} & 0 \\
                                                                                                                  \end{array}
                                                                                                                \right).
\end{eqnarray}

On the other hand, let $\widehat{D}$ be the Levi-Civita connection of the induced metric of $\widetilde{g}$ on $M_{z_0}$, then
\begin{equation*}
  \widehat{D}\left(
                 \begin{array}{c}
                   \frac{\partial}{\partial w^t} \\\\
                   \frac{\partial}{\partial \overline{w}^t} \\
                 \end{array}
               \right)= \widehat{\Omega}_{\mathbb{C}}\otimes\left(
                 \begin{array}{c}
                   \frac{\partial}{\partial w^t} \\\\
                   \frac{\partial}{\partial \overline{w}^t} \\
                 \end{array}
               \right).
\end{equation*}
By Lemma \ref{Le:4.1}, we get $\widehat{\Omega}_{\mathbb{C}}=\widetilde{\Omega}_{\mathbb{C}}^w$, this implies that
\begin{equation*}
  \widetilde{D}_{e_j}e_i=\widehat{D}_{e_i}e_j
\end{equation*}
at $z=z_0$ for all $e_i,e_j\in\{\frac{\partial}{\partial w_1},\ldots,\frac{\partial}{\partial w_{r}},\frac{\partial}{\partial \overline{w_1}},\ldots,\frac{\partial}{\partial \overline{w_{r}}}\}$ , thus  $M_{z_0}$ is a totally geodesic submanifold of $(M,\widetilde{g})$.

\end{proof}

\begin{Theorem}\label{Th:3.2}
Under the assumption of  Theorem \ref{Th:1.1}, let
\begin{equation*}
 l =l(t_0)+\int_{x(t_0)}^{x(t)}\frac{1}{(ax+b)\sqrt{\varphi(x)}}\mathrm{d}x, x_0<x(t_0)<x_1.
\end{equation*}

If $  \lim_{x\rightarrow x_0^{+}}l=-\infty$ for $M$ is any one of
$$E^{*}, \mathbb{B}_{L_0}^{r*}, E-\{u\in E:h(u,u)\leq 1\},\mathbb{P}(E \oplus 1)-M_0, \mathbb{P}(E \oplus 1)-\{u\in E:h(u,u)\leq 1\};$$
and $ \lim_{x\rightarrow x_1^{-}}l=+\infty$ for $M$ is any one of $E, E^{*}, \mathbb{B}_{L_0}^{r}$, $ E-\{u\in E:h(u,u)\leq 1\}$ and $\mathbb{B}_{L_0}^{r*}$. Then the Riemannian metric $\widetilde{g}$ is complete on $M$.
\end{Theorem}

\begin{Remark}
By Lemma 3.2 of \cite{Feng-2020-2} and Theorem \ref{Th:3.2},  if the metric $g_0$ on $M_0$ is complete and  fiber metrics of the metric $\widetilde{g}$ on $M$ are complete, then the metric $\widetilde{g}$ must be complete on $M$.
\end{Remark}

\begin{proof}[Proof of Theorem \ref{Th:3.2}]
The proof is divided into two parts.

$(\mathrm{i})$  Firstly, we calculate the square of the norm of the differential of a function  on $M$ with respect to the metric $\widetilde{g}$.

Let  $H=\psi(z)+p(t)$, $z\in M_0$, $t=\log h(u,u)$ and $u\in E$.

 According to the proof of Lemma \ref{Le:8.1}, the K\"{a}hler form of the K\"{a}hler  metric $g_F$  can be expressed locally as $ \omega_F=\sqrt{-1}\partial\overline{\partial}\Phi_F$, where
\begin{equation*}
  \Phi_F=\nu\phi(z)+F(t), t=\lambda\phi(z)+\log\rho^2, \rho=\|w\|,z\in \Omega\subset M_0, u=(z,w\sigma(z))\in E.
\end{equation*}

Since
\begin{equation*}
  \frac{\partial H}{\partial z}= \frac{\partial \psi}{\partial z}+\lambda p'\frac{\partial \phi}{\partial z}
\end{equation*}
and
\begin{equation*}
   \frac{\partial H}{\partial w}=p' \frac{\overline{w}}{\rho^2},
\end{equation*}
thus
\begin{equation*}
 \frac{\partial H}{\partial Z}=\left(  \frac{\partial \psi}{\partial z}+\lambda p' \frac{\partial \phi}{\partial z},p' \frac{\overline{w}}{\rho^2}\right),
 \|\mathrm{d}H\|_{\widetilde{g}}^2=2\frac{\partial H}{\partial \overline{Z}}\left(\frac{1}{f^2}T\right)^{-1}\frac{\partial H}{\partial Z^t},
\end{equation*}
where
\begin{equation*}
  Z=(z,w), T=\frac{\partial^2\Phi_F}{\partial Z^t\partial \overline{Z}}.
\end{equation*}
By Lemma \ref{Le:4.1}, we have
\begin{eqnarray*}
 \frac{\partial H}{\partial \overline{z}}T^{-1}_{11}\frac{\partial H}{\partial z^t}  &=&  \frac{1}{\nu+\lambda F'}
 \left(  \frac{\partial \psi}{\partial \overline{z}}+\lambda p' \frac{\partial \phi}{\partial \overline{z}}\right)
 \left(\frac{\partial^2\phi}{\partial z^t\partial \bar z}\right)^{-1}
 \left(  \frac{\partial \psi}{\partial z^t}+\lambda p' \frac{\partial \phi}{\partial z^t}\right)\\
   &=& \frac{1}{2(\nu+\lambda F')} \| \mathrm{d}\psi+\lambda p' \mathrm{d}\phi\|^2_{g_0},
\end{eqnarray*}
\begin{eqnarray*}
 \frac{\partial H}{\partial \overline{z}}T^{-1}_{12}\frac{\partial H}{\partial w^t}  &=&  -\frac{\lambda}{\nu+\lambda F'}
\left(  \frac{\partial \psi}{\partial \overline{z}}+\lambda p' \frac{\partial \phi}{\partial \overline{z}}\right)
 \left(\frac{\partial^2\phi}{\partial z^t\partial \bar z}\right)^{-1}
 \frac{\partial\phi}{\partial z^t}wp' \frac{\overline{w}^t}{\rho^2}\\
   &=&- \frac{\lambda p' }{(\nu+\lambda F')}
  \left(  \frac{\partial \psi}{\partial \overline{z}}+\lambda p' \frac{\partial \phi}{\partial \overline{z}}\right)
   \left(\frac{\partial^2\phi}{\partial z^t\partial \bar z}\right)^{-1}
   \frac{\partial\phi}{\partial z^t},
\end{eqnarray*}
\begin{equation*}
   \frac{\partial H}{\partial \overline{w}}T^{-1}_{21}\frac{\partial H}{\partial z^t}=- \frac{\lambda p' }{(\nu+\lambda F')}
  \frac{\partial \phi}{\partial \overline{z}}\left(\frac{\partial^2\phi}{\partial z^t\partial \bar z}\right)^{-1}
   \left(  \frac{\partial \psi}{\partial z^t}+\lambda p' \frac{\partial \phi}{\partial z^t}\right),
\end{equation*}
\begin{eqnarray*}
 \frac{\partial H}{\partial \overline{w}}T^{-1}_{22}\frac{\partial H}{\partial w^t}  &=&\frac{(p' )^2}{\rho^4} w \left(\frac{r^2}{F'}I_{r}+\left(\frac{1}{F''}-\frac{1}{F'}\right)\overline{w}^tw+\frac{\lambda^2}{\nu+\lambda F'}\frac{\partial\phi}{\partial \bar{z}}\left(\frac{\partial^2\phi}{\partial z^t\partial \bar z}\right)^{-1}\frac{\partial\phi}{\partial z^t}\overline{w}^tw\right)\overline{w}^t\\
  &=&(p' )^2 \left(\frac{1}{F'}+\left(\frac{1}{F''}-\frac{1}{F'}\right)+\frac{\lambda^2}{\nu+\lambda F'}\frac{\partial\phi}{\partial \bar{z}}\left(\frac{\partial^2\phi}{\partial z^t\partial \bar z}\right)^{-1}\frac{\partial\phi}{\partial z^t}\right)\\
   &=&(p' )^2\left(\frac{1}{F''}+\frac{\lambda^2}{2(\nu+\lambda F')}\|\mathrm{d}\phi\|^2_{g_0}\right).
\end{eqnarray*}
So
\begin{eqnarray*}
   & &\|\mathrm{d}H\|_{\widetilde{g}}^2 =2f^2\frac{\partial H}{\partial \overline{Z}}T^{-1}\frac{\partial H}{\partial Z^t}\\  
   &=& 2f^2\left( \frac{1}{2(\nu+\lambda F')} \| \mathrm{d}\psi+\lambda p' \mathrm{d}\phi\|^2_{g_0}
   - \frac{\lambda p' }{(\nu+\lambda F')} ( \mathrm{d}\psi+\lambda p' \mathrm{d}\phi,\mathrm{d}\phi)_{g_0}\right.\\
   &&\left.+(p' )^2\Big(\frac{1}{F''}+\frac{\lambda^2}{2(\nu+\lambda F')}\|\mathrm{d}\phi\|^2_{g_0}\Big)\right)  \\
   &=&  2f^2\left( \frac{ 1}{2(\nu+\lambda F')} \|\mathrm{d}\psi\|^2_{g_0}+\frac{(p' )^2}{F''}\right)\\
   &=&  \frac{ f^2}{\nu+\lambda F'} \|\mathrm{d}\psi\|^2_{g_0}+2\frac{(fp' )^2}{F''}.
\end{eqnarray*}

$(\mathrm{ii})$ Secondly, we prove that the metric $\widetilde{g}$ is complete.

Since the metric $g_0$ is complete, it follows that there is a  smooth proper function $\psi(z)$  on $M_0$ such that $\psi(z)>1$ (otherwise,  replace $\psi$
with the function  $c\log(\psi^2+4)+1$) and $\|\psi\|^2_{g_0}\leq 1$. For a sufficiently large positive integer $m$, let $p=(e^l+e^{-l})^m$, $H=\log(p(t)+\psi(z))$, then $H$ is a smooth proper function on $M$, and
\begin{equation*}
 \|\mathrm{d}H\|_{\widetilde{g}}^2= \frac{\|\mathrm{d}(p+\psi)\|_{\widetilde{g}}^2 }{(p+\psi)^2}= \frac{1}{(p+\psi)^2}\frac{f^2}{\nu+\lambda F'} \|\mathrm{d}\psi\|^2_{g_0}
 + \frac{2}{(p+\psi)^2}\frac{(fp')^2}{F''}.
\end{equation*}

From
\begin{equation*}
l=l(t_0)+\int_{x(t_0)}^{x(t)}\frac{1}{(ax+b)\sqrt{\varphi(x)}}\mathrm{d}x=l(t_0)+\int_{t_0}^t\frac{\sqrt{F''(t)}}{f(t)}\mathrm{d}t,
\end{equation*}
we get
\begin{equation*}
  f(t)l' (t)=\sqrt{F''(t)}.
\end{equation*}
Then
\begin{equation*}
 \frac{1}{(p+\psi)^2}\frac{(fp')^2}{F''}= \frac{p^2}{(p+\psi)^2}\left(\frac{p'}{p}\right)^2\frac{f^2}{F''}\leq \left(\frac{p'}{p}\right)^2\frac{f^2}{F''}= m^2\left(\frac{e^l-e^{-l}}{e^l+e^{-l}}\right)^2\frac{(fl' )^2}{F''}\leq m^2.
\end{equation*}

Let
\begin{equation*}
  x_0=\inf\{F'(t)\},\;x_1=\sup\{F'(t)\},\;X_0=\nu+\lambda x_0,\;X_1=\nu+\lambda x_1.
\end{equation*}
Then $x_0<x_1$, $X_0\geq 0$ and $X_1\geq 0$.

If $0<X_0,X_1<+\infty$, then $\frac{f(x)^2}{\nu+\lambda x}$ is bounded on $[x_0,x_1]$, so by
\begin{equation*}
  \frac{1}{(p+\psi)^2}\frac{f^2}{\nu+\lambda F'} \|\mathrm{d}\psi\|^2_{g_0}\leq \frac{f^2}{p^2(\nu+\lambda F')}=\frac{f^2}{p^2(\nu+\lambda x)}\leq \frac{f^2}{\nu+\lambda x},
\end{equation*}
we get $\frac{1}{(p+\psi)^2}\frac{f^2}{\nu+\lambda F'}  \|\mathrm{d}\psi\|^2_{g_0}$ is bounded on $M$. Therefore, the metric $\widetilde{g}$ is complete on  $M$.

If $f=\nu+\lambda x$ and $0\leq X_0, X_1<+\infty$, we also obtain that the metric $\widetilde{g}$ is complete on  $M$.

For the case of $X_0=0$ or $X_1=0$ or $X_0=+\infty$ or $X_1=+\infty$.  Without losing generality, we assume that $X_0=0$ or $+\infty$.

When $X_0=0$, let $u=\frac{1}{x-x_0}$ and $l_0=l(t_0)$. By
\begin{equation*}
  \int_{x(t_0)}^{x_0}\frac{1}{(ax+b)\sqrt{\varphi(x)}}\mathrm{d}x=-\infty,
\end{equation*}
 there exists a positive rational number $k\geq 1$ and  a positive real number $\delta$ such that $(ax+b)\sqrt{\varphi(x)}=(x-x_0)^k\psi(x)$, $\psi(x_0)>0$, and the Taylor expansion of a function $\frac{1}{\psi(x)}$ at $x_0$:
 \begin{equation*}
   \frac{1}{\psi(x)}=\sum_{j=0}^{+\infty}A_j(x-x_0)^j
 \end{equation*}
 converges uniformly on the interval $[x_0,x_0+\delta]$. So a series
\begin{equation*}
  \frac{1}{(ax+b)\sqrt{\varphi(x)}}=\sum_{j=0}^{[k]-1}\frac{A_j}{(x-x_0)^{k-j}}+\sum_{j=[k]}^{\infty}A_j(x-x_0)^{j-k},
\end{equation*}
converges uniformly on the closed subinterval of the interval $(x_0,x_0+\delta)$, where $[k]$  is the largest integer not exceeding $k$.

Using
$$\nu+\lambda x=\lambda (x-x_0)\;(\mathrm{by}\; \nu+\lambda x_0=0 ),$$
\begin{eqnarray*}
 l  &=&l_0+\int_{x({t_0})}^{x}\frac{1}{(au+b)\sqrt{\varphi(u)}}\mathrm{d}u\;\left(x(t_0)\in (x_0,x_0+\delta),x\in (x_0,x_0+\delta)\right)  \\
   &=& C+\sum_{j=0}^{[k]-2}\frac{A_j}{(j+1-k)(x-x_0)^{k-j-1}} \\
   & & +A_{[k]-1}\times\left\{\begin{array}{l}
                                                                                                                                           \log(x-x_0)\;(\mathrm{for}\;k=[k]) \\\\
                                                                                                                                           \frac{1}{([k]-k)(x-x_0)^{k-[k]}}\;(\mathrm{for} \;k>[k])
                                                                                                                                         \end{array}
  \right.+o((x-x_0)^{[k]+1-k})
\end{eqnarray*}
and
\begin{equation*}
  e^{-l}=e^{-C+o((x-x_0)^{[k]+1-k}) }e^{\sum_{j=0}^{[k]-2}\frac{A_j}{(k-j-1)(x-x_0)^{k-j-1}}}\times\left\{\begin{array}{l}
                                                                                    \frac{1}{(x-x_0)^{A_{[k]-1}}}\;(\mathrm{for}\;k=[k]) \\\\
                                                                                    e^{\frac{A_{[k]-1}}{k-[k]}\frac{1}{(x-x_0)^{k-[k]}}}\;(\mathrm{for}\;k>[k])
                                                                                  \end{array}
  \right.,
\end{equation*}
 it follows that
\begin{equation*}
   e^{-l}\sim e^{-C}e^{\sum_{j=0}^{[k]-2}\frac{A_ju^{k-j-1}}{k-j-1}}\times\left\{\begin{array}{l}
                                                                                    u^{A_{[k]-1}}\;(\mathrm{for}\;k=[k]) \\\\
                                                                                    e^{\frac{A_{[k]-1}}{k-[k]}u^{k-[k]}}\;(\mathrm{for}\;k>[k])
                                                                                  \end{array}
  \right.\;\;(\text{as}\;u\rightarrow +\infty)
\end{equation*}
and
\begin{equation*}
  \frac{1}{p^2(\nu+\lambda x)} \sim \frac{u}{\lambda e^{-2ml}} \sim \frac{e^{2mC}}{\lambda} \times\left\{ \begin{array}{l}
                                      u^{1-2mA_0}\;(\text{for}\;  k=1) \\\\
                                        \frac{u^{1-2mA_{[k]-1}}}{e^{2m\sum_{j=0}^{[k]-2}\frac{A_ju^{k-j-1}}{k-j-1}}}\;(\text{for}\;k=[k]>1)\\\\
                                          \frac{u}{e^{2m\sum_{j=0}^{[k]-1}\frac{A_ju^{k-j-1}}{k-j-1}}}\;(\text{for}\;k>[k]\geq 1)
                                      \end{array}
    \right.
\end{equation*}
as $\ u\rightarrow+\infty$, here $C$ is a constant that depends on $x_0$, $l_0$ and $x(t_0)$.
By $f>0$ and $\varphi>0$ on $M$, we have $A_0>0$, so there exists a positive integer $m$ such that $mA_0>1$.
Thus
\begin{equation*}
  \frac{1}{(p+\psi)^2}\frac{f^2}{\nu+\lambda F'} \|\mathrm{d}\psi\|^2_{g_0}\leq \frac{f^2}{p^2(\nu+\lambda F')}\rightarrow 0\;\mathrm{as}\;x\rightarrow x_0^{+}.
\end{equation*}

When $X_0=+\infty$, namely $x_0=-\infty$ and $\lambda<0$. If $f(x)=ax+b\equiv b$, then
\begin{equation*}
  \frac{1}{(p+\psi)^2}\frac{f^2}{\nu+\lambda F'} \|\mathrm{d}\psi\|^2_{g_0}\leq \frac{f^2}{p^2(\nu+\lambda x)}\rightarrow 0\;\mathrm{as}\;x\rightarrow -\infty.
\end{equation*}
If $f(x)=ax+b$ and $a<0$, then by $\lim_{x\rightarrow x_0^{+}}l=-\infty$, $f>0$ and $\varphi>0$, there are constants $c$ and $C$ such that
\begin{equation*}
  \frac{1}{(ax+b)\sqrt{\varphi(x)}}=\frac{1}{-x}\sum_{j=0}^{+\infty}\frac{A_j}{x^j} \;\;(x\in (-\infty,c),A_0>0),
\end{equation*}
\begin{eqnarray*}
 l  &=&l_0+\int_{x(t_0)}^{x}\frac{1}{(au+b)\sqrt{\varphi(u)}}\mathrm{d}u  \\
   &=& C-A_0\log|x|+\sum_{j=1}^{+\infty}\frac{A_j}{jx^j} \\
   &= & C-A_0\log|x|+o(\frac{1}{x})
\end{eqnarray*}
and
\begin{equation*}
   \frac{f^2}{p^2(\nu+\lambda x)}\sim \frac{a^2x}{\lambda e^{-2ml}}\sim \frac{a^2e^{2mC}}{\lambda}\frac{x}{|x|^{2mA_0}}
\end{equation*}
as $\ x\rightarrow-\infty$. Since $A_0>0$, there exists a positive integer $m$ such that $mA_0>1$, so
\begin{equation*}
  \frac{1}{(p+\psi)^2}\frac{f^2}{\nu+\lambda F'} \|\mathrm{d}\psi\|^2_{g_0}\leq \frac{f^2}{p^2(\nu+\lambda x)}\rightarrow 0\;\mathrm{as}\;x\rightarrow -\infty.
\end{equation*}

From the above conclusions, we have that
\begin{equation*}
 \|\mathrm{d}H\|_{\widetilde{g}}^2= \frac{\|\mathrm{d}(p+\psi)\|_g^2 }{(p+\psi)^2}= \frac{1}{(p+\psi)^2}\frac{f^2}{\nu+\lambda F'} \|\mathrm{d}\psi\|^2_{g_0}
 + \frac{1}{(p+\psi)^2}\frac{(fp')^2}{F''}.
\end{equation*}
 is bounded on $M$, so by \cite{gordon,gordon2}  the metric $\widetilde{g}$ is complete on $M$.
\end{proof}

\setcounter{equation}{0}
\section{Proof of Theorem \ref{Th:1.1} }

\begin{proof}[Proof of Theorem \ref{Th:1.1}]

From Lemma 4.1 of \cite{Fu-Yau-Zhou}, the metric $g_F$ can be extended across $M=\mathbb{P}(E \oplus 1)-E$ if and only if $\varphi(x_1)=0$ and $\varphi'(x_1)=-1$.

Since the metric $g_0$ is complete on $M_0$, using Lemma \ref{Le:8.1} and Theorem \ref{Th:3.2}, it follows that the metric $\widetilde{g}$ is complete on $M$ if and only if its fiber metrics are complete.

From the proof of Lemma \ref{Le:8.1}, the K\"{a}hler form of the K\"{a}hler  metric $g_F$  can be expressed locally as $ \omega_F=\sqrt{-1}\partial\overline{\partial}\Phi_F$, where
\begin{equation*}
  \Phi_F=\nu\phi(z)+F(t), t=\lambda\phi(z)+\log\rho^2, \rho=\|w\|,z\in \Omega\subset M_0, u=(z,w\sigma(z))\in E.
\end{equation*}
Therefore, the proof of Theorem \ref{Th:1.1} is converted to the proof of Theorem \ref{Th:4.2} below.

\end{proof}

\begin{Theorem}\label{Th:4.2}{

Let $g_0$ be a K\"{a}hler  metric  on a domain $\Omega\subset\mathbb{C}^d$ associated with the K\"{a}hler form $\omega_{\phi}=\sqrt{-1}\partial\overline{\partial}\phi$, namely $g_0(X,Y)=\omega_{\phi}(X,JY)$,  where $\phi$ is globally defined  the K\"{a}hler potential on the domain $\Omega\subset \mathbb{C}^{d}$, and $J$ is the canonical complex structure on the complex Euclidean space. Suppose $M$ represents any of the following sets

$$\Omega\times\mathbb{C}^{r},\;\left\{(z,w)\in \Omega\times\mathbb{C}^{r}:0< e^{\lambda\phi(z)}\|w\|^2\right\},\;\left\{(z,w)\in \Omega\times\mathbb{C}^{r}: e^{\lambda\phi(z)}\|w\|^2<1\right\}\;$$
and
$$ \left\{(z,w)\in \Omega\times\mathbb{C}^{r}: 0<e^{\lambda\phi(z)}\|w\|^2<1\right\},\;\left\{(z,w)\in \Omega\times\mathbb{C}^{r}: 1<e^{\lambda\phi(z)}\|w\|^2\right\}.$$

Set $g_F$ is a K\"{a}hler  metric  on the domain $M$ associated with the K\"{a}hler form $\omega_F=\sqrt{-1}\partial\overline{\partial}\Phi_F$, here
$$\Phi_F(z,w)=\nu\phi(z)+F(t),\;t=\lambda\phi(z)+\log\|w\|^2,\;\nu=0,1,\;\lambda\neq 0.$$

Let
\begin{equation*}
n=d+r,\; x=F'(t), \;\varphi(x)=F''(t),\;f=ax+b>0,\;a\neq 0,\;  x_0=\inf\{F'(t)\},\;x_1=\sup\{F'(t)\}.
\end{equation*}

The following conclusions hold.

$(\mathrm{I})$  For $\nu=1$ and $b=1$,  $\widetilde{g}=\frac{1}{f^2}g_F$ is an Einstein  metric (namely $\mathrm{Ric}_{\widetilde{g}}=\mu \;\widetilde{g}$) with the complete fibre metrics on the domain $M$ if and only if
\begin{equation}\label{ne4.1.1}
 \left\{ \begin{array}{l}
  \mathrm{Ric}_{g_0}=(\lambda+\mu-2a)\;g_0,\\\\
  \varphi(x)=\frac{(1+ax)^{2d+1}(a\lambda x+2a-\lambda)}{(1+\lambda x)^{d}}\int_0^x\frac{(1+\lambda u)^{d}\left(\mu(1+\lambda u)-(\lambda+\mu-2a)(1+au)^2\right)}{(1+au)^{2d+2}(a\lambda u+2a-\lambda)^2}\mathrm{d}u,0\leq x<-\frac{1}{a},\\\\
 \mu\leq 0,a\leq \lambda,a<0, r=1,
 \end{array}\right.
\end{equation}
here
\begin{equation*}
  M=\left\{\begin{array}{l}
             \left\{(z,w)\in \Omega\times\mathbb{C}^{r}: e^{\lambda\phi(z)}\|w\|^2<1\right\} \mathrm{for}\; \mu<0,\\\\
             \Omega\times\mathbb{C}^{r}\;\mathrm{for}\;\mu=0.
           \end{array}
  \right.
\end{equation*}

$(\mathrm{II})$

$(\mathrm{II-1})$  $\widetilde{g}=\frac{1}{f^2}g_F$ is an Einstein  metric (namely $\mathrm{Ric}_{\widetilde{g}}=\mu \;\widetilde{g}$) with the complete fibre metrics on the domain $M=\left\{(z,w)\in \Omega\times\mathbb{C}^{r}: 0<e^{\lambda\phi(z)}\|w\|^2<1\right\}$ if and only if
\begin{equation}\label{ne4.1.2.1}
 \left\{ \begin{array}{l}
  \mathrm{Ric}_{g_0}=\mu\frac{\nu+\lambda x_0}{(b+ax_0)^2}\;g_0,\\\\
  \varphi(x)=\mu\frac{(b+ax)^{2d+1}(a\lambda x+2a\nu-b\lambda)}{(\nu+\lambda x)^{d}}\int_{x_0}^x\frac{(\nu+\lambda u)^{d}\left((\nu+\lambda u)-\frac{\nu+\lambda x_0}{(b+ax_0)^2}(b+au)^2\right)}{(b+au)^{2d+2}(a\lambda u+2a\nu-b\lambda)^2}\mathrm{d}u,\;x_0< x<x_1,\\\\
 r=1,\mu<0,a<0,
 \end{array}\right.
\end{equation}
where
$$\left\{\begin{array}{l}
    -\infty\leq x_0<x_1=-\frac{b}{a}\; \mathrm{for}\; \lambda<0\;\mathrm{and}\;-\frac{b}{a}\leq -\frac{\nu}{\lambda}, \\\\
    -\frac{\nu}{\lambda}\leq x_0<x_1=-\frac{b}{a}\; \mathrm{for}\; \lambda>0\;\mathrm{and}\; -\frac{\nu}{\lambda}<-\frac{b}{a}.
  \end{array}\right.
$$

$(\mathrm{II-2})$  $\widetilde{g}=\frac{1}{f^2}g_F$ is an Einstein  metric (namely $\mathrm{Ric}_{\widetilde{g}}=\mu \;\widetilde{g}$) with the complete fibre metrics on the domain $M=\left\{(z,w)\in \Omega\times\mathbb{C}^{r}: 1<e^{\lambda\phi(z)}\|w\|^2\right\}$ if and only if
\begin{equation}\label{ne4.1.2.1}
 \left\{ \begin{array}{l}
  \mathrm{Ric}_{g_0}=\mu\frac{\nu+\lambda x_1}{(b+ax_1)^2}\;g_0,\\\\
  \varphi(x)=\mu\frac{(b+ax)^{2d+1}(a\lambda x+2a\nu-b\lambda)}{(\nu+\lambda x)^{d}}\int_{x_1}^x\frac{(\nu+\lambda u)^{d}\left((\nu+\lambda u)-\frac{\nu+\lambda x_1}{(b+ax_1)^2}(b+au)^2\right)}{(b+au)^{2d+2}(a\lambda u+2a\nu-b\lambda)^2}\mathrm{d}u,\;x_0< x<x_1,\\\\
 r=1,\mu<0,a>0,
 \end{array}\right.
\end{equation}
where
$$\left\{\begin{array}{l}
   -\frac{b}{a}=x_0<  x_1\leq +\infty\; \mathrm{for}\; \lambda>0\;\mathrm{and}\;-\frac{b}{a}\geq -\frac{\nu}{\lambda}, \\\\
   -\frac{b}{a}=  x_0<x_1\leq-\frac{\nu}{\lambda}\; \mathrm{for}\; \lambda<0\;\mathrm{and}\; -\frac{\nu}{\lambda}>-\frac{b}{a}.
  \end{array}\right.
$$

$(\mathrm{III})$ Let $r=1$, $\varphi(x_1)=0$ and $\varphi'(x_1)=-1$.

$(\mathrm{III-1})$ $\widetilde{g}=\frac{1}{f^2}g_F$ is an Einstein  metric (namely $\mathrm{Ric}_{\widetilde{g}}=\mu \;\widetilde{g}$) with the fibre metrics complete at $t=-\infty$ on the domain $ M=\left\{(z,w)\in \Omega\times\mathbb{C}^{r}: 0< e^{\lambda\phi(z)}\|w\|^2\right\}$ if and only if
\begin{equation}\label{ne4.1.4}
 \left\{ \begin{array}{l}
  \mathrm{Ric}_{g_0}=\gamma\;g_0,\\\\
  \gamma=\frac{a\lambda x_1+2a\nu-b\lambda}{ax_1+b},\\\\
 \varphi(x)=-\gamma\frac{(a\lambda x+2a\nu-b\lambda)(ax+b)^{2d+1}}{(\nu+\lambda x)^d}\int_{x_1}^x\frac{(\nu+\lambda u)^d}{(au+b)^{2d}(a\lambda u+2a\nu-b\lambda)^2}\mathrm{d}u,\;x_0<x<x_1,\\\\
  a>0, \mu=0,
 \end{array}\right.
\end{equation}
where
$$\left\{\begin{array}{l}
   -\frac{b}{a}=x_0<  x_1< +\infty\; \mathrm{for}\; \lambda>0\;\mathrm{and}\;-\frac{b}{a}\geq -\frac{\nu}{\lambda}, \\\\
   -\frac{b}{a}=  x_0<x_1\leq-\frac{\nu}{\lambda}\; \mathrm{for}\; \lambda<0\;\mathrm{and}\; -\frac{\nu}{\lambda}>-\frac{b}{a}.
  \end{array}\right.
$$

$(\mathrm{III-2})$ $\widetilde{g}=\frac{1}{f^2}g_F$ is an Einstein  metric (namely $\mathrm{Ric}_{\widetilde{g}}=\mu \;\widetilde{g}$) with the fibre metrics complete at $t=0$ on the domain $ M=\left\{(z,w)\in \Omega\times\mathbb{C}^{r}: 1< e^{\lambda\phi(z)}\|w\|^2\right\}$ if and only if
\begin{equation}\label{ne4.1.4}
 \left\{ \begin{array}{l}
  \mathrm{Ric}_{g_0}=\gamma\;g_0,\\\\
  \gamma=\mu\frac{\nu+\lambda x_1}{(b+ax_1)^2}+\frac{a\lambda x_1+2a\nu-b\lambda}{ax_1+b},\\\\
 \varphi(x)=\frac{(a\lambda x+2a\nu-b\lambda)(ax+b)^{2d+1}}{(\nu+\lambda x)^d}\int_{x_1}^x\frac{(\nu+\lambda u)^d(\mu(\nu+\lambda u)-\gamma (b+au)^2)}{(au+b)^{2d}(a\lambda u+2a\nu-b\lambda)^2}\mathrm{d}u,\;x_0<x<x_1,\\\\
 a>0, \mu<0,
 \end{array}\right.
\end{equation}
where
$$\left\{\begin{array}{l}
   -\frac{b}{a}=x_0<  x_1< +\infty\; \mathrm{for}\; \lambda>0\;\mathrm{and}\;-\frac{b}{a}\geq -\frac{\nu}{\lambda}, \\\\
   -\frac{b}{a}=  x_0<x_1\leq-\frac{\nu}{\lambda}\; \mathrm{for}\; \lambda<0\;\mathrm{and}\; -\frac{\nu}{\lambda}>-\frac{b}{a}.
  \end{array}\right.
$$

$(\mathrm{IV})$ Let $r>1$, $\varphi(x_1)=0$ and $\varphi'(x_1)=-1$.

$(\mathrm{IV-1})$ $\widetilde{g}=\frac{1}{f^2}g_F$ is an Einstein  metric (namely $\mathrm{Ric}_{\widetilde{g}}=\mu \;\widetilde{g}$) with the fibre metrics complete at $t=-\infty$ on the domain $ M=\left\{(z,w)\in \Omega\times\mathbb{C}^{r}: 0< e^{\lambda\phi(z)}\|w\|^2\right\}$ if and only if
\begin{equation}\label{V-1.1}
 \left\{ \begin{array}{l}
  \mathrm{Ric}_{g_0}=r\lambda\;g_0,\\\\
  \varphi(x)= r\frac{(ax+b)^{2n-1}(ax-b)}{x^{n-1}}\int^{x_1}_x \frac{ u^{n-1}}{(au+b)^{2n-2}(au-b)^2}\mathrm{d}u,\;-\frac{b}{a}=x_0<x<x_1=-\frac{(r+1)b}{(r-1)a},\\\\
  a>0, b<0,\lambda>0,\mu=\nu=0.
 \end{array}\right.
\end{equation}

$(\mathrm{IV-2})$ $\widetilde{g}=\frac{1}{f^2}g_F$ is an Einstein  metric (namely $\mathrm{Ric}_{\widetilde{g}}=\mu \;\widetilde{g}$) with the fibre metrics complete at $t=0$ on the domain $ M=\left\{(z,w)\in \Omega\times\mathbb{C}^{r}: 1< e^{\lambda\phi(z)}\|w\|^2\right\}$ if and only if
\begin{equation}\label{V-2.1}
 \left\{ \begin{array}{l}
  \mathrm{Ric}_{g_0}=r\lambda\;g_0,\\\\
   \mu=\frac{(ax_1+b)(r(ax_1+b)-(ax_1-b))}{x_1}<0,\\\\
 \varphi(x)= \frac{(ax+b)^{2n-1}(ax-b)}{x^{n-1}}\int_{x_1}^x \frac{ u^{n-1}\left(\mu u-r(b+au)^2\right)}{(au+b)^{2n}(au-b)^2}\mathrm{d}u,\;-\frac{b}{a}=x_0<x<x_1<-\frac{(r+1)b}{(r-1)a},\\\\
 a>0, b<0,\lambda>0,\nu=0.
 \end{array}\right.
\end{equation}

 }\end{Theorem}

\begin{proof}[Proof]

By the well-known theorems of Bonnet and Myers, it follows that non-compact manifolds do not admit any Einstein metric with positive scalar curvature. Since $M$ is not compact,  we get $\mu\leq 0$. The following proof is divided into four parts.\\

\begin{center}
\textbf{Part $(\mathrm{I})$}\\
\end{center}

 From the metric $g_F$ is smooth at $t=-\infty$, it follows that
\begin{equation*}
  0=x_0=\lim_{t\rightarrow-\infty}x(t)< \sup x(t)=x_1,\;\varphi(0)=0,\;\varphi'(0)=1.
\end{equation*}

Using \eqref{eq7.3.4}, we obtain $\gamma=\lambda+\mu-2a$ and
\begin{equation}\label{2.20}
 \varphi(x)=\frac{(1+ax)^{2d+1}(2a-\lambda+a\lambda x)}{(1+\lambda x)^{d}}m(x),
\end{equation}
where
\begin{equation}\label{2.20.1}
 m(x)= \int_0^x\frac{(1+\lambda u)^{d}
     \left(\mu(1+\lambda u)-(\lambda+\mu-2a)(1+au)^2\right)}{(1+au)^{2d+2}(a\lambda u+2a-\lambda)^2}\mathrm{d}u.
\end{equation}\\

$(\mathrm{I-i})$ For the case of $a=\lambda$.\\

By \eqref{2.20}, we get
\begin{eqnarray*}
 \varphi(x)&=&\frac{1}{a}(1+ax)^{d+2}\int_0^x\frac{\mu-(\lambda+\mu-2a)(1+au)}{(1+au)^{d+3}}\mathrm{d}u\\
   &=&\frac{1}{a^2}(1+ax)^{d+2}\left(\frac{\mu}{d+2}(1-\frac{1}{(1+ax)^{d+2}})-\frac{\lambda+\mu-2a}{d+1}(1-\frac{1}{(1+ax)^{d+1}})\right)  \\
   &=& \frac{1}{a^2}\left(\frac{\mu}{d+2}-\frac{\lambda+\mu-2a}{d+1}\right)(1+ax)^{d+2}+\frac{\lambda+\mu-2a}{(d+1)a^2}(1+ax) -\frac{\mu}{(d+2)a^2}\\
   &=& -\frac{\mu+(d+2)(\lambda-2a)}{(d+2)(d+1)a^2}(1+ax)^{d+2}+\frac{\lambda+\mu-2a}{(d+1)a^2}(1+ax) -\frac{\mu}{(d+2)a^2}\\
   &=& \frac{(d+2)a-\mu}{(d+2)(d+1)a^2}(1+ax)^{d+2}+\frac{\mu-a}{(d+1)a^2}(1+ax) -\frac{\mu}{(d+2)a^2}.
\end{eqnarray*}

If $a>0$, by $\mu\leq 0$, then $\mu-(\lambda+\mu-2a)(1+ax)=a(1+(a-\mu)x)\geq 0$ for $x\geq 0$, thus $\varphi(x)>0$ for $x>0$, so whether $x_1$ is finite or infinite, there is
\begin{equation*}
  \int_{0}^{x_1}\frac{\mathrm{d}x}{(1+ax)\sqrt{\varphi(x)}}<+\infty.
\end{equation*}
This indicates that the fiber metrics for $\widetilde{g}$ are incomplete.

If $a<0$, by $1+ax>0$, then $0<x<-\frac{1}{a}$, thus $\frac{\mu-(\lambda+\mu-2a)(1+ax)}{a}>0$ for $0<x<-\frac{1}{a}$. If $0<x_1<-\frac{1}{a}$, then
\begin{equation*}
  \int_{0}^{x_1}\frac{\mathrm{d}x}{(1+ax)\sqrt{\varphi(x)}}<+\infty,
\end{equation*}
hence $x_1=-\frac{1}{a}$.

From $\varphi(-\frac{1}{a})=-\frac{\mu}{(d+2)a^2}\geq 0$, we have
\begin{equation*}
  \int_{0}^{x_1}\frac{\mathrm{d}x}{(1+ax)\sqrt{\varphi(x)}}=+\infty,
\end{equation*}
This show that fiber metrics are complete.

Let
\begin{equation*}
  t=\int_{-\frac{1}{a}}^{x(t)}\frac{\mathrm{d}v}{\varphi(v)},\;F(t)=\int_{-\frac{1}{a}}^{x(t)}\frac{v\mathrm{d}v}{\varphi(v)}
\end{equation*}
for $\mu<0$, and
\begin{equation*}
  t=\int_{-\frac{1}{2a}}^{x(t)}\frac{\mathrm{d}v}{\varphi(v)},\;F(t)=\int_{-\frac{1}{2a}}^{x(t)}\frac{v\mathrm{d}v}{\varphi(v)}
\end{equation*}
for $\mu=0$.

It follows by $\varphi(0)=0$, $\varphi'(0)=1$ and $\varphi(-\frac{1}{a})=-\frac{\mu}{(d+2)a^2}$ that
\begin{equation*}
  \lim_{x\rightarrow 0^{+}} t=-\infty,\; \lim_{x\rightarrow (-\frac{1}{a})^{-}} t=0
\end{equation*}
for $\mu<0$, and
\begin{equation*}
  \lim_{x\rightarrow 0^{+}} t=-\infty,\; \lim_{x\rightarrow (-\frac{1}{a})^{-}} t=+\infty
\end{equation*}
for $\mu=0$.\\

$(\mathrm{I-ii})$ For the case of $\lambda=2a$.\\

Using
\begin{equation*}
  \gamma=\lambda+\mu-2a=\mu, \mu(1+\lambda x)-\gamma (1+ax)^2=-a^2\mu x^2
\end{equation*}
and \eqref{2.20}, we get
\begin{equation*}
     \varphi(x)=-\frac{\mu}{2}\frac{x(1+ax)^{2d+1}}{(1+\lambda x)^{d}} \int_0^x\frac{(1+\lambda u)^{d}
     }{(1+au)^{2d+2}}\mathrm{d}u,
\end{equation*}
so
\begin{equation*}
  \varphi(0)=\varphi'(0)=0.
\end{equation*}
This conflicts with the smoothness of the metric $\widetilde{g}$ at $t=-\infty$.\\

$(\mathrm{I-iii})$ For the case of $\lambda\neq a$ and $\lambda\neq 2a$.\\

$(\mathrm{I-iii}-1)$ If $\frac{\lambda-2a}{a\lambda}=\frac{1}{a}-\frac{2}{\lambda}<-\frac{1}{\lambda}<-\frac{1}{a}<0$ or $-\frac{1}{a}<-\frac{1}{\lambda}<\frac{1}{a}-\frac{2}{\lambda}<0$, that is $0<\lambda<a$ or $0<a<\lambda<2a$.

By $\mu\leq 0$, then
\begin{equation*}
 \gamma=\lambda+\mu-2a<0,  2a-\lambda+a\lambda x=a\lambda\left(x+\frac{2}{\lambda}-\frac{1}{a}\right)>0, 1+\lambda x>0,1+ax>0
\end{equation*}
for $x\geq 0$.

Let
\begin{equation*}
 q(x)= \mu(1+\lambda x)-\gamma (1+ax)^2.
\end{equation*}
By
\begin{equation*}
 q(0)=\mu-\gamma=2a-\lambda>0,\; q'(x)=\lambda\mu-2a\gamma(1+ax)=(2a-\lambda)(2a-\mu)-2a^2\gamma x>0
\end{equation*}
for $x\geq 0$, it follows that $q(x)>0$ for $x\geq 0$.

Using \eqref{2.20}, we get $\varphi(x)>0$ for $x>0$ and $\varphi(x)\sim Cx^{d+2}$ as $x\rightarrow+\infty$, thus
\begin{equation*}
  \int_{0}^{x_1}\frac{\mathrm{d}x}{(1+ax)\sqrt{\varphi(x)}}<+\infty.
\end{equation*}
This indicates that the fiber metrics for $\widetilde{g}$ are incomplete.\\

$(\mathrm{I-iii}-2)$ If  $\frac{\lambda-2a}{a\lambda}=\frac{1}{a}-\frac{2}{\lambda}<-\frac{1}{\lambda}<0<-\frac{1}{a}$ or $0<-\frac{1}{a}<-\frac{1}{\lambda}<\frac{\lambda-2a}{a\lambda}$, that is $a<0<\lambda$ or $a<\lambda<0$, then
\begin{equation*}
  \lambda-2a-a\lambda x>0
\end{equation*}
for $0\leq x\leq c:=-\frac{1}{a}$.

Let
\begin{equation*}
  q(x)=\mu(1+\lambda x)-(\lambda+\mu-2a)(1+ax)^2.
\end{equation*}

Using $\mu\leq 0$,  $q(0)=2a-\lambda<0$ and $q(c)=\mu(1-\frac{\lambda}{a})\leq 0$, we obtain:

For $\gamma=\lambda+\mu-2a=0$, $q(x)\leq 0$ for $x\in (0,c)$;

For $\gamma=\lambda+\mu-2a<0$,  an equation $q(x)=0$  has no root in the interval $(0, c)$, so $q(x)\leq 0$ for  $x\in(0,c)$;

For $\gamma=\lambda+\mu-2a>0$, $\mu(1+ax)\leq 0$ and $\gamma(1+ax)^2>0$ for $x\in(0,c)$,  therefore $q(x)\leq 0$ for  $x\in(0,c)$.

By \eqref{2.20}, the above indicates that $\varphi(x)> 0$ for  $x\in(0,c)$.

Since $1+ax>0$, it follows that $0\leq x<-\frac{1}{a}$, so $0<x_1\leq -\frac{1}{a}$. By the fiber metrics for $\widetilde{g}$ are complete, that is
\begin{equation*}
  \int_{0}^{x_1}\frac{\mathrm{d}x}{(1+ax)\sqrt{\varphi(x)}}=+\infty,
\end{equation*}
we have $x_1=-\frac{1}{a}$.

For $\mu<0$,
\begin{eqnarray*}
 \lim_{x\rightarrow c^{-}} \varphi(x)&=& \frac{2(a-\lambda)}{(1-\frac{\lambda}{a})^d} \lim_{x\rightarrow c^{-}}
 \frac{\int_0^x\frac{(1+\lambda u)^{d}\left(\mu(1+\lambda u)-(\lambda+\mu-2a)(1+au)^2\right)}{(1+au)^{2d+2}(a\lambda u+2a-\lambda)^2}\mathrm{d}u}{(1+ax)^{-(2d+1)}} \\
   &=&\frac{2(a-\lambda)}{(1-\frac{\lambda}{a})^{d}}\frac{(1-\frac{\lambda}{a} )^{d+1}\mu}{(2a-2\lambda)^2}\frac{1}{-(2d+1)a}  \\
   &=& \frac{-\mu}{2(2d+1)a^2}>0.
\end{eqnarray*}
Let
\begin{equation*}
  t=\int_{c}^{x(t)}\frac{\mathrm{d}v}{\varphi(v)},\;F(t)=\int_{c}^{x(t)}\frac{v\mathrm{d}v}{\varphi(v)},
\end{equation*}
it follows that
\begin{equation*}
  \lim_{x\rightarrow 0^{+}} t=-\infty,\; \lim_{x\rightarrow c^{-}} t=0,
\end{equation*}
which implies that the $g_F$ is defined on $M=\left\{(z,w)\in \Omega\times\mathbb{C}^{r}: e^{\lambda\phi(z)}\|w\|^2<1\right\}$. Using
\begin{equation*}
\int_{0}^{c}\frac{1}{(1+ax)\sqrt{\varphi(x)}}\mathrm{d}x=+\infty,
\end{equation*}
we obtain that the fiber metrics of $\widetilde{g}$ is complete on $M=\left\{(z,w)\in \Omega\times\mathbb{C}^{r}: e^{\lambda\phi(z)}\|w\|^2<1\right\}$.

For $\mu=0$,
\begin{eqnarray*}
 \lim_{x\rightarrow c^{-}} \frac{\varphi(x)}{(1+ax)^2} &=& \frac{2(a-\lambda)}{(1-\frac{\lambda}{a})^d} \lim_{x\rightarrow c^{-}}
 \frac{\int_0^x\frac{(1+\lambda u)^{d}(2a-\lambda)}{(1+au)^{2d}(a\lambda u+2a-\lambda)^2}\mathrm{d}u}{(1+ax)^{-(2d-1)}} \\
   &=&\frac{2(a-\lambda)}{(1-\frac{\lambda}{a})^d}\frac{(1-\frac{\lambda}{a} )^{d}(2a-\lambda)}{(2a-2\lambda)^2}\frac{1}{-(2d-1)a}  \\
   &=& \frac{\lambda-2a}{2(2d-1)a^2(1-\frac{\lambda}{a})}>0.
\end{eqnarray*}

Let
\begin{equation*}
  t=\int_{\frac{c}{2}}^{x(t)}\frac{\mathrm{d}v}{\varphi(v)},\;F(t)=\int_{\frac{c}{2}}^{x(t)}\frac{v\mathrm{d}v}{\varphi(v)},
\end{equation*}
it follows that
\begin{equation*}
  \lim_{x\rightarrow 0^{+}} t=-\infty,\; \lim_{x\rightarrow c^{-}} t=+\infty.
\end{equation*}
 Using
\begin{equation*}
\int_{\frac{c}{2}}^{c}\frac{1}{(1+ax)\sqrt{\varphi(x)}}\mathrm{d}x=+\infty,
\end{equation*}
we obtain that the fiber metrics of $\widetilde{g}$ is complete on $M=\Omega\times \mathbb{C}^{r}$.\\

$(\mathrm{I-iii}-3)$ If  $\frac{\lambda-2a}{a\lambda}=\frac{1}{a}-\frac{2}{\lambda}< 0<-\frac{1}{\lambda}<-\frac{1}{a}$ or $-\frac{1}{a}<0<-\frac{1}{\lambda}<\frac{1}{a}-\frac{2}{\lambda}$, then $\lambda<2a<0$ or $\lambda<0<a$.

 By $1+ax>0$ and $1+\lambda x>0$, we have $0\leq x<-\frac{1}{\lambda}$, so $0<x_1\leq c:=-\frac{1}{\lambda}$.

 Since $\varphi(x)>0$ for $0=x_0<x<x_1$, by \eqref{2.20} and \eqref{2.20.1}, thus $m(x_1)\geq 0$.

When $x_1=-\frac{1}{\lambda}$ and
\begin{equation*}
 m(x_1)= \int_0^{x_1}\frac{(1+\lambda u)^{d}
     \left(\mu(1+\lambda u)-(\lambda+\mu-2a)(1+au)^2\right)}{(1+au)^{2d+2}(a\lambda u+2a-\lambda)^2}\mathrm{d}u=0,
\end{equation*}
from \eqref{2.20}, we get
\begin{equation*}
  \lim_{x\rightarrow x_1^{-}}\frac{\varphi(x)}{1+\lambda x}=\frac{\mu+\lambda-2a}{(d+1)\lambda^2}<0.
\end{equation*}
This conflicts with $\varphi(x)>0$ for $0<x<x_1$.

When $x_1<-\frac{1}{\lambda}$, or $x_1=-\frac{1}{\lambda}$ and $m(x_1)>0$, by \eqref{2.20} we have
\begin{equation*}
  \int_{0}^{x_1}\frac{\mathrm{d}x}{(1+ax)\sqrt{\varphi(x)}}<+\infty.
\end{equation*}
This indicates that the fiber metrics for $\widetilde{g}$ are incomplete.\\

$(\mathrm{I-iii}-4)$ If  $0<\frac{\lambda-2a}{a\lambda}=\frac{1}{a}-\frac{2}{\lambda}< -\frac{1}{\lambda}<-\frac{1}{a}$ or $-\frac{1}{a}<-\frac{1}{\lambda}<0<\frac{\lambda-2a}{a\lambda}$, then $2a< \lambda<a<0$ or $0<2a<\lambda$. Since $1+ax>0$ and $1+\lambda x>0$ for $0=x_0<x<x_1$, it follows that $0<x_1\leq -\frac{1}{\lambda}$ for $2a< \lambda<a<0$.

Let
\begin{equation*}
\psi(x)= d\gamma(1+ax)^2-(d+1)\mu(1+\lambda x).
\end{equation*}

If $x_1=+\infty$, using $\varphi(x)\sim Cx^{d+2}$ as $x\rightarrow+\infty$, then
\begin{equation*}
   \int_{0}^{x_1}\frac{\mathrm{d}x}{(1+ax)\sqrt{\varphi(x)}}<+\infty,
\end{equation*}
this is in contradiction with the completeness of the metric $\widetilde{g}$.

If $x_1<+\infty$, from the completeness of $\widetilde{g}$, we get $\varphi(x_1)=\varphi'(x_1)=0$, thus by \eqref{eq7.1} and \eqref{eq7.3.4}, we have
\begin{equation}\label{2.20-1}
  \varphi(x)=\frac{(1+ax)^{2d+1}}{(1+\lambda x)^d}\int_x^{x_1}(u-x)\frac{(1+\lambda u)^{d-1}\psi(u)}{(1+au)^{2d+3}}\mathrm{d}u \;(\text{by}\;\varphi(x_1)=0\;\text{and}\;\varphi'(x_1)=0)
\end{equation}
and
\begin{equation*}
  \mu(1+\lambda x_1)-\gamma(1+ax_1)^2=0  \;(\text{by}\;\varphi(x_1)=0\;\text{and}\;\varphi'(x_1)=0).
\end{equation*}
Then
\begin{equation*}
  \gamma=\frac{\mu(1+\lambda x_1)}{(1+ax_1)^2}\leq 0
\end{equation*}
and
\begin{equation*}
  \psi(x)= \mu\left\{d\frac{(1+\lambda x_1)}{(1+ax_1)^2}(1+ax)^2-(d+1)(1+\lambda x)\right\},
\end{equation*}
which implies that $$\psi(x_1)=-\mu(1+\lambda x_1)\geq 0.$$By $\mu\leq 0$, $\gamma=\lambda+\mu-2a$ and $\lambda-2a>0$, we have
\begin{equation*}
  \psi(0)=d\gamma-(d+1)\mu=d(\lambda-2a)-\mu>0,
\end{equation*}
thus $\psi(x)\geq 0$ for $0\leq x\leq x_1$. So from  \eqref{2.20-1}, we have $\varphi(0)>0$, this conflicts with $\varphi(0)=0$.

\newpage
\begin{center}
\textbf{Part $(\mathrm{II})$}
\end{center}

\noindent\textbf{Part $(\mathrm{II})$ for $r=1$}\\

$(\mathrm{II-i})$ For the case of $a<0$ and $\lambda<0$.\\

By $ax+b>0$ and $\nu+\lambda x>0$, we get $x<\min\{-\frac{b}{a},-\frac{\nu}{\lambda}\}$, so $x_1\leq \min\{-\frac{b}{a},-\frac{\nu}{\lambda}\}$.

For $x_0>-\infty$, according to the completeness of the fibre metric for $\widetilde{g}$, namely
\begin{equation*}
  \int_{x_0}^{\frac{x_0+x_1}{2}}\frac{1}{(b+ax)\sqrt{\varphi(x)}}\mathrm{d}x=+\infty,
\end{equation*}
we have $\varphi(x_0)=\varphi'(x_0)=0$. Then by \eqref{eq7.3.4}, we obtain $\gamma=\frac{\mu(\nu+\lambda x_0)}{(b+ax_0)^2}$ and
\begin{equation}\label{4.16-1}
    \varphi(x)= -\frac{(b+ax)^{2d+1}(b\lambda-2a\nu-a\lambda x)}{(\nu+\lambda x)^{d}}\int_{x_0}^x\frac{(\nu+\lambda u)^{d}\left(\mu(\nu+\lambda u)-\gamma(b+au)^2\right)}{(b+au)^{2d+2}(b\lambda-2a\nu-a\lambda u)^2}\mathrm{d}u
\end{equation}
for $x\in(x_0,x_1)$.\\

$(\mathrm{II-i}-1)$ If $-\frac{b}{a}\leq-\frac{\nu}{\lambda}\leq\frac{b}{a}-\frac{2\nu}{\lambda}$.

 For  $-\infty<x_0<x<x_1\leq c:=-\frac{b}{a}$, using
\begin{equation*}
  b\lambda-2a\nu-a\lambda x>0, \frac{{\nu+\lambda x_0}}{(b+ax_0)^2}>0,\;{\nu+\lambda x}-\frac{{\nu+\lambda x_0}}{(b+ax_0)^2}(b+ax)^2>0,
\end{equation*}
we have
\begin{equation*}
  \mu(\nu+\lambda x)-\gamma(b+ax)^2=\mu\left({\nu+\lambda x}-\frac{{\nu+\lambda x_0}}{(b+ax_0)^2}(b+ax)^2\right)<0
\end{equation*}
for  $x_0<x<x_1\leq c$ only if $\mu<0$. So $\varphi(x)>0$ for $x\in(x_0,x_1)$ only if $\mu<0$.

Using
\begin{equation*}
   \lim_{x\rightarrow c^{-}} \varphi(x)=\left\{\begin{array}{l}
                                                 \frac{-\mu}{2(2d+1)a^2}>0\;(\text{for}\;-\frac{b}{a}<-\frac{\nu}{\lambda}), \\\\
                                                 \frac{-\mu}{(d+2)a^2}>0\;(\text{for}\;-\frac{b}{a}=-\frac{\nu}{\lambda})
                                               \end{array}
   \right. \;(\text{by}\;\eqref{eq7.3.4})
\end{equation*}
and fiber metrics for $\widetilde{g}$ are complete, we have $x_1=c$.

Let
\begin{equation*}
  t=\int_{c}^{x(t)}\frac{\mathrm{d}u}{\varphi(u)},\;F(t)=\int_{c}^{x(t)}\frac{v\mathrm{d}v}{\varphi(v)},
\end{equation*}
by
\begin{equation*}
  \lim_{x\rightarrow x_0^{+}}\frac{\varphi(x)}{(x-x_0)^2}=-\frac {\mu}{2(b+ax_0)^2}\;(\text{by}\;\eqref{4.16-1}),
\end{equation*}
then
\begin{equation*}
  \lim_{x\rightarrow x_0^{+}}t(x)=-\infty, \;\lim_{x\rightarrow c^{-}}t(x)=0.
\end{equation*}
This shows that the metric $g_F$ is defined on $ M=\left\{(z,w)\in \Omega\times\mathbb{C}^{r}: 0< e^{\lambda\phi(z)}\|w\|^2<1\right\}$. Since
\begin{equation*}
\int_{x_0}^{\frac{x_0+x_1}{2}}\frac{1}{(b+ax)\sqrt{\varphi(x)}}\mathrm{d}x=+\infty,\;\int_{\frac{x_0+x_1}{2}}^{x_1}\frac{1}{(b+ax)\sqrt{\varphi(x)}}\mathrm{d}x=+\infty,
\end{equation*}
it follows that any fiber metric for $\widetilde{g}$ is complete.

For $x_0=-\infty$, from \eqref{eq7.3.4}, we obtain
\begin{equation*}
  \frac{(\nu+\lambda x)^d\varphi(x)}{(ax+b)^{2d+1}(a\lambda x+2a\nu-b\lambda)}
  =C_1+\int_{\widetilde{x_0}}^x  \frac{(\nu+\lambda u)^d\left(\mu(\nu+\lambda u)-\gamma (au+b)^2\right)}{(au+b)^{2d+2}(a\lambda u+2a\nu-b\lambda)^2}\mathrm{d}u,
\end{equation*}
where
\begin{equation*}
  \widetilde{x_0}:=x_1-1,C_1= \frac{(\nu+\lambda \widetilde{x_0})^d\varphi(\widetilde{x_0})}{(a\widetilde{x_0}+b)^{2d+1}(a\lambda \widetilde{x_0}+2a\nu-b\lambda)}.
\end{equation*}
Since
\begin{equation*}
 -\infty <C_1+\int_{\widetilde{x_0}}^{-\infty}  \frac{(\nu+\lambda u)^d\left(\mu(\nu+\lambda u)-\gamma (au+b)^2\right)}{(au+b)^{2d+2}(a\lambda u+2a\nu-b\lambda)^2}\mathrm{d}u<+\infty
\end{equation*}
and $a\lambda x+2a\nu-b\lambda<0$ for $-\infty<x<x_1$, it follows that
\begin{equation*}
  \lim_{x\rightarrow-\infty} \frac{(\nu+\lambda x)^d\varphi(x)}{(ax+b)^{2d+1}(a\lambda x+2a\nu-b\lambda)}=C\leq 0.
\end{equation*}
If $C<0$, then $\varphi(x)\sim C_2x^{d+2}$ as $x\rightarrow -\infty$, this is in contradiction with
\begin{equation*}
  \int_{-\infty}^{\widetilde{x_0}}\frac{\mathrm{d}x}{(b+ax)\sqrt{\varphi(x)}}=+\infty,
\end{equation*}
thus $C=0$, so
\begin{equation}\label{2.20.1.1}
     \varphi(x)=\frac{(b+ax)^{2d+1}(a\lambda x+2a\nu-b\lambda)}{(\nu+\lambda x)^{d}}
     \int_{-\infty}^x\frac{(\nu+\lambda u)^{d}\left(\mu(\nu+\lambda u)-\gamma(b+au)^2\right)}{(b+au)^{2d+2}(a\lambda u+2a\nu-b\lambda)^2}\mathrm{d}u
\end{equation}
for $x\in(-\infty,x_1)$.

As $x\rightarrow-\infty,$ by
\begin{equation*}
 -\frac{(b+ax)^{2d+1}(b\lambda-2a\nu-a\lambda x)}{(\nu+\lambda x)^{d}} \sim \frac{a^{2d+2}}{\lambda^{d-1}}x^{d+2},
\end{equation*}

\begin{equation*}
  \frac{(\nu+\lambda x)^{d}\left(\mu(\nu+\lambda x)-\gamma(b+ax)^2\right)}{(b+ax)^{2d+2}(a\lambda x+2a\nu-b\lambda)^2}\sim\left\{\begin{array}{l}
                                                                                                                        -\frac{\lambda^{d-2}\gamma}{a^{2d+2}}\frac{1}{x^{d+2}},   \gamma\neq 0, \\\\
                                                                                                                         \frac{\lambda^{d-1}\mu}{a^{2d+4}}\frac{1}{x^{d+3}},   \gamma=0
                                                                                                                         \end{array}
  \right.
\end{equation*}
and \eqref{2.20.1.1}, we get
\begin{equation*}
 \lim_{x\rightarrow-\infty}\frac{\varphi(x)}{x}=\frac{\gamma}{(d+1)\lambda}\;(\gamma\neq 0)
\end{equation*}
and
\begin{equation*}
 \lim_{x\rightarrow-\infty}\varphi(x)=-\frac{\mu}{(d+2)a^2}\;(\gamma=0).
\end{equation*}
By
\begin{equation*}
  \int_{-\infty}^{\widetilde{x_0}}\frac{\mathrm{d}x}{(b+ax)\sqrt{\varphi(x)}}=+\infty,
\end{equation*}
it follows that $\gamma=0$. From \eqref{2.20.1.1}, so $\varphi(x)>0$ for $x\in(-\infty,c)$ only $\mu<0$.

If $x_1<c$, then
\begin{equation*}
  \int_{x_1-1}^{x_1}\frac{1}{(b+ax)\sqrt{\varphi(x)}}\mathrm{d}x<+\infty,
\end{equation*}
according to the completeness of fiber metrics for $\widetilde{g}$, we get $x_1=c$

Let
\begin{equation*}
  t=\int_{c}^{x(t)}\frac{\mathrm{d}u}{\varphi(u)},\;F(t)=\int_{c}^{x(t)}\frac{v\mathrm{d}v}{\varphi(v)},
\end{equation*}
then
\begin{equation*}
  \lim_{x\rightarrow -\infty}t(x)=-\infty, \;\lim_{x\rightarrow c^{-}}t(x)=0.
\end{equation*}
This shows that the metric $g_F$ is defined on $ M=\left\{(z,w)\in \Omega\times\mathbb{C}^{r}: 0<e^{\lambda\phi(z)} \|w\|^2<1\right\}$. Since
\begin{equation*}
\int_{-\infty}^{c-1}\frac{1}{(b+ax)\sqrt{\varphi(x)}}\mathrm{d}x=+\infty,\;\int_{c-1}^{c}\frac{1}{(b+ax)\sqrt{\varphi(x)}}\mathrm{d}x=+\infty,
\end{equation*}
it follows that any fiber metric of $\widetilde{g}$ is complete.\\

$(\mathrm{II-i}-2)$ If $\frac{b}{a}-\frac{2\nu}{\lambda}<-\frac{\nu}{\lambda}<-\frac{b}{a}$, then $ x_0<x_1\leq-\frac{\nu}{\lambda}<-\frac{b}{a}$. The following we prove that this situation is impossible.

If $\frac{b}{a}-\frac{2\nu}{\lambda}= x_0<x_1\leq-\frac{\nu}{\lambda}<-\frac{b}{a}$ , by $\varphi(x_0)=\varphi'(x_0)=0$ and $\eqref{eq7.3.4}$, we get
\begin{equation*}
  \gamma=-{\frac {\mu\,{\lambda}^{2}}{4a \left( a\nu-b\lambda \right) }},\;\mu(\nu+\lambda x)-\gamma(b+ax)^2=-a^2\gamma(x-x_0)^2.
\end{equation*}
Thus
\begin{equation*}
  \varphi(x)= -\frac{a\gamma}{\lambda}\frac{(b+ax)^{2d+1}(x-x_0)}{(\nu+\lambda x)^{d}}\int_{x_0}^x\frac{(\nu+\lambda u)^{d}}{(b+au)^{2d+2}}\mathrm{d}u
\end{equation*}
and
\begin{equation*}
  \int_{\frac{x_0+x_1}{2}}^{x_1}\frac{1}{(b+ax)\sqrt{\varphi(x)}}\mathrm{d}x<+\infty.
\end{equation*}
This is in contradiction with the completeness of fiber metrics of $\widetilde{g}$.

If $-\infty< x_0<x_1\leq-\frac{\nu}{\lambda}<-\frac{b}{a}$ and $x_0\neq \frac{b}{a}-\frac{2\nu}{\lambda}$,   by the completeness of fiber metrics of $\widetilde{g}$, we get $\varphi(x_0)=\varphi'(x_0)=\varphi(x_1)=\varphi'(x_1)=0$, so
\begin{equation*}
 \gamma=\frac{\mu(\nu+\lambda x_0)}{(b+ax_0)^2}=\frac{\mu(\nu+\lambda x_1)}{(b+ax_1)^2}\leq 0,
\end{equation*}
\begin{equation*}
  \mu(\nu+\lambda x)-\gamma(ax+b)^2=a^2\gamma(x-x_0)(x_1-x)
\end{equation*}
and
\begin{equation*}
\psi(x):= d\gamma(ax+b)^2-(d+1)\mu(\nu+\lambda x)=-da^2\gamma(x-x_0)(x_1-x)-\mu(\nu+\lambda x)\geq 0
\end{equation*}
for $x_0<x<x_1$. Using \eqref{eq7.1}, we have
\begin{equation*}
  \varphi(x)=\frac{(b+ax)^{2d+1}}{(\nu+\lambda x)^d}\int^x_{x_0}(x-u)\frac{(\nu+\lambda u)^{d-1}\psi(u)}{(b+au)^{2d+3}}\mathrm{d}u \;(\text{by}\;\varphi(x_0)=0\;\text{and}\;\varphi'(x_0)=0),
\end{equation*}
which implies that $\varphi(x_1)>0$, this is in contradiction with the completeness of $\widetilde{g}$.

If $x_0=-\infty$, by $\eqref{2.20.1.1}$, we have $\gamma=0$ and $\mu<0$. By $\varphi(x_1)=0, \varphi'(x_1)=0$ and \eqref{eq7.1}, we have
\begin{equation*}
  \varphi(x)=\frac{(b+ax)^{2d+1}}{(\nu+\lambda x)^d}\int^x_{x_1}(x-u)\frac{(\nu+\lambda u)^{d-1}\psi(u)}{(b+au)^{2d+3}}\mathrm{d}u
\end{equation*}
which implies that $\varphi(x)\sim C_3x^{d+2}$ as $x\rightarrow-\infty$, this is in contradiction with the completeness of $\widetilde{g}$.\\

$(\mathrm{II-ii})$  For the case of $a>0$ and $\lambda>0$.\\

By $ax+b>0$ and $\nu+\lambda x>0$, we get $x>\max\{-\frac{b}{a},-\frac{\nu}{\lambda}\}$, so $x_1>x_0\geq \max\{-\frac{b}{a},-\frac{\nu}{\lambda}\}$.

If $x_1<+\infty$, by
\begin{equation*}
  \int_{\frac{x_0+x_1}{2}}^{x_1}\frac{1}{(b+ax)\sqrt{\varphi(x)}}\mathrm{d}x= +\infty,
\end{equation*}
we get $\varphi(x_1)=\varphi'(x_1)=0$. Then by \eqref{eq7.3.4}, we obtain $\gamma=\frac{\mu(\nu+\lambda x_1)}{(b+ax_1)^2}$ and
\begin{equation}\label{4.30}
    \varphi(x)= \frac{(b+ax)^{2d+1}(a\lambda x+2a\nu-b\lambda)}{(\nu+\lambda x)^{d}}\int_{x_1}^x\frac{(\nu+\lambda u)^{d}\left(\mu(\nu+\lambda u)-\gamma(b+au)^2\right)}{(b+au)^{2d+2}(b\lambda-2a\nu-a\lambda u)^2}\mathrm{d}u
\end{equation}
for $x\in(x_0,x_1)$. It follows from $\varphi\neq 0$ that $\mu<0$.\\

$(\mathrm{II-ii}-1)$ If $-\frac{b}{a}\geq-\frac{\nu}{\lambda}\geq\frac{b}{a}-\frac{2\nu}{\lambda}$, then
\begin{equation*}
  a\lambda x+2a\nu-b\lambda>0, \frac{{\nu+\lambda x_1}}{(b+ax_1)^2}>0,\;{\nu+\lambda x}-\frac{{\nu+\lambda x_1}}{(b+ax_1)^2}(b+ax)^2>0
\end{equation*}
for $c:=-\frac{b}{a}\leq x_0<x<x_1<+\infty$. Thus
\begin{equation*}
  \mu(\nu+\lambda x)-\gamma(b+ax)^2=\mu\left({\nu+\lambda x}-\frac{{\nu+\lambda x_0}}{(b+ax_0)^2}(b+ax)^2\right)<0,
\end{equation*}
so $\varphi(x)>0$ for $c<x<x_1$ from \eqref{4.30}.

Using
\begin{equation*}
   \lim_{x\rightarrow c^{+}} \varphi(x)=\left\{\begin{array}{l}
                                                 \frac{-\mu}{2(2d+1)a^2}>0\;(\text{for}\;-\frac{b}{a}<-\frac{\nu}{\lambda}), \\\\
                                                 \frac{-\mu}{(d+2)a^2}>0\;(\text{for}\;-\frac{b}{a}=-\frac{\nu}{\lambda})
                                               \end{array}
   \right. \;(\text{by}\;\eqref{eq7.3.4})
\end{equation*}
and fiber metrics for $\widetilde{g}$ are complete, we have $x_0=c$.

Let
\begin{equation*}
  t=\int_{c}^{x(t)}\frac{\mathrm{d}u}{\varphi(u)},\;F(t)=\int_{c}^{x(t)}\frac{v\mathrm{d}v}{\varphi(v)},
\end{equation*}
by
\begin{equation*}
  \lim_{x\rightarrow x_1^{-}}\frac{\varphi(x)}{(x-x_1)^2}=-\frac {\mu}{2(b+ax_1)^2},
\end{equation*}
then
\begin{equation*}
  \lim_{x\rightarrow x_1^{-}}t(x)=+\infty, \;\lim_{x\rightarrow c^{+}}t(x)=0.
\end{equation*}
This shows that the metric $g_F$ is defined on $ M=\left\{(z,w)\in \Omega\times\mathbb{C}^{r}: 1< e^{\lambda\phi(z)}\|w\|^2\right\}$. Since
\begin{equation*}
\int_{x_0}^{\frac{x_0+x_1}{2}}\frac{1}{(b+ax)\sqrt{\varphi(x)}}\mathrm{d}x=+\infty,\;\int_{\frac{x_0+x_1}{2}}^{x_1}\frac{1}{(b+ax)\sqrt{\varphi(x)}}\mathrm{d}x=+\infty,
\end{equation*}
it follows that any fiber metric for $\widetilde{g}$ is complete.

For $x_1=+\infty$, from \eqref{eq7.3.4}, we obtain
\begin{equation*}
  \frac{(\nu+\lambda x)^d\varphi(x)}{(ax+b)^{2d+1}(a\lambda x+2a\nu-b\lambda)}
  =C_1+\int_{\widetilde{x_0}}^x  \frac{(\nu+\lambda u)^d\left(\mu(\nu+\lambda u)-\gamma (au+b)^2\right)}{(au+b)^{2d+2}(a\lambda u+2a\nu-b\lambda)^2}\mathrm{d}u,
\end{equation*}
where
\begin{equation*}
  \widetilde{x_0}:=x_0+1,C_1= \frac{(\nu+\lambda \widetilde{x_0})^d\varphi(\widetilde{x_0})}{(a\widetilde{x_0}+b)^{2d+1}(a\lambda \widetilde{x_0}+2a\nu-b\lambda)}.
\end{equation*}
Since
\begin{equation*}
 -\infty <C_1+\int_{\widetilde{x_0}}^{+\infty}  \frac{(\nu+\lambda u)^d\left(\mu(\nu+\lambda u)-\gamma (au+b)^2\right)}{(au+b)^{2d+2}(a\lambda u+2a\nu-b\lambda)^2}\mathrm{d}u<+\infty,
\end{equation*}
it follows that
\begin{equation*}
  \lim_{x\rightarrow+\infty} \frac{(\nu+\lambda x)^d\varphi(x)}{(ax+b)^{2d+1}(a\lambda x+2a\nu-b\lambda)}=C\geq 0.
\end{equation*}
If $C>0$, then $\varphi(x)\sim C_4x^{d+2}$ ($C_4>0$) as $x\rightarrow +\infty$, this is in contradiction with
\begin{equation*}
  \int^{+\infty}_{\widetilde{x_0}}\frac{\mathrm{d}x}{(b+ax)\sqrt{\varphi(x)}}=+\infty,
\end{equation*}
thus $C=0$, so
\begin{equation}\label{4.31}
     \varphi(x)=\frac{(b+ax)^{2d+1}(a\lambda x+2a\nu-b\lambda)}{(\nu+\lambda x)^{d}}
     \int_{+\infty}^x\frac{(\nu+\lambda u)^{d}\left(\mu(\nu+\lambda u)-\gamma(b+au)^2\right)}{(b+au)^{2d+2}(a\lambda u+2a\nu-b\lambda)^2}\mathrm{d}u
\end{equation}
for $x\in(x_0,+\infty)$.

As $x\rightarrow+\infty,$
\begin{equation*}
 \frac{(b+ax)^{2d+1}(a\lambda x+2a\nu-b\lambda)}{(\nu+\lambda x)^{d}} \sim \frac{a^{2d+2}}{\lambda^{d-1}}x^{d+2}
\end{equation*}
and
\begin{equation*}
  \frac{(\nu+\lambda x)^{d}\left(\mu(\nu+\lambda x)-\gamma(b+ax)^2\right)}{(b+ax)^{2d+2}(a\lambda x+2a\nu-b\lambda)^2}\sim\left\{\begin{array}{l}
                                                                                                                        -\frac{\lambda^{d-2}\gamma}{a^{2d+2}}\frac{1}{x^{d+2}},   \gamma\neq 0, \\\\
                                                                                                                         \frac{\lambda^{d-1}\mu}{a^{2d+4}}\frac{1}{x^{d+3}},   \gamma=0,
                                                                                                                         \end{array}
  \right.
\end{equation*}
it follows from \eqref{4.31} that
\begin{equation*}
 \lim_{x\rightarrow+\infty}\frac{\varphi(x)}{x}=\frac{\gamma}{(d+1)\lambda}\;(\gamma\neq 0)
\end{equation*}
and
\begin{equation*}
 \lim_{x\rightarrow+\infty}\varphi(x)=-\frac{\mu}{(d+2)a^2}\;(\gamma=0).
\end{equation*}
By
\begin{equation*}
  \int^{+\infty}_{\widetilde{x_0}}\frac{\mathrm{d}x}{(b+ax)\sqrt{\varphi(x)}}=+\infty,
\end{equation*}
it follows that $\gamma=0$, then $\varphi(x)>0$ for $x\in(x_0,+\infty)$ only $\mu<0$ from \eqref{4.31}.

If $x_0>c$, then
\begin{equation*}
  \int^{x_0+1}_{x_0}\frac{1}{(b+ax)\sqrt{\varphi(x)}}\mathrm{d}x<+\infty,
\end{equation*}
according to the completeness of fiber metrics for $\widetilde{g}$, we get $x_0=c$

Let
\begin{equation*}
  t=\int_{c}^{x(t)}\frac{\mathrm{d}u}{\varphi(u)},\;F(t)=\int_{c}^{x(t)}\frac{v\mathrm{d}v}{\varphi(v)},
\end{equation*}
then
\begin{equation*}
  \lim_{x\rightarrow +\infty}t(x)=+\infty, \;\lim_{x\rightarrow c^{+}}t(x)=0.
\end{equation*}
This shows that the metric $g_F$ is defined on $ M=\left\{(z,w)\in \Omega\times\mathbb{C}^{r}: 1< e^{\lambda\phi(z)}\|w\|^2\right\}$. Since
\begin{equation*}
\int^{+\infty}_{c+1}\frac{1}{(b+ax)\sqrt{\varphi(x)}}\mathrm{d}x=+\infty,\;\int^{c+1}_{c}\frac{1}{(b+ax)\sqrt{\varphi(x)}}\mathrm{d}x=+\infty,
\end{equation*}
it follows that any fiber metric of $\widetilde{g}$ is complete.\\

$(\mathrm{II-ii}-2)$ If $\frac{b}{a}-\frac{2\nu}{\lambda}>-\frac{\nu}{\lambda}>-\frac{b}{a}$, then $ x_1>x_0\geq-\frac{\nu}{\lambda}>-\frac{b}{a}$. We prove that this situation is impossible.

If $\frac{b}{a}-\frac{2\nu}{\lambda}= x_1>x_0\geq-\frac{\nu}{\lambda}>-\frac{b}{a}$ , by $\varphi(x_1)=\varphi'(x_1)=0$ and $\eqref{eq7.3.4}$, we get
\begin{equation*}
  \gamma=-{\frac {\mu\,{\lambda}^{2}}{4a \left( a\nu-b\lambda \right) }},\;\mu(\nu+\lambda x)-\gamma(b+ax)^2=-a^2\gamma(x-x_1)^2.
\end{equation*}
Thus
\begin{equation*}
  \varphi(x)= \frac{a\gamma}{\lambda}\frac{(b+ax)^{2d+1}(x_1-x)}{(\nu+\lambda x)^{d}}\int_{x_1}^x\frac{(\nu+\lambda u)^{d}}{(b+au)^{2d+2}}\mathrm{d}u
\end{equation*}
and
\begin{equation*}
  \int^{\frac{x_0+x_1}{2}}_{x_0}\frac{1}{(b+ax)\sqrt{\varphi(x)}}\mathrm{d}x<+\infty.
\end{equation*}
This is in contradiction with the completeness of fiber metrics of $\widetilde{g}$.

If $-\frac{b}{a}<-\frac{\nu}{\lambda}\leq x_0<x_1<+\infty$ and $x_1\neq \frac{b}{a}-\frac{2\nu}{\lambda}$,   by the completeness of fiber metrics of $\widetilde{g}$, we get $\varphi(x_0)=\varphi'(x_0)=\varphi(x_1)=\varphi'(x_1)=0$, so
\begin{equation*}
 \gamma=\frac{\mu(\nu+\lambda x_0)}{(b+ax_0)^2}=\frac{\mu(\nu+\lambda x_1)}{(b+ax_1)^2}\leq 0,
\end{equation*}
\begin{equation*}
  \mu(\nu+\lambda x)-\gamma(ax+b)^2=a^2\gamma(x-x_0)(x_1-x)
\end{equation*}
and
\begin{equation*}
\psi(x):= d\gamma(ax+b)^2-(d+1)\mu(\nu+\lambda x)=-da^2\gamma(x-x_0)(x_1-x)-\mu(\nu+\lambda x)\geq 0
\end{equation*}
for $x_0<x<x_1$. Using \eqref{eq7.1}, we have
\begin{equation*}
  \varphi(x)=\frac{(b+ax)^{2d+1}}{(\nu+\lambda x)^d}\int^x_{x_1}(x-u)\frac{(\nu+\lambda u)^{d-1}\psi(u)}{(b+au)^{2d+3}}\mathrm{d}u \;(\text{by}\;\varphi(x_1)=0\;\text{and}\;\varphi'(x_1)=0),
\end{equation*}
which implies that $\varphi(x_0)>0$, this is in contradiction with the completeness of $\widetilde{g}$.

If $x_1=+\infty$, by $\eqref{4.31}$, we have $\gamma=0$ and $\mu<0$. By $\varphi(x_0)=0, \varphi'(x_0)=0$ and \eqref{eq7.1}, we have
\begin{equation*}
  \varphi(x)=\frac{(b+ax)^{2d+1}}{(\nu+\lambda x)^d}\int^x_{x_0}(x-u)\frac{(\nu+\lambda u)^{d-1}\psi(u)}{(b+au)^{2d+3}}\mathrm{d}u
\end{equation*}
which implies that $\varphi(x)\sim C_5x^{d+2}$ as $x\rightarrow+\infty$, this is in contradiction with the completeness of $\widetilde{g}$.\\

$(\mathrm{II-iii})$ For the case of $a>0$ and $\lambda<0$.\\

By $ax+b>0$ and $\nu+\lambda x>0$, we get $-\frac{b}{a}<x<-\frac{\nu}{\lambda}$, so $-\frac{b}{a}\leq x_0<x_1\leq-\frac{\nu}{\lambda}<\frac{b}{a}-\frac{2\nu}{\lambda}$.

By
\begin{equation*}
  \int_{\frac{x_0+x_1}{2}}^{x_1}\frac{1}{(b+ax)\sqrt{\varphi(x)}}\mathrm{d}x= +\infty,
\end{equation*}
we get $\varphi(x_1)=\varphi'(x_1)=0$. Then by \eqref{eq7.3.4}, we obtain $\gamma=\frac{\mu(\nu+\lambda x_1)}{(b+ax_1)^2}$ and
\begin{equation}\label{4.32}
    \varphi(x)= \frac{(b+ax)^{2d+1}(a\lambda x+2a\nu-b\lambda)}{(\nu+\lambda x)^{d}}\int_{x_1}^x\frac{(\nu+\lambda u)^{d}\left(\mu(\nu+\lambda u)-\gamma(b+au)^2\right)}{(b+au)^{2d+2}(b\lambda-2a\nu-a\lambda u)^2}\mathrm{d}u
\end{equation}
for $x\in(x_0,x_1)$. From \eqref{4.32} and $\varphi\neq 0$, it follows that $\mu<0$.

Note that when $c:=-\frac{b}{a}\leq x_0<x<x_1$,
\begin{equation*}
  a\lambda x+2a\nu-b\lambda>0, \frac{{\nu+\lambda x_1}}{(b+ax_1)^2}\geq 0,\;{\nu+\lambda x}-\frac{{\nu+\lambda x_1}}{(b+ax_1)^2}(b+ax)^2>0,
\end{equation*}
 so
\begin{equation*}
  \mu(\nu+\lambda x)-\gamma(b+ax)^2=\mu\left({\nu+\lambda x}-\frac{{\nu+\lambda x_0}}{(b+ax_0)^2}(b+ax)^2\right)<0,
\end{equation*}
 thus  $\varphi(x)>0$ for $c<x<x_1$ by \eqref{4.32}.

Using
\begin{equation*}
   \lim_{x\rightarrow c^{+}} \varphi(x)= \frac{-\mu}{2(2d+1)a^2}>0
\end{equation*}
and fiber metrics for $\widetilde{g}$ are complete, we have $x_0=c$.

Let
\begin{equation*}
  t=\int_{c}^{x(t)}\frac{\mathrm{d}u}{\varphi(u)},\;F(t)=\int_{c}^{x(t)}\frac{v\mathrm{d}v}{\varphi(v)},
\end{equation*}
by
\begin{equation*}
  \lim_{x\rightarrow x_1^{-}}\frac{\varphi(x)}{(x-x_1)^2}=-\frac {\mu}{2(b+ax_1)^2},
\end{equation*}
then
\begin{equation*}
  \lim_{x\rightarrow x_1^{-}}t(x)=+\infty, \;\lim_{x\rightarrow c^{+}}t(x)=0.
\end{equation*}
This shows that the metric $g_F$ is defined on $ M=\left\{(z,w)\in \Omega\times\mathbb{C}^{r}: 1< e^{\lambda\phi(z)}\|w\|^2\right\}$. Since
\begin{equation*}
\int_{x_0}^{\frac{x_0+x_1}{2}}\frac{1}{(b+ax)\sqrt{\varphi(x)}}\mathrm{d}x=+\infty,\;\int_{\frac{x_0+x_1}{2}}^{x_1}\frac{1}{(b+ax)\sqrt{\varphi(x)}}\mathrm{d}x=+\infty,
\end{equation*}
it follows that any fiber metric for $\widetilde{g}$ is complete.\\

$(\mathrm{II-iv})$ For the case of $a<0$ and $\lambda>0$.\\

By $ax+b>0$ and $\nu+\lambda x>0$, we get $-\frac{\nu}{\lambda}<x<-\frac{b}{a}$, so $\frac{b}{a}-\frac{2\nu}{\lambda}<-\frac{\nu}{\lambda}\leq x_0<x_1\leq -\frac{b}{a}$.

 By
\begin{equation*}
  \int^{\frac{x_0+x_1}{2}}_{x_0}\frac{1}{(b+ax)\sqrt{\varphi(x)}}\mathrm{d}x= +\infty,
\end{equation*}
we get $\varphi(x_0)=\varphi'(x_0)=0$. Then by \eqref{eq7.3.4}, we obtain $\gamma=\frac{\mu(\nu+\lambda x_0)}{(b+ax_0)^2}$ and
\begin{equation}\label{4.32.1}
    \varphi(x)= \frac{(b+ax)^{2d+1}(a\lambda x+2a\nu-b\lambda)}{(\nu+\lambda x)^{d}}\int_{x_0}^x\frac{(\nu+\lambda u)^{d}\left(\mu(\nu+\lambda u)-\gamma(b+au)^2\right)}{(b+au)^{2d+2}(b\lambda-2a\nu-a\lambda u)^2}\mathrm{d}u
\end{equation}
for $x\in(x_0,x_1)$. From \eqref{4.32.1} and $\varphi\neq 0$, it follows that $\mu<0$.

Note that when $x_0<x< c:=-\frac{b}{a}$,
\begin{equation*}
  b\lambda-2a\nu-a\lambda x>0, \frac{{\nu+\lambda x_0}}{(b+ax_0)^2}\geq 0,\;{\nu+\lambda x}-\frac{{\nu+\lambda x_0}}{(b+ax_0)^2}(b+ax)^2>0,
\end{equation*}
 so
\begin{equation*}
  \mu(\nu+\lambda x)-\gamma(b+ax)^2=\mu\left({\nu+\lambda x}-\frac{{\nu+\lambda x_0}}{(b+ax_0)^2}(b+ax)^2\right)<0
\end{equation*}
which implies that  $\varphi(x)>0$ for $x_0<x<c$.

Using
\begin{equation*}
   \lim_{x\rightarrow c^{-}} \varphi(x)= \frac{-\mu}{2(2d+1)a^2}>0
\end{equation*}
and fiber metrics for $\widetilde{g}$ are complete, we have $x_1=c$.

Let
\begin{equation*}
  t=\int_{c}^{x(t)}\frac{\mathrm{d}u}{\varphi(u)},\;F(t)=\int_{c}^{x(t)}\frac{v\mathrm{d}v}{\varphi(v)},
\end{equation*}
by
\begin{equation*}
  \lim_{x\rightarrow x_0^{+}}\frac{\varphi(x)}{(x-x_0)^2}=-\frac {\mu}{2(b+ax_0)^2},
\end{equation*}
then
\begin{equation*}
  \lim_{x\rightarrow x_0^{+}}t(x)=-\infty, \;\lim_{x\rightarrow c^{-}}t(x)=0.
\end{equation*}
This shows that the metric $g_F$ is defined on $ M=\left\{(z,w)\in \Omega\times\mathbb{C}^{r}: 0< e^{\lambda\phi(z)}\|w\|^2<1\right\}$. Since
\begin{equation*}
\int_{x_0}^{\frac{x_0+x_1}{2}}\frac{1}{(b+ax)\sqrt{\varphi(x)}}\mathrm{d}x=+\infty,\;\int_{\frac{x_0+x_1}{2}}^{x_1}\frac{1}{(b+ax)\sqrt{\varphi(x)}}\mathrm{d}x=+\infty,
\end{equation*}
it follows that any fiber metric for $\widetilde{g}$ is complete.\\

\noindent\textbf{Part $(\mathrm{II})$ for $r>1$}\\

Since $a\neq 0$, from \eqref{eq7.3.5}, we get $\nu=0$.

By Remark \ref{Re:2.1}, when $\nu=0$ and $r>1$,  there must be $\lambda>0$.\\

$(\mathrm{II-i})$  For the case of $a>0$ and $\lambda>0$.\\

By $ax+b>0$ and $\nu+\lambda x=\lambda x>0$, we get $x>\max\{-\frac{b}{a},0\}$, so $x_1>x_0\geq \max\{-\frac{b}{a},0\}$.

If $x_1<+\infty$, by
\begin{equation*}
  \int_{\frac{x_0+x_1}{2}}^{x_1}\frac{1}{(b+ax)\sqrt{\varphi(x)}}\mathrm{d}x= +\infty,
\end{equation*}
we get $\varphi(x_1)=\varphi'(x_1)=0$. Then by \eqref{eq7.3.5},  we obtain $\mu=\frac{r(ax_1+b)^2}{x_1}>0$. This is in contradiction with $\mu\leq 0$.

For $x_1=+\infty$, from \eqref{eq7.3.5}, we obtain
\begin{equation*}
  \frac{x^{n-1}\varphi(x)}{(ax+b)^{2n-1}(ax-b)}
  =C_1+\int_{\widetilde{x_0}}^x  \frac{ u^{n-1}\left(\mu u-r(b+au)^2\right)}{(b+au)^{2n}(au-b)^2}\mathrm{d}u,
\end{equation*}
where
\begin{equation*}
  \widetilde{x_0}:=x_0+1,C_1= \frac{ \widetilde{x_0}^{n-1}\varphi(\widetilde{x_0})}{(a\widetilde{x_0}+b)^{2n-1}(a \widetilde{x_0}-b)}.
\end{equation*}
Since
\begin{equation*}
 -\infty <C_1+\int_{\widetilde{x_0}}^{+\infty}  \frac{ u^{n-1}\left(\mu u-r(b+au)^2\right)}{(b+au)^{2n}(au-b)^2}\mathrm{d}u<+\infty,
\end{equation*}
 it follows that
\begin{equation*}
  \lim_{x\rightarrow+\infty} \frac{x^{n-1}\varphi(x)}{(ax+b)^{2n-1}(ax-b)}=C.
\end{equation*}
If $C\neq 0$, then $\varphi(x)\sim C_2x^{n+1}$ as $x\rightarrow +\infty$, this is in contradiction with
\begin{equation*}
  \int^{+\infty}_{\widetilde{x_0}}\frac{\mathrm{d}x}{(b+ax)\sqrt{\varphi(x)}}=+\infty,
\end{equation*}
thus $C=0$, so
\begin{equation}\label{IV-ii-1}
    \varphi(x)= \frac{(ax+b)^{2n-1}(ax-b)}{x^{n-1}}\int_{+\infty}^x \frac{ u^{n-1}\left(\mu u-r(b+au)^2\right)}{(b+au)^{2n}(au-b)^2}\mathrm{d}u
\end{equation}
for $x\in(x_0,+\infty)$.

As $x\rightarrow+\infty,$
\begin{equation*}
 \frac{(ax+b)^{2n-1}(ax-b)}{x^{n-1}} \sim a^{2n}x^{n+1}
\end{equation*}
and
\begin{equation*}
 \frac{ x^{n-1}\left(\mu x-r(b+ax)^2\right)}{(b+ax)^{2n}(ax-b)^2}\sim \frac{-r}{a^{2n}}\frac{1}{x^{n+1}}
\end{equation*}
it follows from \eqref{IV-ii-1} that
\begin{equation*}
 \lim_{x\rightarrow+\infty}\frac{\varphi(x)}{x}=\frac{r}{n}.
\end{equation*}
So
\begin{equation*}
  \int^{+\infty}_{\widetilde{x_0}}\frac{\mathrm{d}x}{(b+ax)\sqrt{\varphi(x)}}<+\infty,
\end{equation*}
this is in contradiction with the completeness of fiber metrics of $\widetilde{g}$.\\

$(\mathrm{II-ii})$  For the case of $a<0$ and $\lambda>0$.\\

By $ax+b>0$ and $\nu+\lambda x=\lambda x>0$, we get $0<x<-\frac{b}{a}$, so $b>0$ and $\frac{b}{a}<0\leq x_0<x_1\leq -\frac{b}{a}$.

 By
\begin{equation*}
  \int^{\frac{x_0+x_1}{2}}_{x_0}\frac{1}{(b+ax)\sqrt{\varphi(x)}}\mathrm{d}x= +\infty,
\end{equation*}
we get $\varphi(x_0)=\varphi'(x_0)=0$. Then by \eqref{eq7.3.5}, we obtain $\mu=\frac{r(ax_0+b)^2}{x_0}>0$, this is in contradiction with $\mu\leq 0$.\\

\begin{center}
\noindent \textbf{Part $(\mathrm{III})$ }\\
\end{center}

By \eqref{eq7.3.4}, using $\varphi(x_1)=0$,  we get
\begin{equation}\label{4.39}
 \varphi(x)=\frac{(a\lambda x+2a\nu-b\lambda)(ax+b)^{2d+1}}{(\nu+\lambda x)^d}\int_{x_1}^x\frac{(\nu+\lambda u)^d\left(\mu(\nu+\lambda u)-\gamma(au+b)^2\right)}{(au+b)^{2d+2}(a\lambda u+2a\nu-b\lambda)^2}\mathrm{d}u
\end{equation}
for $x_0<x<x_1$. Then
\begin{equation*}
 -1= \varphi'(x_1)=\lim_{x\rightarrow x_1^{-}}\frac{\varphi(x)}{x-x_1}=\frac {\mu \left(\lambda x_1+\nu \right) -\gamma (ax_1+b)^{2}}{(ax_1+b)(a\lambda x_1+2a\nu-b\lambda) }.
\end{equation*}
So
\begin{equation*}
 \gamma=\mu\frac{\nu+\lambda x_1}{(b+ax_1)^2}+\frac{a\lambda x_1+2a\nu-b\lambda}{b+ax_1}.
\end{equation*}

$(\mathrm{III-i})$ For the case of $a<0$ and $\lambda<0$.\\

By $ax+b>0$ and $\nu+\lambda x>0$, we get $x<\min\{-\frac{b}{a},-\frac{\nu}{\lambda}\}$, so $x_1\leq \min\{-\frac{b}{a},-\frac{\nu}{\lambda}\}$.

For $x_0>-\infty$, according to the completeness of the fibre metric for $\widetilde{g}$, namely
\begin{equation*}
  \int_{x_0}^{\frac{x_0+x_1}{2}}\frac{1}{(b+ax)\sqrt{\varphi(x)}}\mathrm{d}x=+\infty,
\end{equation*}
we have $\varphi(x_0)=\varphi'(x_0)=0$. Then by \eqref{eq7.3.4}, we obtain $\gamma=\frac{\mu(\nu+\lambda x_0)}{(b+ax_0)^2}<0$, $\mu<0$ and
\begin{equation}\label{4.40}
    \varphi(x)= -\frac{(b+ax)^{2d+1}(b\lambda-2a\nu-a\lambda x)}{(\nu+\lambda x)^{d}}\int_{x_0}^x\frac{(\nu+\lambda u)^{d}\left(\mu(\nu+\lambda u)-\gamma(b+au)^2\right)}{(b+au)^{2d+2}(b\lambda-2a\nu-a\lambda u)^2}\mathrm{d}u
\end{equation}
for $x\in(x_0,x_1)$.\\

$(\mathrm{III-i}-1)$ If $-\frac{b}{a}\leq-\frac{\nu}{\lambda}\leq\frac{b}{a}-\frac{2\nu}{\lambda}$.

Note that when $-\infty<x_0<x<x_1\leq c:=-\frac{b}{a}$, by
\begin{equation*}
  b\lambda-2a\nu-a\lambda x>0,\frac{{\nu+\lambda x_0}}{(b+ax_0)^2}>0,\;{\nu+\lambda x}-\frac{{\nu+\lambda x_0}}{(b+ax_0)^2}(b+ax)^2>0
\end{equation*}
and \eqref{4.40}, we get  $\varphi(x_1)>0$. This is in contradiction with $\varphi(x_1)=0$.

For $x_0=-\infty$, according to the proof of $(\mathrm{II-i}-1)$, we get $\gamma=0$, $\mu<0$ and
\begin{equation}\label{4.41}
     \varphi(x)=\frac{(b+ax)^{2d+1}(a\lambda x+2a\nu-b\lambda)}{(\nu+\lambda x)^{d}}
     \int_{-\infty}^x\frac{\mu(\nu+\lambda u)^{d+1}}{(b+au)^{2d+2}(a\lambda u+2a\nu-b\lambda)^2}\mathrm{d}u
\end{equation}
for $x\in(-\infty,x_1)$, which implies that $\varphi(x_1)>0$. This conflicts with $\varphi(x_1)=0$.\\

$(\mathrm{III-i}-2)$ If $\frac{b}{a}-\frac{2\nu}{\lambda}<-\frac{\nu}{\lambda}<-\frac{b}{a}$, then $\frac{b}{a}-\frac{2\nu}{\lambda}\leq x_0<x_1\leq-\frac{\nu}{\lambda}<-\frac{b}{a}$ or $ x_0<x_1\leq\frac{b}{a}-\frac{2\nu}{\lambda}<-\frac{\nu}{\lambda}<-\frac{b}{a}$ or $x_0<\frac{b}{a}-\frac{2\nu}{\lambda} <x_1\leq-\frac{\nu}{\lambda}<-\frac{b}{a}$.

If $-\infty< x_0$, and $\frac{b}{a}-\frac{2\nu}{\lambda}<x_1\leq-\frac{\nu}{\lambda}<-\frac{b}{a}$,   by
\begin{equation*}
 \gamma=\mu\frac{\nu+\lambda x_1}{(b+ax_1)^2}+\frac{a\lambda x_1+2a\nu-b\lambda}{b+ax_1}=\mu\frac{\nu+\lambda x_0}{(b+ax_0)^2}< 0,
\end{equation*}
we get
\begin{equation*}
  \mu(\nu+\lambda x)-\gamma(ax+b)^2<0
\end{equation*}
for $x_0<x<x_1$. So
\begin{equation*}
\psi(x):= d\gamma(ax+b)^2-(d+1)\mu(\nu+\lambda x)=-d(\mu(\nu+\lambda x)-\gamma(ax+b)^2)-\mu(\nu+\lambda x)> 0
\end{equation*}
for $x_0<x<x_1$. Using \eqref{eq7.1}, we have
\begin{equation*}
  \varphi(x)=\frac{(b+ax)^{2d+1}}{(\nu+\lambda x)^d}\int^x_{x_0}(x-u)\frac{(\nu+\lambda u)^{d-1}\psi(u)}{(b+au)^{2d+3}}\mathrm{d}u \;(\text{by}\;\varphi(x_0)=0\;\text{and}\;\varphi'(x_0)=0),
\end{equation*}
which implies that $\varphi(x_1)>0$, this conflicts with $\varphi(x_1)=0$.

If $-\infty< x_0<x_1\leq c:=\frac{b}{a} -\frac{2\nu}{\lambda}<-\frac{\nu}{\lambda}<-\frac{b}{a}$, by
\begin{equation*}
 \gamma=\frac{\mu(\nu+\lambda x_0)}{(b+ax_0)^2}<0, \mu(\nu+\lambda c)-\gamma(ac+b)^2=-{\frac {\mu\, \left( \nu\,a-b\lambda \right)  \left( a\lambda\,x_{{0}}+2\,\nu\,a-b\lambda \right) ^{2}}{a \left( ax_{{0}}+b \right) ^{2}{\lambda}^{2}}}<0
\end{equation*}
and
\begin{equation*}
   b\lambda-2a\nu-a\lambda x>0,\mu(\nu+\lambda x)-\gamma(ax+b)^2<0
\end{equation*}
for $x_0<x<c$. From \eqref{4.40} we have $\varphi(x_1)>0$, this conflicts with $\varphi(x_1)=0$.

If $-\infty= x_0<x_1\leq -\frac{\nu}{\lambda}<-\frac{b}{a}$, by $\eqref{2.20.1.1}$, we have $\gamma=0$ and $\mu<0$. Using
\begin{equation*}
  \gamma=\mu\frac{\nu+\lambda x_1}{(b+ax_1)^2}+\frac{a\lambda x_1+2a\nu-b\lambda}{b+ax_1},
\end{equation*}
it follows that $\frac{b}{a}-\frac{2\nu}{\lambda} <x_1$. From \eqref{4.39}, we get $\varphi(x)<0$ for $x<\frac{b}{a}-\frac{2\nu}{\lambda}$, this conflicts with $\varphi(x)>0$ for $x_0<x<x_1$. \\

$(\mathrm{III-ii})$ For the case of $a>0$ and $\lambda>0$.\\

By $ax+b>0$ and $\nu+\lambda x>0$, we get $x>\max\{-\frac{b}{a},-\frac{\nu}{\lambda}\}$, so $x_1>x_0\geq \max\{-\frac{b}{a},-\frac{\nu}{\lambda}\}$,
\begin{equation}\label{4.42}
    \varphi(x)= \frac{(b+ax)^{2d+1}(a\lambda x+2a\nu-b\lambda)}{(\nu+\lambda x)^{d}}\int_{x_1}^x\frac{(\nu+\lambda u)^{d}\left(\mu(\nu+\lambda u)-\gamma(b+au)^2\right)}{(b+au)^{2d+2}(b\lambda-2a\nu-a\lambda u)^2}\mathrm{d}u
\end{equation}
for $x\in(x_0,x_1)$, and
\begin{equation*}
 \gamma=\mu\frac{\nu+\lambda x_1}{(b+ax_1)^2}+\frac{a\lambda x_1+2a\nu-b\lambda}{b+ax_1}.
\end{equation*}

$(\mathrm{III-ii}-1)$ If $-\frac{b}{a}\geq-\frac{\nu}{\lambda}\geq\frac{b}{a}-\frac{2\nu}{\lambda}$.

Let
\begin{equation*}
  q(x)=\mu(\nu+\lambda x)-\gamma(b+ax)^2.
\end{equation*}

Note that when $c:=-\frac{b}{a}< x<x_1$, by $\mu\leq 0$, $q(-\frac{b}{a})\leq 0$ and $q(x_1)<0$, we get
\begin{equation*}
  a\lambda x+2a\nu-b\lambda>0,
\end{equation*}
\begin{equation*}
 \mu(\nu+\lambda x)-\gamma(b+ax)^2\leq 0\;(\gamma\leq 0)
\end{equation*}
and
\begin{equation*}
\mu(\nu+\lambda x)-\gamma(b+ax)^2\leq -\gamma(b+ax)^2\leq 0\;(\gamma> 0),
\end{equation*}
 so $\varphi(x)>0$ for $x\in(c,x_1)$ by \eqref{4.42}.

Using
\begin{equation*}
   \lim_{x\rightarrow c^{+}} \varphi(x)=\left\{\begin{array}{l}
                                                 \frac{-\mu}{2(2d+1)a^2}\;(\text{for}\;\frac{b}{a}<\frac{\nu}{\lambda}), \\\\
                                                 \frac{-\mu}{(d+2)a^2}\;(\text{for}\;\frac{b}{a}=\frac{\nu}{\lambda}),                                               \end{array}
   \right.
\end{equation*}
\begin{equation*}
  \lim_{x\rightarrow c^{+}} \frac{\varphi(x)}{(b+ax)^2}= \frac{\gamma}{2(2d-1)a(a\nu-b\lambda)}>0\;(\text{for}\;\mu=0\;\mathrm{and}\;\frac{b}{a}<\frac{\nu}{\lambda}),
\end{equation*}
\begin{equation*}
  \lim_{x\rightarrow c^{+}} \frac{\varphi(x)}{b+ax}=\frac{\gamma}{(d+1)a\lambda}>0\;(\text{for}\;\mu=0\;\mathrm{and}\;\frac{b}{a}=\frac{\nu}{\lambda})
\end{equation*}
and fiber metrics for $\widetilde{g}$ are complete, we have $x_0=c$.

For $\mu<0$, let
\begin{equation*}
  t=\int_{c}^{x(t)}\frac{\mathrm{d}u}{\varphi(u)},\;F(t)=\int_{c}^{x(t)}\frac{v\mathrm{d}v}{\varphi(v)},
\end{equation*}
by
\begin{equation*}
 \lim_{x\rightarrow c^{+}} \varphi(x)= \frac{-\mu}{2(2d+1)a^2}\;(\text{for}\;\frac{b}{a}<\frac{\nu}{\lambda}),
   \;\;\lim_{x\rightarrow c^{+}} \varphi(x)= \frac{-\mu}{(d+2)a^2}\;(\text{for}\;\frac{b}{a}=\frac{\nu}{\lambda})
\end{equation*}
and $\lim_{x\rightarrow x_1^{-}}\frac{\varphi(x)}{x_1-x}=1$, we have
\begin{equation*}
  \lim_{x\rightarrow x_1^{-}}t(x)=+\infty, \;\lim_{x\rightarrow c^{+}}t(x)=0.
\end{equation*}
This shows that the metric $g_F$ is defined on $ M=\left\{(z,w)\in \Omega\times\mathbb{C}^{r}: 1< e^{\lambda\phi(z)}\|w\|^2\right\}$. Since
\begin{equation*}
\int_{x_0}^{\frac{x_0+x_1}{2}}\frac{1}{(b+ax)\sqrt{\varphi(x)}}\mathrm{d}x=+\infty,
\end{equation*}
it follows that any fiber metric for $\widetilde{g}$ is complete at $x=x_0$.

For $\mu=0$, let
\begin{equation*}
  t=\int_{\frac{x_0+x_1}{2}}^{x(t)}\frac{\mathrm{d}u}{\varphi(u)},\;F(t)=\int_{\frac{x_0+x_1}{2}}^{x(t)}\frac{v\mathrm{d}v}{\varphi(v)},
\end{equation*}
by
\begin{equation*}
 \lim_{x\rightarrow c^{+}} \frac{\varphi(x)}{(b+ax)^2}=\frac{\gamma}{2(2d-1)a(a\nu-b\lambda)}>0\;(\text{for}\;\frac{b}{a}<\frac{\nu}{\lambda}),
   \lim_{x\rightarrow x_1^{-}}\frac{\varphi(x)}{x_1-x}=1
\end{equation*}
and
\begin{equation*}
   \lim_{x\rightarrow c^{+}} \frac{\varphi(x)}{b+ax}= \frac{\gamma}{(d+1)a\lambda}>0\;(\text{for}\;\frac{b}{a}=\frac{\nu}{\lambda}),
\end{equation*}
we get
\begin{equation*}
  \lim_{x\rightarrow x_1^{-}}t(x)=+\infty, \;\lim_{x\rightarrow c^{+}}t(x)=-\infty.
\end{equation*}
This shows that the metric $g_F$ is defined on $ M=\left\{(z,w)\in \Omega\times\mathbb{C}^{r}: 0< e^{\lambda\phi(z)}\|w\|^2\right\}$. Since
\begin{equation*}
\int_{x_0}^{\frac{x_0+x_1}{2}}\frac{1}{(b+ax)\sqrt{\varphi(x)}}\mathrm{d}x=+\infty,
\end{equation*}
it follows that any fiber metric for $\widetilde{g}$ is complete at $x=x_0$.\\

$(\mathrm{III-ii}-2)$ If $\frac{b}{a}-\frac{2\nu}{\lambda}>-\frac{\nu}{\lambda}>-\frac{b}{a}$, then $\frac{b}{a}-\frac{2\nu}{\lambda}\geq x_1>x_0\geq-\frac{\nu}{\lambda}>-\frac{b}{a}$ or $x_1>x_0\geq\frac{b}{a}-\frac{2\nu}{\lambda}>-\frac{\nu}{\lambda}>-\frac{b}{a}$ or $x_1>\frac{b}{a}-\frac{2\nu}{\lambda}>x_0\geq-\frac{\nu}{\lambda}>-\frac{b}{a}$.

If $\frac{b}{a}-\frac{2\nu}{\lambda}= x_1>x_0\geq-\frac{\nu}{\lambda}>-\frac{b}{a}$ , by $\mu\leq 0$, $\varphi(x_1)=0,\varphi'(x_1)=-1$ and $\eqref{eq7.3.4}$, we get
\begin{equation*}
  \gamma={\frac {\mu\,{\lambda}^{2}}{4a \left( b\lambda-a\nu \right) }}\leq 0,\;\mu(\nu+\lambda x)-\gamma(b+ax)^2=-a^2\gamma(x-x_1)^2.
\end{equation*}
Thus
\begin{equation*}
  \varphi(x)=\frac{(b+ax)^{2d+1}(x_1-x)}{(\nu+\lambda x)^{d}}\left (\frac{(\nu+\lambda x_1)^{d}}{(b+ax_1)^{2d+1}}+\frac{a\gamma}{\lambda}\int_{x_1}^x\frac{(\nu+\lambda u)^{d}}{(b+au)^{2d+2}}\mathrm{d}u\right)>0\;(x_0\leq x<x_1)
\end{equation*}
and
\begin{equation*}
  \int^{\frac{x_0+x_1}{2}}_{x_0}\frac{1}{(b+ax)\sqrt{\varphi(x)}}\mathrm{d}x<+\infty.
\end{equation*}
This is in contradiction with the completeness of fiber metrics of $\widetilde{g}$.

If $\frac{b}{a}-\frac{2\nu}{\lambda}> x_1>x_0\geq-\frac{\nu}{\lambda}>-\frac{b}{a}$ , then
\begin{equation*}
\gamma=\mu\frac{\nu+\lambda x_1}{(b+ax_1)^2}+\frac{a\lambda x_1+2a\nu-b\lambda}{b+ax_1}, \; b\lambda-2a\nu-a\lambda x>0
\end{equation*}
for $x\leq x_1$. By the completeness of fiber metrics of $\widetilde{g}$, we get $\varphi(x_0)=\varphi'(x_0)=0$, thus by $\eqref{eq7.3.4}$,
\begin{equation*}
  \gamma=\mu\frac{\nu+\lambda x_0}{(b+ax_0)^2}=\mu\frac{\nu+\lambda x_1}{(b+ax_1)^2}+\frac{a\lambda x_1+2a\nu-b\lambda}{b+ax_1},
\end{equation*}
namely,
\begin{equation*}
  \mu\frac{\nu+\lambda x_0}{(b+ax_0)^2}-\mu\frac{\nu+\lambda x_1}{(b+ax_1)^2}=\frac{a\lambda x_1+2a\nu-b\lambda}{b+ax_1}.
\end{equation*}
Since
\begin{equation*}
  \frac{\mathrm{d}}{\mathrm{dx}}\frac{\nu+\lambda x}{(b+ax)^2}=\frac{b\lambda-2a\nu-a\lambda x}{(ax+b)^3}>0
\end{equation*}
for $x_0\leq x\leq x_1$ and $\mu\leq 0$, it follows that
\begin{equation*}
   \mu\frac{\nu+\lambda x_0}{(b+ax_0)^2}-\mu\frac{\nu+\lambda x_1}{(b+ax_1)^2}\geq 0,
\end{equation*}
this conflicts with $\frac{a\lambda x_1+2a\nu-b\lambda}{b+ax_1}<0$.

If $-\frac{b}{a}<-\frac{\nu}{\lambda}\leq x_0$ and $\frac{b}{a}-\frac{2\nu}{\lambda}<x_1$,   by
\begin{equation*}
 \gamma=\mu\frac{\nu+\lambda x_1}{(b+ax_1)^2}+\frac{a\lambda x_1+2a\nu-b\lambda}{b+ax_1}=\mu\frac{\nu+\lambda x_0}{(b+ax_0)^2}< 0,
\end{equation*}
we get
\begin{equation*}
  \mu(\nu+\lambda x)-\gamma(ax+b)^2<0
\end{equation*}
for $x_0<x<x_1$. So
\begin{equation*}
\psi(x):= d\gamma(ax+b)^2-(d+1)\mu(\nu+\lambda x)=-d(\mu(\nu+\lambda x)-\gamma(ax+b)^2)-\mu(\nu+\lambda x)> 0
\end{equation*}
for $x_0<x<x_1$. Using \eqref{eq7.1}, we have
\begin{equation*}
  \varphi(x)=\frac{(b+ax)^{2d+1}}{(\nu+\lambda x)^d}\int^x_{x_0}(x-u)\frac{(\nu+\lambda u)^{d-1}\psi(u)}{(b+au)^{2d+3}}\mathrm{d}u \;(\text{by}\;\varphi(x_0)=0\;\text{and}\;\varphi'(x_0)=0),
\end{equation*}
which implies that $\varphi(x_1)>0$, this conflicts with $\varphi(x_1)=0$.\\

$(\mathrm{III-iii})$ For the case of $a>0$ and $\lambda<0$.\\

By $ax+b>0$ and $\nu+\lambda x>0$, we get $-\frac{b}{a}<x<-\frac{\nu}{\lambda}$, so $-\frac{b}{a}\leq x_0<x_1\leq-\frac{\nu}{\lambda}<\frac{b}{a}-\frac{2\nu}{\lambda}$.

Let
\begin{equation*}
  q(x):=\mu(\nu+\lambda x)-\gamma(b+ax)^2.
\end{equation*}

Note that when $c:=-\frac{b}{a}\leq x_0<x<x_1$, by
\begin{equation*}
 \gamma=\mu\frac{\nu+\lambda x_1}{(b+ax_1)^2}+\frac{a\lambda x_1+2a\nu-b\lambda}{b+ax_1}, \;q(c)\leq 0,\;q(x_1)<0,\;\mu\leq 0,
\end{equation*}
 we have
\begin{equation*}
  q(x)=\mu(\nu+\lambda x)-\gamma(b+ax)^2< 0,\; a\lambda x+2a\nu-b\lambda>0,
\end{equation*}
for $x\in(c,x_1)$. From \eqref{4.39}, so $\varphi(x)>0$ for $x\in(c,x_1)$.

Using
\begin{equation*}
   \lim_{x\rightarrow c^{+}} \varphi(x)= \frac{-\mu}{2(2d+1)a^2}>0\;(\mu<0),
\end{equation*}
\begin{equation*}
  \lim_{x\rightarrow c^{+}} \frac{\varphi(x)}{(b+ax)^2}= \frac{\gamma}{2(2d-1)a(a\nu-b\lambda)}>0\;(\mu=0)
\end{equation*}
and fiber metrics for $\widetilde{g}$ are complete, we have $x_0=c$.

For $\mu<0$, let
\begin{equation*}
  t=\int_{c}^{x(t)}\frac{\mathrm{d}u}{\varphi(u)},\;F(t)=\int_{c}^{x(t)}\frac{v\mathrm{d}v}{\varphi(v)},
\end{equation*}
by
\begin{equation*}
  \lim_{x\rightarrow x_1^{-}}\frac{\varphi(x)}{x_1-x}=1,
\end{equation*}
then
\begin{equation*}
  \lim_{x\rightarrow x_1^{-}}t(x)=+\infty, \;\lim_{x\rightarrow c^{+}}t(x)=0.
\end{equation*}
This shows that the metric $g_F$ is defined on $ M=\left\{(z,w)\in \Omega\times\mathbb{C}^{r}: 1< e^{\lambda\phi(z)}\|w\|^2\right\}$. Since
\begin{equation*}
\int_{x_0}^{\frac{x_0+x_1}{2}}\frac{1}{(b+ax)\sqrt{\varphi(x)}}\mathrm{d}x=+\infty,
\end{equation*}
it follows that any fiber metric for $\widetilde{g}$ is complete.

For $\mu=0$, let
\begin{equation*}
  t=\int_{\frac{x_0+x_1}{2}}^{x(t)}\frac{\mathrm{d}u}{\varphi(u)},\;F(t)=\int_{\frac{x_0+x_1}{2}}^{x(t)}\frac{v\mathrm{d}v}{\varphi(v)},
\end{equation*}
by
\begin{equation*}
 \lim_{x\rightarrow c^{+}} \frac{\varphi(x)}{(x-c)^2}= \frac{a\gamma}{2(2d-1)(a\nu-b\lambda)}>0,\; \lim_{x\rightarrow x_1^{-}}\frac{\varphi(x)}{x_1-x}=1,
\end{equation*}
then
\begin{equation*}
  \lim_{x\rightarrow x_1^{-}}t(x)=+\infty, \;\lim_{x\rightarrow c^{+}}t(x)=-\infty.
\end{equation*}
This shows that the metric $g_F$ is defined on $ M=\left\{(z,w)\in \Omega\times\mathbb{C}^{r}: 0< e^{\lambda\phi(z)}\|w\|^2\right\}$. Since
\begin{equation*}
\int_{x_0}^{\frac{x_0+x_1}{2}}\frac{1}{(b+ax)\sqrt{\varphi(x)}}\mathrm{d}x=+\infty,
\end{equation*}
it follows that any fiber metric for $\widetilde{g}$ is complete.\\

$(\mathrm{III-iv})$ For the case of $a<0$ and $\lambda>0$.\\

By $ax+b>0$ and $\nu+\lambda x>0$, we get $-\frac{\nu}{\lambda}<x<-\frac{b}{a}$, so $\frac{b}{a}-\frac{2\nu}{\lambda}<-\frac{\nu}{\lambda}\leq x_0<x_1\leq -\frac{b}{a}$.

By $\varphi(x_0)=\varphi'(x_0)=0$ and \eqref{eq7.3.4}, we obtain $\gamma=\frac{\mu(\nu+\lambda x_0)}{(b+ax_0)^2}$ and
\begin{equation*}
    \varphi(x)= \frac{(b+ax)^{2d+1}(a\lambda x+2a\nu-b\lambda)}{(\nu+\lambda x)^{d}}\int_{x_0}^x\frac{(\nu+\lambda u)^{d}\left(\mu(\nu+\lambda u)-\gamma(b+au)^2\right)}{(b+au)^{2d+2}(b\lambda-2a\nu-a\lambda u)^2}\mathrm{d}u
\end{equation*}
for $x\in(x_0,x_1)$.

When $x_0<x< c:=-\frac{b}{a}$, from
\begin{equation*}
  b\lambda-2a\nu-a\lambda x>0, \frac{{\nu+\lambda x_0}}{(b+ax_0)^2}\geq 0,\;{\nu+\lambda x}-\frac{{\nu+\lambda x_0}}{(b+ax_0)^2}(b+ax)^2>0,
\end{equation*}
 it follows that
\begin{equation*}
  \mu(\nu+\lambda x)-\gamma(b+ax)^2=\mu\left({\nu+\lambda x}-\frac{{\nu+\lambda x_0}}{(b+ax_0)^2}(b+ax)^2\right)\leq 0.
\end{equation*}
 Hence $\varphi(x)<0$ for  $\mu<0$ and $x\in(x_0,c)$, and $\varphi(x)=0$ for  $\mu=0$ and $x\in(x_0,c)$.
 This is in contradiction with  $\varphi(x)>0$ for $x\in(x_0,x_1)$.\\

\begin{center}
\noindent \textbf{Part $(\mathrm{IV})$ }\\
\end{center}

Since $a\neq 0$, from \eqref{eq7.3.5}, we get $\nu=0$.

By Remark \ref{Re:2.1}, when $\nu=0$ and $r>1$,  there must be $\lambda>0$ and $x=F'(t)>0$.

By \eqref{eq7.3.5}, using $\varphi(x_1)=0$,  we get
\begin{equation}\label{III-1}
    \varphi(x)= \frac{(ax+b)^{2n-1}(ax-b)}{x^{n-1}}\int_{x_1}^x \frac{ u^{n-1}\left(\mu u-r(b+au)^2\right)}{(b+au)^{2n}(au-b)^2}\mathrm{d}u
\end{equation}
for $x\in(x_0,x_1)$. Then
\begin{equation*}
 -1= \varphi'(x_1)=\lim_{x\rightarrow x_1^{-}}\frac{\varphi(x)}{x-x_1}=\frac {\mu  x_1 -r (ax_1+b)^{2}}{(ax_1+b)(a x_1-b)}.
\end{equation*}
So
\begin{equation*}
 \mu=\frac{(ax_1+b)(r(ax_1+b)-(ax_1-b))}{x_1}.
\end{equation*}
Since $\mu\leq 0$, $x_1>0$ and $ax_1+b>0$, we have $r\leq \frac{ax_1-b}{ax_1+b}$.\\

$(\mathrm{IV-i})$ For the case of $a>0$ and $\lambda>0$.\\

By $ax+b>0$ and $\nu+\lambda x>0$, we get $x>\max\{-\frac{b}{a},0\}$, so $x_1>x_0\geq \max\{-\frac{b}{a},0\}$.\\

$(\mathrm{IV-i}-1)$ If $b\leq 0$, namely $-\frac{b}{a}\geq 0\geq\frac{b}{a}$.

By $1<r\leq \frac{ax_1-b}{ax_1+b}$, we have $b<0$ and $x_1\leq -\frac{(r+1)b}{(r-1)a}$.

Let
\begin{equation*}
  q(x)=\mu x-r(b+ax)^2.
\end{equation*}
Note that when $c:=-\frac{b}{a}< x<x_1$, by $\mu\leq 0$, $q(-\frac{b}{a})\leq 0$ and $q(x_1)<0$, we get
\begin{equation*}
  a x-b>0\;\text{and}\;\mu x-r(b+ax)^2\leq -r(b+ax)^2\leq 0,
\end{equation*}
 so $\varphi(x)>0$ for $x\in(c,x_1)$ by \eqref{III-1}.

Using
\begin{equation*}
   \lim_{x\rightarrow c^{+}} \varphi(x)=\frac{-\mu}{2(2n-1)a^2}\geq 0,
\end{equation*}
\begin{equation*}
  \lim_{x\rightarrow c^{+}} \frac{\varphi(x)}{(b+ax)^2}= \frac{-r}{2(2n-3)ab}>0\;(\text{for}\;\mu=0)
\end{equation*}
and fiber metrics for $\widetilde{g}$ are complete, we have $x_0=c$.

For $\mu<0$, let
\begin{equation*}
  t=\int_{c}^{x(t)}\frac{\mathrm{d}u}{\varphi(u)},\;F(t)=\int_{c}^{x(t)}\frac{v\mathrm{d}v}{\varphi(v)},
\end{equation*}
by
\begin{equation*}
 \lim_{x\rightarrow c^{+}} \varphi(x)= \frac{-\mu}{2(2d+1)a^2}>0
\end{equation*}
and $\lim_{x\rightarrow x_1^{-}}\frac{\varphi(x)}{x_1-x}=1$, we have
\begin{equation*}
  \lim_{x\rightarrow x_1^{-}}t(x)=+\infty, \;\lim_{x\rightarrow c^{+}}t(x)=0.
\end{equation*}
This shows that the metric $g_F$ is defined on $ M=\left\{(z,w)\in \Omega\times\mathbb{C}^{r}: 1< e^{\lambda\phi(z)}\|w\|^2\right\}$. Since
\begin{equation*}
\int_{x_0}^{\frac{x_0+x_1}{2}}\frac{1}{(b+ax)\sqrt{\varphi(x)}}\mathrm{d}x=+\infty,
\end{equation*}
it follows that any fiber metric for $\widetilde{g}$ is complete at $x=x_0$.

For $\mu=0$, let
\begin{equation*}
  t=\int_{\frac{x_0+x_1}{2}}^{x(t)}\frac{\mathrm{d}u}{\varphi(u)},\;F(t)=\int_{\frac{x_0+x_1}{2}}^{x(t)}\frac{v\mathrm{d}v}{\varphi(v)},
\end{equation*}
by
\begin{equation*}
 \lim_{x\rightarrow c^{+}} \frac{\varphi(x)}{(b+ax)^2}=\frac{-r}{2(2n-3)ab}>0\;\text{and}\;
   \lim_{x\rightarrow x_1^{-}}\frac{\varphi(x)}{x_1-x}=1,
\end{equation*}
we get
\begin{equation*}
  \lim_{x\rightarrow x_1^{-}}t(x)=+\infty, \;\lim_{x\rightarrow c^{+}}t(x)=-\infty.
\end{equation*}
This shows that the metric $g_F$ is defined on $ M=\left\{(z,w)\in \Omega\times\mathbb{C}^{r}: 0< e^{\lambda\phi(z)}\|w\|^2\right\}$. Since
\begin{equation*}
\int_{x_0}^{\frac{x_0+x_1}{2}}\frac{1}{(b+ax)\sqrt{\varphi(x)}}\mathrm{d}x=+\infty,
\end{equation*}
it follows that any fiber metric for $\widetilde{g}$ is complete at $x=x_0$.\\

$(\mathrm{IV-i}-2)$ If $b>0$,  then
\begin{equation*}
 \mu=\frac{(ax_1+b)(r(ax_1+b)-(ax_1-b))}{x_1}>0.
\end{equation*}
This conflicts with $\mu\leq 0$.\\

$(\mathrm{IV-ii})$ For the case of $a<0$ and $\lambda>0$.\\

By $ax+b>0$ and $\nu+\lambda x>0$, we get $0<x<-\frac{b}{a}$, so $b>0$ and $0\leq x_0<x_1\leq -\frac{b}{a}$.

Using $\varphi(x_0)=\varphi'(x_0)=0$ and \eqref{eq7.3.5},  we obtain $\mu=\frac{r(ax_0+b)^2}{x_0}>0$. This conflicts with $\mu\leq 0$.

\end{proof}

\vskip0.4cm
\noindent
\textbf{Acknowledgement.}\ {\small This work is partially supported by NSFC (No.12071354) and  the Scientific Research Fund of Leshan Normal University (No.DGZZ202024)..

\end{document}